\documentclass{gtmon_a}
\pdfoutput=1

\usepackage{mdwtab}
\usepackage{float}
\usepackage{amscd}
\usepackage{hhline}
\usepackage{threeparttable}
\usepackage{footmisc}

\proceedingstitle{The Zieschang Gedenkschrift}
\conferencestart{5 September 2007}
\conferenceend{8 September 2007}
\conferencename{Conference in honour of Heiner Zieschang}
\conferencelocation{Toulouse, France}

\editor{Michel Boileau}
\givenname{Michel}
\surname{Boileau}

\editor{Martin Scharlemann}
\givenname{Martin}
\surname{Scharlemann}

\editor{Richard Weidmann}
\givenname{Richard}
\surname{Weidmann}

\title{Some quadratic equations in the free group of rank 2}

\author{Daciberg L Gon\c{c}alves\newline
        Elena Kudryavtseva\newline
        Heiner Zieschang}
\givenname{Daciberg}
\surname{Gon\c{c}alves}
\address{Departamento de Matem\'atica\\
IME-USP\\\newline
Caixa Postal 66281\\
Ag\^encia Cidade de S\~ao Paulo\\
05314-970 S\~ao Paulo SP\\
Brasil\vspace{3pt}\\\newline
Department of Mathematics and Mechanics\\
Moscow State University\\\newline
Moscow 119992\\
Russia\vspace{3pt}\\\newline
Fakult\"at f\"ur Mathematik\\
Ruhr-Universit\"at Bochum\\
44780 Bochum\\
Germany}
\email{dlgoncal@ime.usp.br}
\urladdr{}

\givenname{Elena}
\surname{Kudryavtseva}
\email{ekudr@gmx.de}
\urladdr{}

\givenname{Heiner}
\surname{Zieschang}
\email{marlene.schwarz@ruhr-uni-bochum.de}
\urladdr{}

\volumenumber{14}
\issuenumber{}
\publicationyear{2008}
\papernumber{12}
\startpage{219}
\endpage{294}

\doi{}
\MR{}
\Zbl{}

\arxivreference{}


\received{31 July 2006}
\revised{24 April 2008}
\accepted{14 February 2007}
\published{29 April 2008}
\publishedonline{29 April 2008}
\proposed{}
\seconded{}
\corresponding{}
\version{}
\makeatletter
\def\cnewtheorem#1[#2]#3{\newtheorem{#1}{#3}[section]
\expandafter\let\csname c@#1\endcsname\c@Thm}
\makeatother

\let\xysavmatrix\xymatrix
\def\xymatrix{\disablesubscriptcorrection\xysavmatrix}
\AtBeginDocument{\let\bar\wbar\let\tilde\wtilde\let\hat\what}

\newcommand{\iMR}{\mathit{MR}}
\newcommand{\iNR}{\mathit{NR}}
\newcommand{\iNC}{\mathit{NC}}
\newcommand{\ab}{\mathit{ab}}
\newcommand{\mmod}{\operatorname{mod}}
\makeop{Ad}
\makeop{gr}
\makeop{tr}
\makeop{deg}

\newtheorem{Thm}{Theorem}[section]
\cnewtheorem{Cor}[Thm]{Corollary}
\cnewtheorem{Lem}[Thm]{Lemma}
\cnewtheorem{Pro}[Thm]{Proposition}
\cnewtheorem{ThmDef}[Thm]{Theorem and Definition}
\cnewtheorem{DefThm}[Thm]{Definition and Theorem}
\cnewtheorem{ProDef}[Thm]{Proposition and Definition}
\cnewtheorem{CorDef}[Thm]{Corollary and Definition}
\cnewtheorem{LemDef}[Thm]{Lemma and Definition}

\theoremstyle{definition}
\cnewtheorem{Rem}[Thm]{Remark}
\cnewtheorem{Rems}[Thm]{Remarks}
\cnewtheorem{Def}[Thm]{Definition}
\cnewtheorem{Not}[Thm]{Notation}
\cnewtheorem{Con}[Thm]{Condition}
\cnewtheorem{Cons}[Thm]{Conditions}
\cnewtheorem{was}[Thm]{}
\cnewtheorem{Ex}[Thm]{Example}
\cnewtheorem{Dia}[Thm]{{Diagram}}
\cnewtheorem{DefNot}[Thm]{Definition and Notation}
\cnewtheorem{ConNot}[Thm]{Construction and Notation}
\cnewtheorem{Quad}[Thm]{Quadratic equations}
\cnewtheorem{subsec}[Thm]{}
\cnewtheorem{GeomInt}[Thm]{Geometric Interpretation }
\makeautorefname{Thm}{Theorem}
\makeautorefname{Pro}{Proposition}
\makeautorefname{Rem}{Remark}
\makeautorefname{Rems}{Remarks}
\makeautorefname{Lem}{Lemma}
\makeautorefname{Def}{Definition}
\makeautorefname{Cor}{Corollary}
\makeautorefname{Ex}{Example}

\newcommand{\ZZ}{\mathbb {Z}}
\newcommand{\NN}{\mathbb N}

\def\Z{{\cal Z}}
\def\Q{{\cal Q}}
\def\L{{\cal O}}

\def\V{{\cal V}}
\newcommand{\Mone}{M_1} %{\tilde S}
\newcommand{\Mtwo}{M_2} % {S}
\newcommand{\M}{\wwbar\Mtwo} % {S}
\newcommand{\x}{P_2} %{s}
\newcommand{\y}{P_1} %{\tilde s}
\newcommand{\xalpha}{x}
\newcommand{\ybeta}{y}
\renewcommand{\aa}{\alpha}
\newcommand{\bb}{\beta}
\newcommand{\ttheta}{\xi}
\newcommand{\vvarphi}{\eta}

\newcommand{\llangle}{\langle\!\langle}
\newcommand{\rrangle}{\rangle\!\rangle}
\newcommand{\bigllangle}{\bigl\langle\mskip-.5mu\bigl\langle}
\newcommand{\bigrrangle}{\bigr\rangle\mskip-.5mu\bigr\rangle}

\newcommand{\Aut}{\mathrm {Aut}}

\newcommand{\id}{\mathrm {id}}

\newcommand{\rank}{\mathrm {rank}}

\newcommand{\sgn}{w} % {\mathrm {sgn}}

\newcommand{\w}{v} % {\mathrm {sgn}}

\newcommand{\f}{\mathrm {f}}
\newcommand{\nf}{\mathrm {nf}}
\newcommand{\Stab}{\mathrm {Stab}}
\newcommand{\ttilde}{\tilde}
\newcommand{\hhat}{\tilde}

\renewcommand{\int}{{\mathrm{int\,}}}

\newcommand{\Int}[1]{\smash{\stackrel{{\scriptscriptstyle\circ}}{#1}}}

\begin{document}

\begin{webabstract}
For a given quadratic equation with any number of unknowns in any free
group $F$, with right-hand side an arbitrary element of $F$, an
algorithm for solving the problem of the existence of a solution was
given by Culler [Topology 20 (1981) 133--145] using a surface method
and generalizing a result of Wicks [J. London Math. Soc. 37 (1962)
433--444]. Based on different techniques, the problem has been studied
by the authors [Manuscripta Math. 107 (2002) 311--341 and Atti
Sem. Mat. Fis. Univ. Modena 49 (2001) 339--400] for parametric
families of quadratic equations arising from continuous maps between
closed surfaces, with certain conjugation factors as the parameters
running through the group $F$. In particular, for a one-parameter
family of quadratic equations in the free group $F_2$ of rank~2,
corresponding to maps of absolute degree~2 between closed surfaces of
Euler characteristic~0, the problem of the existence of faithful
solutions has been solved in terms of the value of the
self-intersection index $\mu: F_2\to\mathbb{Z}[F_2]$ on the
conjugation parameter.  The present paper investigates the existence
of faithful, or non-faithful, solutions of similar families of
quadratic equations corresponding to maps of absolute degree~0.  The
existence results are proved by constructing solutions.  The
non-existence results are based on studying two equations in
$\mathbb{Z}[\pi]$ and in its quotient $Q$, respectively, which are
derived from the original equation and are easier to work with, where
$\pi$ is the fundamental group of the target surface, and $Q$ is the
quotient of the abelian group $\mathbb{Z}[\pi\setminus\{1\}]$ by the
system of relations $g\sim-g^{-1}$, $g\in\pi\setminus\{1\}$.  Unknown
variables of the first and second derived equations belong to $\pi$,
$\mathbb{Z}[\pi]$, $Q$, while the parameters of these equations are
the projections of the conjugation parameter to $\pi$ and $Q$,
respectively.  In terms of these projections, sufficient conditions
for the existence, or non-existence, of solutions of the quadratic
equations in $F_2$ are obtained.
\end{webabstract}

\begin{asciiabstract}
For a given quadratic equation with any number of unknowns in any free
group F, with right-hand side an arbitrary element of F, an algorithm
for solving the problem of the existence of a solution was given by
Culler [Topology 20 (1981) 133--145] using a surface method and
generalizing a result of Wicks [J. London Math. Soc. 37 (1962)
433--444]. Based on different techniques, the problem has been studied
by the authors [Manuscripta Math. 107 (2002) 311--341 and Atti
Sem. Mat. Fis. Univ. Modena 49 (2001) 339--400] for parametric
families of quadratic equations arising from continuous maps between
closed surfaces, with certain conjugation factors as the parameters
running through the group F. In particular, for a one-parameter family
of quadratic equations in the free group F_2 of rank 2, corresponding
to maps of absolute degree 2 between closed surfaces of Euler
characteristic 0, the problem of the existence of faithful solutions
has been solved in terms of the value of the self-intersection index
mu: F_2 --> Z[F_2] on the conjugation parameter.  The present paper
investigates the existence of faithful, or non-faithful, solutions of
similar families of quadratic equations corresponding to maps of
absolute degree 0.  The existence results are proved by constructing
solutions.  The non-existence results are based on studying two
equations in Z[pi] and in its quotient Q, respectively, which are
derived from the original equation and are easier to work with, where
pi is the fundamental group of the target surface, and Q is the
quotient of the abelian group Z[pi - {1}] by the system of relations g
~ -g^{-1}, g in pi - {1}.  Unknown variables of the first and second
derived equations belong to pi, Z[pi], Q, while the parameters of
these equations are the projections of the conjugation parameter to pi
and Q, respectively.  In terms of these projections, sufficient
conditions for the existence, or non-existence, of solutions of the
quadratic equations in F_2 are obtained.
\end{asciiabstract}

\begin{htmlabstract}
For a given quadratic equation with any number of unknowns in any free
group F, with right-hand side an arbitrary element of F, an
algorithm for solving the problem of the existence of a solution was
given by Culler [Topology 20 (1981) 133&ndash;145] using a surface method
and generalizing a result of Wicks [J. London Math. Soc. 37 (1962)
433&ndash;444]. Based on different techniques, the problem has been studied
by the authors [Manuscripta Math. 107 (2002) 311&ndash;341 and Atti
Sem. Mat. Fis. Univ. Modena 49 (2001) 339&ndash;400] for parametric
families of quadratic equations arising from continuous maps between
closed surfaces, with certain conjugation factors as the parameters
running through the group F. In particular, for a one-parameter
family of quadratic equations in the free group F<sub>2</sub> of rank&nbsp;2,
corresponding to maps of absolute degree&nbsp;2 between closed surfaces of
Euler characteristic&nbsp;0, the problem of the existence of faithful
solutions has been solved in terms of the value of the
self-intersection index &micro;: F<sub>2</sub>&rarr;<b>Z</b>[F<sub>2</sub>] on the
conjugation parameter.  The present paper investigates the existence
of faithful, or non-faithful, solutions of similar families of
quadratic equations corresponding to maps of absolute degree&nbsp;0.  The
existence results are proved by constructing solutions.  The
non-existence results are based on studying two equations in
<b>Z</b>[&pi;] and in its quotient Q, respectively, which are
derived from the original equation and are easier to work with, where
&pi; is the fundamental group of the target surface, and Q is the
quotient of the abelian group <b>Z</b>[&pi;&#x2572;{1}] by the
system of relations g&sim;-g<sup>-1</sup>, g&isin;&pi;&#x2572;{1}.  Unknown
variables of the first and second derived equations belong to &pi;,
<b>Z</b>[&pi;], Q, while the parameters of these equations are
the projections of the conjugation parameter to &pi; and Q,
respectively.  In terms of these projections, sufficient conditions
for the existence, or non-existence, of solutions of the quadratic
equations in F<sub>2</sub> are obtained.
\end{htmlabstract}

\begin{abstract}
For a given quadratic equation with any number of unknowns in any
free group $F$, with right-hand side an arbitrary element of $F$, an
algorithm for solving the problem of the existence of a solution was
given by Culler~\cite{Cu} using a surface method and generalizing a result
of Wicks~\cite{W}. Based on different techniques, the problem has been
studied by the authors~\cite{GKZ2,GKZ1} for parametric families of quadratic
equations arising from continuous maps between closed surfaces, with
certain conjugation factors as the parameters running through the group
$F$. In particular, for a one-parameter family of quadratic equations in
the free group $F_2$ of rank~2, corresponding to maps of absolute degree~2
between closed surfaces of Euler characteristic~0, the problem of the
existence of faithful solutions has been solved in terms of the value
of the self-intersection index $\mu\co F_2\to\mathbb{Z}[F_2]$ on the
conjugation parameter.  The present paper investigates the existence of
faithful, or non-faithful, solutions of similar families of quadratic
equations corresponding to maps of absolute degree~0.  The existence
results are proved by constructing solutions.  The non-existence
results are based on studying two equations in $\mathbb{Z}[\pi]$ and
in its quotient $Q$, respectively, which are derived from the original
equation and are easier to work with, where $\pi$ is the fundamental
group of the target surface, and $Q$ is the quotient of the abelian group
$\mathbb{Z}[\pi\setminus\{1\}]$ by the system of relations $g\sim-g^{-1}$,
$g\in\pi\setminus\{1\}$.  Unknown variables of the first and second
derived equations belong to $\pi$, $\mathbb{Z}[\pi]$, $Q$, while the
parameters of these equations are the projections of the conjugation
parameter to $\pi$ and $Q$, respectively.  In terms of these projections,
sufficient conditions for the existence, or non-existence, of solutions
of the quadratic equations in $F_2$ are obtained.
\end{abstract}

\maketitle

\section{Introduction}\label{sec:Ind0}

Equations in free groups have been extensively studied for many
years: see Culler~\cite{Cu}, Hmelevski{\u\i}~\cite{Hm1a,Hm1b,Hm2}, Lyndon~\cite{L2,L3},
Lyndon and Schupp~\cite[Sections~1.6 and~1.8]{LS}, Makanin~\cite{M},
Razborov~\cite{R}, Steinberg~\cite{St} and Wicks~\cite{W};
see also Gon\c{c}alves and Zieschang~\cite{GZ}, Grigorchuk and
Kurchanov~\cite{GK}, Grigorchuk, Kurchanov and Zieschang~\cite{GriKurZie},
Ol'shanski{\u\i}~\cite{Ol'shanski}, Osborne and Zieschang~\cite{OZ}, and
Zieschang~\cite{Z1964,Z1965}.

For a given quadratic equation $Q(z_1,\dots,z_q)=W$ with any number of unknowns
$z_1,\dots,z_q$ in any free group $F$ with an arbitrary right-hand side
$W\in F$, the problem of the existence of a solution can be studied using Wicks
forms (see Wicks~\cite{W}, Culler~\cite{Cu} and Vdovina~\cite{V1,V2,V3}) which are due to
the geometric approach of Culler~\cite{Cu}. Special case of quadratic
equations has been studied by the authors~\cite{GKZ2,GKZ1,KWZ}
for parametric families of quadratic equations which correpond to maps between
closed surfaces, see~\eqref{eq}.
Also the notions of faithful and non-faithful solutions of such equations were
there introduced, which correspond to the orientation-true maps and the maps
which are not orientation-true, respectively (see Definitions~\ref{def:5.1}(C),
\ref{def:types}(a)). In particular, the problem of the existence of faithful
solutions has been solved in~\cite{GKZ2,GKZ1} in terms of the
self-intersection index $\mu\co F_2\to\ZZ[F_2]$, for
families of quadratic equations with two unknowns in the free group $F_2$
of rank 2, which correspond to
maps of non-vanishing {\it absolute degree} (\fullref{def:abs:degree})
between closed surfaces of Euler characteristic 0.
In this work, we study the existence of faithful, or non-faithful, solutions of
the latter quadratic equations, which correspond to maps of absolute degree 0.

Specifically, let $F_2 = \langle a,b \,|\,\rangle$ be the free group of rank
$2$, $v$ an element of $F_2$, and $\vartheta \in \{1,-1\}$.
We consider the following equations in $F_2$ with the unknowns $z_1,z_2\in F_2$:
\begin{align}
[z_1,z_2] &= \w [a,b]^\vartheta\w^{-1}\cdot [a,b], \label{eqn1} \\
[z_1,z_2] &= \w (a^2b^2)^\vartheta\w^{-1}\cdot a^2b^2, \label{eqn2} \\
z_1^2z_2^2 &= \w [a,b]^\vartheta\w^{-1}\cdot [a,b], \label{eqn3} \\
z_1^2z_2^2 &= \w (a^2b^2)^\vartheta\w^{-1}\cdot a^2b^2. \label{eqn4}
\end{align}
Here $[a,b]=aba^{-1}b^{-1}$, and the conjugation factor $v\in F_2$ is called
the {\it conjugation parameter} of the equation.
The elements $v$, $R_\varepsilon(a,b)\in F_2$, where $R_\varepsilon(a,b)$ is
defined below, can
be regarded as the {\it coefficients} of the equation, see Lyndon and
Schupp~\cite[Section~1.6]{LS}.
The equations~\eqref{eqn1}--\eqref{eqn4} have the form
\begin{equation}
\label{eqF2}
Q_\delta(z_1,z_2) = \w \left( R_\varepsilon(a,b) \right)^\vartheta \w^{-1}
\cdot R_\varepsilon(a,b)
\end{equation}
where
\begin{align*}
Q_\delta(z_1,z_2) &= \begin{cases}
                        [z_1,z_2],   & \delta = +\ ,\\
                        z_1^2 z_2^2, & \delta = -\ ,
                     \end{cases} &
R_\varepsilon(a,b) &= \begin{cases}
                        [a,b],   & \varepsilon = +\ ,\\
                        a^2b^2,  & \varepsilon = -\ .
                     \end{cases}
\end{align*}
As in~\cite{GKZ2}, we denote by $\sgn_\varepsilon\co  F_2 \to \{1,-1\}$
the homomorphism with $\sgn_\varepsilon(a) = \sgn_\varepsilon(b) = \varepsilon$,
called the {\it orientation character}, see \fullref{def:5.1} and
\fullref{rem:or:char}.
Recall~\cite{GKZ1} that a solution $(z_1,z_2)$ of \eqref{eqF2} is called
{\it faithful} if $\sgn_\varepsilon(z_1)=\sgn_\varepsilon(z_2)=\delta$,
and otherwise the solution is called {\it non-faithful}, compare
\fullref{def:5.1}(C).
Of course, every solution of~\eqref{eqn1} is faithful, since $\varepsilon=\delta=+1$;
every solution of~\eqref{eqn3} is non-faithful, since $\varepsilon=+1$ and $\delta=-1$.

We use the following geometric interpretations of
the equations \eqref{eqn1}--\eqref{eqn4}.
These quadratic equations have two unknowns in the free group $F_2$ of rank 2.
Such equations correspond to mappings from a compact surface of Euler
characteristic $-1$ having one boundary component to the bouquet of two
circles (see Culler~\cite{Cu}).
The right-hand sides of \eqref{eqn1}--\eqref{eqn4} have special form which arises from maps
between two closed surfaces of Euler characteristic 0, see~\fullref{sec:Quad15}.
A solution is faithful if and only if the corresponding map is
orientation-true, see~\cite{GKZ2} or \fullref{lem:geom}.

Some faithful solutions of the equation~\eqref{eqn4}, whose corresponding maps are
self-maps of the Klein bottle, were listed in~\cite{GKZ1}, see
also \fullref{rem:gkz2}.
The problem of the existence of faithful solutions of~\eqref{eqF2} with
$w_\varepsilon(v)=\vartheta$
was solved by the authors in~\cite{GKZ2} in terms of the self-intersection
index $\mu(v)\in\ZZ[F_2]$ of the conjugation parameter $v$, see
\fullref{rem:gkz2}. These results are illustrated in \fullref{tbl0} for special
values of $v$.

The goal of the present paper is to investigate the existence of faithful, or
non-faithful, solutions of equation~\eqref{eqn4} in the remaining cases
formulated in detail as follows:
\begin{equation}
\label{eq:cond}
\begin{array}{l}
\mbox{the solution is faithful and $w_\varepsilon(v)=-\vartheta$, or} \\
\mbox{the solution is non-faithful.}
\end{array}
\end{equation}
Such solutions actually correspond to mappings of absolute degree~0
(see \fullref{def:abs:degree} and \fullref{cor:cond}).
Our main results are given in Tables~\ref{tbl1} and~\ref{tbl3}, for faithful solutions, and
in Tables~\ref{tbl2} and~\ref{tbl4}, for non-faithful solutions, of an equation~\eqref{eqN} which
is equivalent to~\eqref{eqF2}.
The results are formulated in terms of the projection $\bar v\in\pi$
of the conjugation parameter $v\in F_2$ to the fundamental group $\pi$ of the
corresponding target surface via
$$p_\pi\co  F_2 \to \pi =
F_2/N, \quad N = \llangle  R_\varepsilon(a,b) \rrangle,$$
as well as in terms of $p_Q(V)\in Q$, which is the image
of $v_0^{-1}v\in N$ under the composition
\begin{equation}
\label{eq:Q}
N \stackrel{q_N}{\longrightarrow} \ZZ[\pi] \stackrel{p_Q}{\longrightarrow}
Q=(\ZZ[\pi\setminus\{1\}])/\langle g+g^{-1}\,|\,g\in\pi\setminus\{1\}\rangle ,
\end{equation}
where $v_0\in F_2$ is a suitable representative of $\bar v\in\pi$ in $F_2$,
see~\eqref{eq:**} and~\eqref{eq:***}, while
$V:=q_N(v_0^{-1}v)\in \ZZ[\pi] \approx N/[N,N]$,
see~\eqref{eq:shortNab} and~\eqref{eq:Nab}.
Here $\llangle u_1,u_2,\ldots\rrangle \subset G$ and $\langle u_1,u_2,\ldots\rangle\subset G$
denote the minimal normal subgroup and the minimal subgroup,
respectively, containing the elements $u_1,u_2,\ldots \in G$ of a group $G$.

To establish the non-existence results given in \fullref{thm:class} and
Tables~\ref{tbl1} and~\ref{tbl2}, we
apply the Nielsen root theory for maps between closed surfaces
(see~\fullref{sec:Quad15}), geometric results of Kneser~\cite{K}
about maps of absolute degree 0 (see also Epstein~\cite{Ep}),
and algebraic
results (see Zieschang~\cite{Z1964,Z1965}, Zieschang, Vogt and
Coldewey~\cite{ZVC} and Ol'shanski{\u\i}~\cite{Ol'shanski}; see also
Kudryavtseva, Weidmann and Zieschang~\cite[Corollary~2.4]{KWZ}) on epimorphisms of
surface groups to free groups (\fullref{lem:geom},
Propositions~\ref{pro:roots} and~\ref{pro:roots:A=0}).
These results allow us to reduce the problem of the existence of (faithful, or
non-faithful, resp.) solutions of the equation~\eqref{eqF2} in $F_2$ satisfying
the condition~\eqref{eq:cond} to the problem of the existence of a (faithful,
or non-faithful, resp.) solution of the following equation in the
subgroup $N=\llangle  \aa\bb\aa^{-\varepsilon} \bb^{-1} \rrangle $ of
$F_2=\langle \aa,\bb\mid\rangle$:
\begin{equation}
\label{eqN}
\xalpha \ybeta \xalpha^{-\delta} \ybeta^{-1}
= \w \bigl( \aa\bb\aa^{-\varepsilon}\bb^{-1} \bigr)^\vartheta \w^{-1}
\cdot \aa\bb\aa^{-\varepsilon}\bb^{-1},
\end{equation}
with the unknowns $\xalpha\in N$, $\ybeta\in F_2$,
see~\eqref{eqn1'}--\eqref{eqn4'}
in \fullref{subsec:appl}, and
Corollaries~\ref{cor:cond} and~\ref{cor:Wecken}(A) (see also
\fullref{thm:class}). Here the free generators $a,b$ of
$F_2=\langle a,b\mid\rangle$ and the unknowns $z_1,z_2$ are replaced by the new
generators and unknowns via
\begin{equation}
\label{eq:change}
\begin{aligned}
 \aa&=a & \bb&=b & \mbox{for } \varepsilon&=1, &
\qquad\qquad \aa&=ab, & \bb&=b^{-1} & \mbox{for } \varepsilon&=-1, \\
 x&=z_1,&    y&=z_2 & \mbox{for } \delta&=1, &
\qquad\qquad x&=z_1z_2,& y&=z_2^{-1} & \mbox{for } \delta&=-1,
\end{aligned}
\end{equation}
thus
$w_\varepsilon(\aa)=1$, $w_\varepsilon(\bb)=\varepsilon$.
A solution $(\xalpha,\ybeta)$ of~\eqref{eqN} in $N$ is called {\it faithful}
if $w_\varepsilon(\ybeta)=\delta$.
We also prove (\fullref{rem:rank1}) that any solution of~\eqref{eqN} in
$N$ satisfies
$$\bar v = \bar \ybeta^k \quad\text{and}\quad \vartheta \delta^k=-1 
 \qquad \mbox{for some} \quad k\in\ZZ.$$
To establish further non-existence results (Tables~\ref{tbl3}
and~\ref{tbl4}), we apply
the algebraic approach developed in this paper (see~\fullref{sec:Quad2}) to the remaining cases of the equation~\eqref{eqN},
namely to those cases where the problem was not solved by the
preceding methods (the so called ``mixed'' cases in Tables~\ref{tbl1}
and~\ref{tbl2}, see
\fullref{rem:tables} and \fullref{def:mixed}). From the
equation~\eqref{eqN} in $N$, two equations are derived using our
algebraic approach, which have solutions corresponding to solutions
of~\eqref{eqN} if the latter exist. The {\it first derived
equation} (\fullref{thm:first}) is
\begin{equation}
\label{first:eq'}
(1-\delta\bar\ybeta) \tilde\xalpha = 1 + \vartheta \bar v,
\end{equation}
in the group ring $\ZZ[\pi] \approx N^{\ab}=N/[N,N]$ of $\pi$
(see~\fullref{pro:Nab}),
with two unknowns $\tilde\xalpha\in \ZZ[\pi]$, $\bar\ybeta\in\pi$,
and the parameter $\bar v \in \pi$, see \fullref{thm:first}. A solution
$(\tilde\xalpha,\bar\ybeta)$ of~\eqref{first:eq'} is called {\it faithful} if
$w_\varepsilon(\bar\ybeta)=\delta$, and it is called {\it non-faithful}
otherwise.
For each solution of the equation \eqref{first:eq'} in the ``mixed'' cases
(see above), we assign an equation
in the quotient $Q$ of $N$, see~\eqref{eq:Q},
namely the equations~\eqref{eqn2_2}, \eqref{eqn3_2}, \eqref{eqn4_2nf} and
\eqref{eqn4_2f},
respectively, in~\fullref{subsec:2der}. We use the fact that
the quotient $Q$
is isomorphic to $[N,N]/[F_2,[N,N]]$, see \fullref{pro:QH2}.
The obtained in this way {\it second derived equation}
(\fullref{thm:second}) has unknowns $X\in Q$, $Y\in \ZZ[\pi]$, a parameter
$p_Q(V) \in Q$ determined by the conjugation parameter $v$, see~\eqref{eq:Q},
and some unknown
integers which are parameters of the solutions of~\eqref{first:eq'}.
We find all values of the parameter $p_Q(V)$ for which the second derived
equation admits a solution (Theorems~\ref{thm:partit34nf} and
\ref{thm:partit24f}), and we use the obvious fact
that the non-existence of a
(faithful or non-faithful) solution of any of the derived equations implies
the non-existence of a (faithful or non-faithful, resp.) solution of the
corresponding quadratic equation \eqref{eqN}.

The paper is organized as follows. In~\fullref{sec:Quad1}, we consider
more general quadratic equations in free groups and briefly formulate
some recent results of the authors about faithful solutions of such equations,
including the equation~\eqref{eqF2} with
$w_\varepsilon(v)=\vartheta$, which correspond to maps of absolute degree 2.
In~\fullref{sec:Quad15}, we recall results of~\cite{GKZ2,GKZ1}, and
Kudryavtseva, Weidmann and Zieschang \cite{KWZ} on
the relationship between the quadratic equations and the Nielsen root theory,
and derive some properties of solutions of~\eqref{eqF2}
satisfying~\eqref{eq:cond} from geometric results of Kneser~\cite{K} about
maps having absolute degree 0 and algebraic results of
Zieschang~\cite{Z1964,Z1965}, Zieschang, Vogt and Coldewey~\cite{ZVC},
and Ol'shanski\u{\i}~\cite{Ol'shanski} on homomorphisms of the surface groups to free groups.
As a result, we obtain Tables~\ref{tbl1} and~\ref{tbl2},
and reduce our problem to study the single equation~\eqref{eqN} in $N$.
In~\fullref{sec:Quad2}, we study some quotients of the subgroup
$N=\llangle \aa\bb\aa^{-\epsilon}\bb^{-1}\rrangle $ of the free group
$F_2=\langle \aa,\bb\mid\rangle$ of rank 2.
In particular, we prove that the quotient $[N,N]/[F_2,[N,N]]$ is isomorphic to
the quotient $Q$ in~\eqref{eq:Q},
see \fullref{pro:QH2}, and we obtain a presentation
for the quotient $N/[F_2,[N,N]]$.
In~\fullref{sec:Quad3}, we describe and derive
two equations, namely the first and the second derived equations, see
above, which are easier to work with than the original equation~\eqref{eqN}.
The second derived equation is constructed when the first derived equation
admits a solution, while the original quadratic equation does not necessarily
admit a solution, see \fullref{ex:Wicks}.
In~\fullref{sec:quad4}, we investigate the existence of a solution of
the second derived equation in the ``mixed'' cases of Tables~\ref{tbl1}
and~\ref{tbl2}.
The results of Sections~\ref{sec:Quad3} and~\ref{sec:quad4} are summarized
in Tables~\ref{tbl3} and~\ref{tbl4} of~\fullref{sec:Quadtab}.

It is not clear whether our results can be obtained using Wicks forms.
The results obtained here are entirely different from
the type of results of Wicks \cite{W} and Vdovina~\cite{V1,V2,V3} using the
Wicks forms, since we are able to consider certain families of equations at
once, in contrast with methods which consider only one equation at the time.

\subsection*{Acknowledgements}
This work was partially done during the visit of the first and third
authors at the Department of Mathematics and Mechanics of Moscow
State University in May--June 2002, and during the visit of the second
author at the Department of Mathematics of S\~ao Paulo University in
October--December 2003. The visits of the first and second authors were
supported by the FAPESP -- Projeto Tem\'atico Topologia Geometrica e
Alg\'ebrica, and the FAPESP, respectively. The visit of the third author
was supported by the Stiftungsinitiative Johann Gottfried Herder.

The first and the second authors are grateful to Heiner Zieschang for
our long collaboration which resulted in fifteen joint papers on different
subjects, like coincidence and intersection theory, low-dimensional
topology, branched coverings, curves and surfaces (where eight papers had
all three of us as coauthors). His friendship, hospitality, and high
level of his research, together with his deep knowledge certainly have
had a great influence to our mathematical growth and achievements.

\section{Recent results on quadratic equations} \label{sec:Quad1}

In the free group $F_r = \langle a_1, \dots, a_r\,|\,\rangle$ we
consider quadratic equations of the form
$$Q(z_1,\dots,z_q) = (\w_1 R \w_1^{-1})^{c_1}
     \cdots (\w_\ell R \w_\ell^{-1})^{c_\ell}
\quad \mbox{with} \quad R = R(a_1,\dots,a_r),$$
where $Q$ and $R$ are some ``quadratic words'' in variables $z_1,\dots,z_q$
and $a_1,\dots,a_r$, respectively, $q \geq 1$, $r \geq 1$ and all
$c_j \neq 0$ are integers, $\w_j \in F_r$. Here $z_1, \dots, z_q$ are
considered as ``unknowns'', while $\ell, c_1, \dots, c_{\ell}$ and
$\w_1, \dots, \w_{\ell}$ are ``given parameters''.
Without loss of generality, one takes $Q$ and
$R$ to be products of squares $z_i^2$ or commutators
$[z_{2i-1}, z_{2i}] = z_{2i-1}z_{2i}z_{2i-1}^{-1}z_{2i}^{-1}$.

The following notation reflects the topological origin of the groups
considered, namely fundamental groups of surfaces with boundary, see
also \fullref{lem:geom}.

\begin{Def}\label{def:5.1} Let $r, q$  be integers $\ge 1$ and
$\varepsilon, \delta \in \{+1, -1\}$; often we will use
$\varepsilon, \delta$ only as signs $+,-$.

(A)\qua Let $F_{r,\varepsilon}$  denote the free group
$F_{r} =  \langle a_1, \dots, a_{r}\ |\ \rangle $ of rank
$r$ together with a homomorphism
$\sgn_\varepsilon \, \colon F_{r} \to \ZZ^* = \{1, -1\}$
called the {\it orientation character} where
   $$
\sgn_+ \co  a_j \mapsto 1, \quad \sgn_- \co  a_j \mapsto -1
\mbox{\quad for } \ 1 \le j \le r.
   $$
We call $F_{r,\varepsilon}$ a \textit{ free group with orientation character}.
Define
\begin{align*}
Q_\delta(z_1, \dots, z_q) &= \begin{cases}
 \prod_{i=1}^{q/2}[z_{2i-1},z_{2i}], & \delta=+, \\
 \prod_{i=1}^{q}   z_i^2,            & \delta=-,
 \end{cases} \\
\text{and}\qquad R_\varepsilon(a_1, \dots, a_r) &= \begin{cases}
\prod_{i=1}^{r/2}[a_{2i-1},a_{2i}], & \varepsilon=+, \\
\prod_{i=1}^{r}   a_i^2,            & \varepsilon=-.
  \end{cases}
\end{align*}
(B)\qua In the group $F_{r,\varepsilon}$ we consider quadratic equations of the form
\begin{equation}
\label{eq}
Q_\delta (z_1, \dots, z_q) =
\prod_{j=1}^\ell \w_j \cdot(R_{\varepsilon}(a_1,\dots,a_r))^{c_j}\cdot
\w_j^{-1}.
\end{equation}
Here $c_j \ne 0$ are integers and $\w_j \in F_r$; of course, when
$\delta = +$ or $\varepsilon = +$ then $q$ or $r$, resp., is even. Now
$z_1, \dots, z_q$ are considered as ``unknowns'', while
$\ell, c_1, \dots, c_{\ell}$ and $\w_1, \dots, \w_{\ell}$ are
``given parameters''.

(C)\qua If $\sgn_\varepsilon (z_j) = \delta$, $1 \le j \le q$, then the solution
$(z_1, \dots, z_q)$ is called \textit{faithful}, and otherwise it is called
{\it non-faithful}. This gives the following restrictions for faithful
solutions: if $\varepsilon= +$ then $\delta$ must be $+$,
if $\varepsilon= -$ and $\delta = +$ then the length of each $z_j$ must be even,
if $\varepsilon=\delta = -$ then all lengths must be odd. Hence, one should
only consider
$(\delta, \varepsilon) \in\{ (+,+), (-,-), (+,-)\}$ in the case of faithful
solutions, and, similarly,
$(\delta, \varepsilon) \in\{ (+,-), (-,+), (-,-)\}$ in the case of non-faithful
solutions.
\end{Def}

\begin{table}[ht!]\centering
\begin{footnotesize}
\begin{tabular}{||r|c|c|c|c|c||}
\hhline{|t:======:t|}
Case& $\delta$ & $\varepsilon$ & $\vartheta$ &
  {conditions on $v$} & faithful solution $(z_1,z_2)$ \\
\hhline{|:=:=:=:=:=:=:|}
(1)\ a\ & + & + & $+$   & $v=a$ & $(a^2,\ b)$
\\
\hhline{||~|~|~|~|-|-||}
b\ &  &  &  & $v=a^{-1}$ & $(ba^{-1}b^{-1}a^{-1}b^{-1},\ ba^2b^{-1})$
\\
\hhline{||~|~|~|~|-|-||}
c\ &  &  &    & $v=a^n, \  |n|\ne 1$ & $\emptyset$
\\
\hhline{|:=:=:=:=:=:=:|}
(2)\ a\           & + & $-$ & $-$   & $v=a^n$, $n$ odd &
       $ (a^nb,b^{-2})$
\\
\hhline{||~|~|~|-|-|-||}
b\              &   &   & $+$   & $v=a^n$, $n$ even   &  $\emptyset$
\\
\hhline{||~|~|~|~|-|-||}
c\              &   &   &    & $v=(ab)^n$ &  $\emptyset$
\\
\hhline{|:=:=:=:=:=:=:|}
(3)\ \phantom{1}\ &$-$& + &     & arbitrary $v$  &  $\emptyset$
\\
\hhline{|:=:=:=:=:=:=:|}
(4)\ a\          &$-$&$-$ & $+$   & $v=ab$ & $(aba, \ b)$
\\
\hhline{||~|~|~|~|-|-||}
b\     & &  &  &  $v=(ab)^{-1}$  & $(b^{-1}ab^3, \ b^{-2}ab^2)$
\\
\hhline{||~|~|~|~|-|-||}
c\     & &  &        & $v=(ab)^n$, $|n|\ne 1$  &  $\emptyset$
\\
\hhline{||~|~|~|~|-|-||}
d\     & &  &        & $v=a^n$, $n$ even & $\emptyset$
\\
\hhline{||~|~|~|~|-|-||}
e\     & &  &        & $v=a^nb$, $n$ odd & $(a^nba^{2-n},\ b)$
\\
\hhline{||~|~|~|-|-|-||}
f\     & &  & $-$ &  $v=a^n$, $n$ odd &  $(a^nb^{-1}a^{-n},\ b)$
\\
\hhline{|b:======:b|}
\end{tabular}
\end{footnotesize}
\smallskip
\caption{Faithful solutions of
$Q_\delta(z_1,z_2)=vR_\varepsilon(a,b)^\vartheta v^{-1}R_\varepsilon(a,b)$
for some values of $v$ with $\sgn_\varepsilon(v)=\vartheta$}
\label{tbl0}
\end{table}

Many other values of $v$ for which the equation has a faithful
solution or not can be obtained from the solutions listed in
\fullref{tbl0}, by applying an automorphism to $v$ of the free group
$F_2$ which sends $B:=R_\varepsilon(a,b)$ to $B^{\pm 1}$, as given
in~\cite[Corollary~7.2]{GKZ1} and~\cite[Corollary~5.22]{GKZ2}.

\begin{Rem} \label{rem:gkz2}
In~\cite{GKZ2}, the authors studied faithful solutions of the quadratic
equation~\eqref{eq} in the case that all numbers $\sgn_\varepsilon(\w_j)c_j$,
$1\leq j \leq \ell$, have the same sign and $A \cdot (r{-}1) = q - 2 + \ell$ where
$A = |c_1| + \ldots + |c_\ell|$ (that is, the value $q$ is ``minimal'',
see Propositions~\ref{pro:roots} and~\ref{pro:roots:A>0}(A)).
We gave an algebraic criterion (see~\cite[Theorem~5.12]{GKZ2},
or~\cite[Theorem~5.17]{KWZ}) for the existence of a faithful
solution of the quadratic equation~\eqref{eq} in terms of the
self-intersection indices
$\mu(\w_1^{-1} \w_j)\in\ZZ[F_{r,\varepsilon}\setminus \{1\}]$,
$2 \leq j \leq \ell$,
and the intersection indices
$\lambda(\w_1^{-1} \w_i, \w_1^{-1} \w_j)\in\ZZ[F_{r,\varepsilon}]$,
$2\leq j<i\leq \ell$.
As an application, we investigated the existence of faithful solutions of
the quadratic equations~\eqref{eqn1}--\eqref{eqn4} for some values of the conjugation
parameter $v$ with
$\vartheta=w_\varepsilon(v)$, see~\cite[Corollary~7.2, Lemma~7.3]{GKZ1},
\cite[Proposition~5.15, Corollary~5.22]{GKZ2},
or~\cite[Proposition~5.21]{KWZ}. The latter condition is equivalent
to the fact that the corresponding maps have absolute degree~2, see
\fullref{def:abs:degree} and \fullref{cor:cond}.  These results are
summarized in \fullref{tbl0} above.
Some faithful solutions from~\cite[Corollary~7.2]{GKZ1}
corresponding to maps of absolute degree $0$
are given in \fullref{tbl1}, case~(4a), and \fullref{tbl3}, case~(4c).
\end{Rem}

\section{Quadratic equations and Nielsen root theory} \label{sec:Quad15}

In Sections~\ref{subsec:abs:degree} and~\ref{subsec:connroot}, we recall
the notion of absolute degree of a map
(\fullref{def:abs:degree}) and some results of~\cite{GKZ1}
and~\cite{GKZ2} (see also~\cite{KWZ}) about
solutions of the quadratic equation~\eqref{eq}.
Then, in \fullref{subsec:appl}, we apply some of these results (\fullref{lem:geom},
Propositions~\ref{pro:roots} and~\ref{pro:roots:A=0}(A),~(C),
and \fullref{cor:Wecken}(A)) to the equations~\eqref{eqn1}--\eqref{eqn4} and summarize the
obtained results in \fullref{thm:class} and Tables~\ref{tbl1}
and~\ref{tbl2}.
Other results of Sections~\ref{subsec:abs:degree} and~\ref{subsec:connroot}
(Propositions~\ref{pro:roots:A>0} and~\ref{pro:roots:A=0}(B), and
\fullref{cor:Wecken}(B),~(C)) will
not be used in our applications and can be skipped in the first reading
(see also \fullref{rem:alternative}).

Every solution of the equation \eqref{eq} provides a continuous map
$\bar f\co  \wwbar\Mone \to \wwbar\Mtwo$ between two closed surfaces (see below)
with exactly $\ell$ roots having the multiplicities $c_1,\dots,c_\ell$,
see~\fullref{subsec:connroot}, \cite{G2}, \cite[5.8, 5.21]{GKZ2},
or~\cite[Lemma~5.5(b)]{GKZ1}.
Here the closed surfaces $\wwbar\Mone$ and $\wwbar\Mtwo$ correspond to the
quadratic words $Q_\delta(z_1,\dots,z_q)$ and $R_\varepsilon(a_1,\dots,a_r)$,
respectively, and are defined as follows, see~\cite{GKZ2}.
If $\varepsilon=1$, we denote $\wwbar\Mtwo:=S_{r/2}$, a closed orientable surface
of genus $r/2$; if $\varepsilon=-1$ then $\wwbar\Mtwo:=N_r$, a closed
non-orientable surface of genus $r$ (that is, the sphere with $r$ crosscuts),
thus $N_{r+1}$ admits $S_r$ as an orientable two-fold covering.
Similarly, we denote $\wwbar\Mone:=S_{q/2}$ if $\delta=1$, and
$\wwbar\Mone:=N_{q}$ if $\delta=-1$.

In particular, the special quadratic equations~\eqref{eqn1}--\eqref{eqn4} that we are going to
study correspond to
maps between closed surfaces of Euler characteristic 0.
We investigate the existence of non-faithful solutions of these equations, and
the existence of faithful solutions of the equations with
$\sgn_\varepsilon(v)=-\vartheta$, see~\eqref{eq:cond}, which actually
correspond to mappings of absolute degree 0, see \fullref{cor:cond}.

Consider two compact surfaces $\Mone$ and $\Mtwo$ having,
respectively, $\ell$ and one boundary components, where $\Mone$
(respectively, $\Mtwo$) is obtained from $\wwbar\Mone$ (respectively,
$\wwbar\Mtwo$) by removing the interiors of $\ell$ disjoint closed
disks $D_1,\dots,D_\ell\subset\wwbar\Mone$ (respectively, the interior
of a closed disk $D\subset \wwbar\Mtwo$). Choose basepoints
$\y\in\partial D_1$, $\x\in\partial D$. The fundamental groups of the
surfaces admit the following {\it canonical presentations}:
\begin{align}
\label{eq:pi1}
&\begin{aligned}
\pi_1(\Mone,\y) &= \langle b_1,\dots,b_q,d_1,\dots,d_\ell \mid
Q_\delta(b_1,\dots,b_q) d_\ell^{-1}\ldots d_1^{-1} \rangle \\
&\approx F_{q+\ell-1}, \\
\pi_1(\Mtwo,\x) &= \langle a_1,\dots,a_r,d \mid
R_\varepsilon(a_1,\dots,a_r) d^{-1} \rangle \\
&\approx F_{r},
\end{aligned} \\
\label{eq:pi1closed}
&\begin{aligned}
\pi_1(\wwbar\Mone,\y) &= \langle b_1,\dots,b_q \mid Q_\delta(b_1,\dots,b_q) \rangle
= F_q/\llangle  Q_\delta (b_1,\dots,b_q)\rrangle , \\
\pi_1(\wwbar\Mtwo,\x) &= \langle a_1,\dots,a_r \mid R_\varepsilon(a_1,\dots,a_r) \rangle
= F_r/\llangle  R_\varepsilon (a_1,\dots,a_r)\rrangle ,
\end{aligned}
\end{align}
which correspond to some ``canonical systems of cuts'' on surfaces,
see~\cite{GKZ2} or~\cite{KWZ}.

A continuous map $f\co M_1\to M_2$ is called {\it proper} if
$\partial M_1=f^{-1}(\partial M_2)$, that is, the boundary of the source is the
preimage of the boundary of the target.

\begin{Lem}[{{\cite[Lemma~5.5]{GKZ1}}}, {{\cite[Lemma~5.9]{KWZ}}}]
\label{lem:geom}
The existence of a solution $(z_1,\dots,z_q)$ of the
equation~\eqref{eq} is equivalent to the
existence of a proper map $f\co  \Mone \to \Mtwo$ such that
$f(\y)=\x$ and the induced homomorphism
$f_\#\co \pi_1(\Mone,\y)\to\pi_1(\Mtwo,\x)$ sends
$$f_\#(d_j) = \w_1^{-1}\w_j \cdot d^{c_j} \cdot \w_j^{-1}\w_1, \quad
1\leq j \leq \ell.$$
Under this correspondence, the elements $z_1,\dots,z_q$ of a
solution are considered as the conjugates (with the conjugating factor $\w_1$)
of the images under $f_\#$ of the elements
$b_1,\dots,b_q$ of the canonical system of generators~\eqref{eq:pi1}, that is
$\w_1^{-1} z_i \w_1 = f_\#(b_i)$, $1\le i\le q$.

The solution $(z_1,\dots,z_q)$ is faithful if and only if the map $f$ is
orientation-true.
\end{Lem}

\subsection{Absolute degree of a continuous map} \label {subsec:abs:degree}

The next two definitions are excerpted from~\cite[Definitions~4.5, 4.6]{KWZ}
and introduce useful tools for studying continuous maps between
manifolds of the same dimension.

\begin{Def}\label{def:types}
(a)\qua In a non-orientable
manifold, the local orientation is either preserved or changed to the inverse
when moved along a closed curve $\gamma$; according to this property
$\gamma$ is called {\it orientation-preserving} or
{\it orientation-reversing}, respectively. Homotopic (even homologic) curves
are the same with respect to orientation.
On a surface, a simple loop $\gamma$ is orientation-preserving
if and only if $\gamma$ is {\it two-sided}; otherwise the curve is
{\it one-sided}. Following P~Olum~\cite{O}, a map $f\colon M_1 \to M_2$
is called {\it orientation-true} if
orientation-preserving loops are sent to orientation-preserving ones and
orientation-reversing loops to orientation-reversing ones.

(b)\qua Following Hopf~\cite{H}, Olum~\cite{O} and Skora~\cite{Sk}, we
distinguish three types of maps.  A {\it map $f$ is of Type I} if it
is orientation-true. If $f$ is not orientation-true and does not map
orientation-reversing loops to null-homotopic ones then $f$ is of {\it
Type II}.  The remaining maps are said to be of {\it Type III}; they
are not orientation-true and map at least one orientation-reversing loop
to a null-homotopic one. Of course, the type of a map can be determined
by studying its effect on the fundamental group.
\end{Def}

\begin{Rem} \label{rem:or:char}
The orientation character
$w_\varepsilon\co F_{r,\varepsilon}\to\ZZ^*=\{1,-1\}$ defined in
\fullref{def:5.1}(A) has the following geometric meaning.
Consider the induced character $\pi_1(\bar M_2,P_2)\to\{1,-1\}$,
see~\eqref{eq:pi1closed}, which will be again denoted by $w_\varepsilon$.
For any closed curve $\gamma$ on $\bar M_2$ based at $P_2$, we have
$w_\varepsilon([\gamma])=1$ if $\gamma$ is orientation-preserving, and
$w_\varepsilon([\gamma])=-1$ if $\gamma$ is orientation-reversing. Here
$[\gamma]\in \pi_1(\bar M_2,P_2)$ denotes the homotopy class of $\gamma$.
\end{Rem}

For mappings between oriented closed manifolds, the notion $\deg(f)$,
the degree of a map $f$, is well known, and there is a variety of ways
to compute it. It is easily generalized to compact oriented manifolds
with boundary if one restricts oneself to proper maps (see
\fullref{lem:geom}). For non-orientable manifolds one can also define
the notion of a degree, as done by H~Hopf~\cite{H}, H~Kneser\cite{K}
and D\,B\,A~Epstein~\cite{Ep}. We recall the definition for surfaces
as given by R~Skora~\cite{Sk}; see also Brown and Schirmer~\cite{BSc}.

\begin{Def}[Absolute degree]
\label{def:abs:degree}
Let $f\colon M_1 \to M_2$ be a proper map between compact surfaces.

(a)\qua The {\it absolute degree of} $f$, denoted by $A(f)$, is defined as follows.
There are three cases according to the type of the mapping $f$.
\begin{enumerate}
\item[(I)] $f$ is of type I, that is, orientation-true. Let $\hat M_i=M_i$ and
$k_i=1$ if $M_i$ is orientable and $\hat M_i$ be the $2$--fold
orientable covering of $M_i$ and $k_i=2$ otherwise. In particular,
$\hat M_i$ is an orientable $k_i$--fold covering of $M_i$. Since $f$ is
orientation-true, there exists a lift $\hat f \colon \hat M_1 \to \hat M_2$.
After fixing orientations on $\hat M_1$ and $\hat M_2$, the
degree of $\hat f$ is defined, and we put
$$A(f) = \tfrac{k_2}{k_1}\bigl|\deg(\hat f)\bigr|.$$

\item[(II)] If $f$ is of type II, we define $A(f) = 0$.

\item[(III)] For $f$  of type III, put $\ell = [\pi_1(M_2) : f_\#(\pi_1(M_1))]$ and
let $\bar M_2 \to M_2$ be the $\ell$--fold (unbranched) covering corresponding
to the subgroup $f_\#(\pi_1(M_1))$. Now  $f$ has a lift
$\bar f \colon M_1 \to \bar M_2$ which induces an epimorphism on the fundamental
groups. Then $A(f)$ is either $\ell$ or $0$ depending on whether the map
\[
\bar f_* \colon\ \ZZ_2 = H_2(M_1, \partial M_1; \ZZ_2)
\to H_2(\bar M_2, \partial \bar M_2; \ZZ_2) =
\begin{cases}
\ZZ_2 & \text{if }\ell < \infty, \\
0 & \text{if }\ell = \infty
\end{cases}
\]
is bijective or not, respectively.
\end{enumerate}
In particular, if $\ell = \infty$, then $A(f) = 0$. Further,  if
$A(f) \neq 0$ then $\ell \mid A(f)$.

(b)\qua The {\it geometric degree of} $f$ is the least non-negative integer $d$
such that, for some  disk $D \subset \Int{M}_2$ and map $g$ properly
homotopic to $f$, the restriction of $g$ to $g^{-1}(D)$
is a $d$--fold covering.
{\it The geometric degree is never smaller than the absolute degree. }
\end{Def}

For branched or unbranched coverings, the definition of the absolute degree
does not give much new and the situation is much simpler.

\begin{Pro}\label{pro:1.c}
{\rm (a)}\qua Every covering, branched or unbranched, is orientation-true.

{\rm (b)}\qua The geometric and the absolute degree of a (branched or
unbranched) covering coincide and are equal to the order of the covering,
that is, the number of leaves.

{\rm (c)}\qua The geometric and the absolute degree of any continuous
map between closed surfaces coincide.
\end{Pro}
\begin{proof}
See Kneser~\cite{K}.
\end{proof}

\subsection{Relation with the Nielsen root theory of maps}
\label {subsec:connroot}

The geometric interpretation of solutions
of~\eqref{eq} by means of
proper maps $f\co \Mone\to\Mtwo$ between the compact surfaces
$\Mone$, $\Mtwo$ with non-empty boundary (see \fullref{lem:geom}) can be
reformulated in terms of maps $\bar f\co \wwbar\Mone \to \wwbar\Mtwo$ between
the closed surfaces $\wwbar\Mone$, $\wwbar\Mtwo$ obtained from $\Mone$, $\Mtwo$
by attaching disks to the boundary components and radially extending the map
$f$ to the disks, see~\cite[5.21]{GKZ2}.
Now, the centers of the disks in $\wwbar\Mone$ form the preimage of the
center $c$ of the disk in $\wwbar\Mtwo$.

The {\it root problem} for a map $\bar f\co \wwbar\Mone\to\wwbar\Mtwo$ and a
point $c\in \wwbar\Mtwo$ is to find a map $\bar g$ homotopic to $\bar f$ which
has the minimal number
$$\iMR[\bar f] := \min_{\bar g\simeq \bar f} |\bar g^{\;-1}(c)|$$
of roots $\bar g^{\;-1}(c)$ among all mappings $\bar g$ homotopic to $\bar f$.
The roots
of $\bar f$ split into {\it Nielsen equivalence classes} similar
to the cases of the coincidence problem and intersection problem,
see~\cite[2.16]{GKZ2} and~\cite[Definition~3.1]{BGZ}.
It follows from Brooks~\cite{B}, Epstein~\cite{Ep} and Kneser~\cite{K}
that the number $\iNR[\bar f]=\iNC[\bar f,c]$ of
{\it essential} Nielsen classes of roots (see Nielsen~\cite{Ni},
or~\cite[Definition~3.6]{BGZ}) equals
\begin{equation}
\label{NR}
\iNR[\bar f] =
\begin{cases}
[\pi_1(\wwbar\Mtwo):\bar f_\#(\pi_1(\wwbar\Mone))] & \mbox{if } A(\bar f)>0, \\
0 & \mbox{if } A(\bar f)=0,
 \end{cases}
\end{equation}
where $A(\bar f)$ denotes the absolute degree of $\bar f$.
The map $\bar f$ has the
{\it Wecken property for the root problem} if the general inequality
\begin{equation}
\label{eq:NR<MR}
\iNR[\bar f] \le \iMR[\bar f]
\end{equation}
is an equality. The root problem for closed surfaces was completely solved
in~\cite{BGKZ,BGZ0,GKZ1}, including the study of the Wecken
property.

Based on the Kneser congruence and the Kneser inequality,
see~\cite{K} or~\cite[Theorem~4.20]{KWZ},
and the geometric meaning of the equation~\eqref{eq}, see
\fullref{lem:geom}, one obtains the following propositions.

\begin{Pro}[{{\cite[Proposition~5.8]{GKZ1} or \cite[Proposition~5.12]{KWZ}}}]
\label{pro:roots}
Suppose that equation~\eqref{eq} admits a solution $(z_1,\dots,z_q)$,
and let $\bar f\co \wwbar\Mone\to\wwbar\Mtwo$ be the corresponding map
between closed surfaces admitting $\ell$ roots of multiplicities
$\sgn_\varepsilon(\w_1)c_1,\ldots,\sgn_\varepsilon(\w_\ell)c_\ell$.
Let $A: = \sgn_\varepsilon(\w_1)c_1+\ldots+\sgn_\varepsilon(\w_\ell)c_\ell$.
If $A(\bar f)>0$ then $A(\bar f) \cdot r \equiv q \mod 2$.
If the solution is faithful then $A(\bar f) = |A|$.  \qed
\end{Pro}

Let, for an element
$u \in \pi_1(\Mtwo) = F_r = \langle a_1,\dots,a_r\,|\,\rangle$, the element
 $$
\bar u \in \pi_1(\wwbar\Mtwo) = F_r / \llangle  R_\varepsilon(a_1,\dots,a_r) \rrangle  =
\langle a_1,\dots,a_r\,|\,R_\varepsilon(a_1,\dots,a_r)\rangle
 $$
denote its image under the natural projection $\pi_1(\Mtwo) \to \pi_1(\wwbar\Mtwo)$.
Denote by $H \subset \pi_1(\wwbar\Mtwo)$ the subgroup generated by the
elements $\bar z_1,\dots,\bar z_q$.
Denote by ${\rm rank}\,H$ the
 the minimal cardinality of a set of generators for $H$ 
\cite[Section~II.2]{L1}.
 If $H$ is a free group, or a free abelian group, this agrees with the
 usual definition of rank.

 \begin{Pro} \label{pro:roots:A>0}
Let, under the hypothesis of \fullref{pro:roots}, $A(\bar f)>0$. Then:

{\rm (A)}\qua
$A(\bar f) \cdot (r-2) \le q - 2$ and
$A(\bar f) \cdot (r-1) \le q - 2 + \iMR[\bar f] \le q - 2 + \ell$.
In particular, if
$\iMR[\bar f] = \ell = |\bar f^{-1}(c)|$, then $\bar f$ is a solution of the
root problem for $\bar f$.

{\rm (B)}\qua If $A(\bar f) \cdot (r-2) = q - 2$ then the
solution is faithful, $\bar f$ is homotopic to an $|A|$--fold covering and
$\iMR[\bar f] = \iNR[\bar f] = A(\bar f) = |A|$, thus $\bar f$ has the Wecken
property for the root problem.

{\rm (C)}\qua
$\iNR[\bar f] = [\pi_1(\M):H] \le \min \{ \ell,A(\bar f) \}$.
Furthermore, consider the subdivision of $\{1,\dots,\ell\}$ into
$\ell_H=[\pi_1(\M):H]$ subsets where $i,j$ belong to the same
subset iff $\bar\w_i\bar\w_j^{\;-1}\in H$
(that is, $\bar\w_i$ and $\bar\w_j$ belong to the same Reidemeister root class);
then each of these subsets is non-empty.
If the solution is faithful then $\sgn_\delta(\ker \bar f_\#) = \{1\}$, and
the sum of $\sgn_\varepsilon(\w_j)c_j$ over
all $j$ belonging to the same subset equals $\smash{\frac{A}{\ell_H}}$.
If the solution is non-faithful then each of these sums is odd,
$\sgn_\delta(\ker \bar f_\#) = \{1,-1\}$, and
$A(\bar f) = \ell_H = \iNR[\bar f] = \iMR[\bar f]$, thus $\bar f$ has the
Wecken property for the root problem.
\end{Pro}

\begin{proof}
(A)\qua Since $A(\bar f)>0$, it follows from the Kneser
inequality~\cite{K} that
$\chi(\wwbar\Mone) \le A(\bar f) \cdot \chi(\wwbar\Mtwo)$.
Since $\chi(\wwbar\Mone)=2-q$, $\chi(\wwbar\Mtwo)=2-r$, this gives the first
inequality. Since the map $\bar f$
has $\ell$ roots, we have $\iMR[\bar f] \le \ell$.
Applying the Kneser inequality to a suitable proper map
$g\co \Mone'\to\Mtwo$
corresponding to a map $\bar g\co \wwbar\Mone\to\wwbar\Mtwo$, which is
homotopic to $\bar f$ and has $\iMR[\bar f]$ roots, one gets the inequality
$$\chi(\wwbar\Mone)- \iMR[\bar f] \le G(g) \cdot (\chi(\wwbar\Mtwo) -
1),$$
where $G(g)$ denotes the geometric degree of $g$, see
\fullref{def:abs:degree}(b) and~\cite[Theorem~4.1]{Sk} (see
also~\cite[Theorem~2.5(A)]{GKZ1}, in the case when $\bar f$ is
orientation-true). On the other hand, $G(g)\ge G(\bar g)=A(\bar g)=A(\bar f)$,
due to \fullref{pro:1.c}(c). This proves~(A).

(B)\qua Since $A(\bar f)>0$, and the Kneser inequality~\cite{K}
 $$
\chi(\wwbar\Mone) \le A(\bar f)\cdot \chi(\wwbar\Mtwo)
 $$
becomes an equality, it follows from~\cite{K} that the map $\bar f$ is
homotopic to an $A(\bar f)$--fold covering (this also follows from the
classification of maps of positive absolute degree,
see~\cite[Theorem~1.1]{Sk}). Therefore $\bar f$ is orientation-true, and
$$\iMR[\bar f]\le[\pi_1(\wwbar\Mtwo):\bar f_\#(\pi_1(\wwbar\Mone))]=A(\bar
f).$$
By
\fullref{lem:geom}, the solution is faithful. Hence, by
\fullref{pro:roots}, $A(\bar f)=|A|$. Together
with~\eqref{NR}, \eqref{eq:NR<MR}, this proves the assertion.

(C)\qua In the case of faithful solutions, this assertion follows
from~\cite[Lemma~5.7]{GKZ1} or~\cite[Lemma~5.18(b)]{KWZ}.
If the solution is non-faithful then the map $\bar f$ is not orientation-true
with $A(\bar f)>0$. Therefore $\bar f$ has Type~III (see
\fullref{def:types}(b))
or, equivalently, $\sgn_\delta(\ker \bar f_\#) = \{1,-1\}$.
It follows from~\cite[Proposition~4.19]{KWZ} that every sum under
consideration is odd. Since the map $\bar f$ is not orientation-true,
it has the Wecken property for the root problem, due to Kneser~\cite{K1928,K}
and~\eqref{NR}, see also~\cite{Ep} or~\cite{GKZ1}. Indeed,
Kneser~\cite{K1928,K} proved that
such $\bar f$ can be deformed to a map having 0 or $\ell_H$ roots
depending on whether $A(\bar f)=0$ or $A(\bar f)>0$, and by~\eqref{NR}
the latter number coincides with $\iNR[\bar f]$.
Therefore $A(\bar f) = \ell_H = \iNR[\bar f] = \iMR[\bar f]$.
\end{proof}

By applying a suitable automorphism of the free group $F_q$, one obtains the
following presentation of the fundamental group of the closed surface
$\wwbar\Mone$,
in addition to~\eqref{eq:pi1closed}, see Lyndon and
Schupp~\cite[Chapter~I, Proposition~7.6]{LS}:
$$
\pi_1(\wwbar\Mone,\y)=\bigl\langle
\xi_1,\dots,\xi_{[\frac{q+1}{2}]}, \eta_1,\dots,\eta_{[\frac{q}{2}]}
\ \big|\
\Q_\delta(\xi_1,\dots,\xi_{[\frac{q+1}{2}]}, \eta_1,\dots,\eta_{[\frac{q}{2}]})
\bigr\rangle,
$$
where
$$
\Q_\delta\bigl(\xi_1,\dots,\xi_{[\frac{q+1}{2}]},
\eta_1,\dots,\eta_{[\frac{q}{2}]}\bigr)
= \begin{cases}
\prod_{i=1}^{\frac q2}[\xi_i,\eta_i], & \delta=1, \\
\Bigl( \prod_{i=1}^{\frac q2-1}[\xi_i,\eta_i] \Bigr) \cdot
[\xi_{\frac q2},\eta_{\frac q2}]_-,
       & \delta=-1, q \text{ even}, \\
\Bigl( \prod_{i=1}^{\frac {q-1}2}[\xi_i,\eta_i] \Bigr) \cdot
\xi_{\frac {q+1}2}^2,
       & \delta=-1, q \text{ odd}. \end{cases}
$$
Here we use the notation
$$[x,y]\,=\,xyx^{-1}y^{-1}, \qquad [x,y]_-\,=\,xyxy^{-1}.$$
By applying the corresponding change of the unknowns, the equation~\eqref{eq}
in $F_r=\langle a_1,\dots,a_r\mid\rangle$ rewrites in the following equivalent
form:
\begin{equation}
\label{eqNeven}
\Q_\delta(x_1,\dots,x_{[\frac{q+1}{2}]}, y_1,\dots,y_{[\frac{q}{2}]})
=
\prod_{j=1}^\ell \w_j \cdot(R_{\varepsilon}(a_1,\dots,a_r))^{c_j}\cdot \w_j^{-1},
\end{equation}
with the new unknowns
$x_1,\dots,x_{[\frac{q+1}{2}]}, y_1,\dots,y_{[\frac{q}{2}]}\in F_r$.
Similarly to \fullref{def:5.1}(C), a solution of the
equation~\eqref{eqNeven} is called {\it faithful} if
$$
w_\varepsilon(x_i)= \left\{ \begin{array}{rl} 1, & 1\le i\le [\frac q2], \\
                     -1, & i=\frac {q+1}2,\ \delta=-1,\ \mbox{$q$ odd},
 \end{array} \right. \quad
w_\varepsilon(y_i)= \left\{ \begin{array}{rl} 1, & 1\le i\le [\frac {q-1}2], \\
                                 \delta, & i=\frac q2,\ \mbox{$q$ even}.
 \end{array} \right.
$$
Otherwise the solution is called
{\it non-faithful}. Actually, a solution of~\eqref{eq} is faithful if and only
if the corresponding solution of~\eqref{eqNeven} is faithful.

Suppose that, for a solution of~\eqref{eqNeven}, all
$x_i\in N=\llangle R_{\varepsilon}(a_1,\dots,a_r)\rrangle $, $1\le i\le [\frac {q+1}2]$. (One
easily shows that, in this case, both sides of the equation belong to $N$.)
If one restricts oneself only to such solutions of~\eqref{eqNeven},
the obtained equation will be refered to as the {\it equation~\eqref{eqNeven}
in the subgroup $N$} of $F_r$. One checks that, for odd $q$, all solutions
of~\eqref{eqNeven} in $N$ are non-faithful, while, for even $q$, a solution is
faithful if and only if $w_\varepsilon(y_i)=1$, $1\le i\le \frac q2-1$, and
$w_{\varepsilon}\bigl(\smash{y_{\frac q2}}\bigr)=\delta$.

For every solution of~\eqref{eqNeven} in $F_r$, consider the corresponding
homomorphism
 $$
h\co F_q=\bigl\langle
\xi_1,\dots,\xi_{[\frac{q+1}{2}]}, \eta_1,\dots,\eta_{[\frac{q}{2}]}
~\big|~ \bigr\rangle
\to F_r=\bigl\langle a_1,\dots,a_r ~\big|~ \bigr\rangle
 $$
sending $\xi_i\mapsto x_i$, $\eta_i\mapsto y_i$. In particular, the subgroup
$H$ is the image of the composition
$$F_q \stackrel{h}{\longrightarrow} F_r
     \stackrel{p_{r,\varepsilon}}{\longrightarrow}
     F_r/\llangle R_\varepsilon(a_1,\dots,a_r)\rrangle ,$$
where $p_{r,\varepsilon}$ is the projection, see~\eqref{eq:pi1closed}. It
follows from \fullref{lem:geom} that $h=j_{v_1}f_\#$, where $j_u$ is the
conjugation by the element $u$ in $F_r$, that is $j_u(v)=uvu^{-1}$, $u,v\in F_r$.

\begin{Pro} \label{pro:roots:A=0}
Suppose that, under the hypothesis of \fullref{pro:roots},
$A(\bar f) = 0$.
Denote by $(x_1,\dots,x_{[\frac{q+1}{2}]}, y_1,\dots,y_{[\frac{q}{2}]})$ the
corresponding solution of the equation~\eqref{eqNeven} in $F_r$, let
$h\co F_q\to F_r$ be the corresponding homomorphism, and $\rho := \rank\,H$.
Then:

{\rm (A)}\qua $\rho \le [\frac{q}{2}]$, moreover there exists an automorphism
$\varphi$ of the free group $F_q$ such that the word
$\Q_\delta(\xi_1,\dots,\xi_{[\frac{q+1}{2}]},
\eta_1,\dots,\eta_{[\frac{q}{2}]})\in F_q$
is preserved under $\varphi$, and 
$$h\varphi(\xi_i)\in N
=\bigllangle R_{\varepsilon}(a_1,\dots,a_r)
\bigrrangle $$
for all $1\le i\le[\frac{q+1}{2}]$. In other words,
for the solution
$\bigl(x'_1,\dots,x'_{[\frac{q+1}{2}]}, y'_1,\dots,y'_{[\frac{q}{2}]}\bigr)$
of the equation~\eqref{eqNeven}, which corresponds to the homomorphism
$h'=h\varphi\co F_q\to F_r$, one has $x'_i\in N$ for all
$1\le i\le[\frac{q+1}{2}]$. The solutions
$\smash{\bigl(x_1,\dots,x_{[\frac{q+1}{2}]},
y_1,\dots,y_{[\frac{q}{2}]}\bigr)}$ and
$\smash{\bigl(x'_1,\dots,x'_{[\frac{q+1}{2}]},
y'_1,\dots,y'_{[\frac{q}{2}]}\bigr)}$
of~\eqref{eqNeven} are both faithful or both non-faithful.

{\rm (B)}\qua
$\iMR[\bar f] = \iNR[\bar f] = 0$;
in particular, $\bar f$ has the Wecken property for the root problem.

{\rm (C)}\qua
Consider the subdivision of $\{1,\dots,\ell\}$ into
subsets where $i,j$ belong to the same subset iff
$\bar\w_i\smash{\bar\w_j^{\;-1}}\in H$
(that is, $\bar\w_i$ and $\bar\w_j$ belong to the same Reidemeister root
class). If $\sgn_\delta(\ker \bar f_\#) = \{+1\}$ then,
for each $i$ with $1\le i\le \ell$, the sum of
$\sgn_\delta({\bar f_\#}^{-1}({\bar\w_i}\bar\w_j^{\;-1})) c_j$ over all
$j$ belonging to the subset containing $i$ vanishes. Otherwise (that is, if
$\sgn_\delta(\ker \bar f_\#) = \{1,-1\}$) each of these sums is even.
 \end{Pro}

\begin{proof}
(A), (B)\qua Since $A(\bar f)=0$, it follows
from Kneser~\cite{K} that the map $\bar f$ is homotopic to a map which
is not surjective (see also Epstein~\cite{Ep}), thus $\iMR[\bar f] = \iNR[\bar f] = 0$.
We also obtain that $\bar f$ is homotopic to a map whose image lies in the
1--skeleton of the target $\wwbar\Mtwo$, and therefore $\bar f_\#$ admits a
composition $\pi_1(\wwbar\Mone,P_1)\to F\to \pi_1(\wwbar\Mtwo,P_2)$ where the
first homomorphism $g\co \pi_1(\wwbar\Mone,P_1)\to F$ is an epimorphism to a
free group $F$.
It follows that $\rank\,F\le [\frac q2]$, see
Zieschang~\cite{Z1964,Z1965}, and Zieschang, Vogt and Coldewey\cite{ZVC} in
the case of orientable $\Mone$, and from Ol'shanski\u{\i}~\cite{Ol'shanski}
in the general case (see also Lyndon and
Schupp~\cite[Proposition~7.13]{LS}, or~\cite[Corollary~2.4]{KWZ}).
Therefore $\rho=\rank\,H\le \rank\,F\le \bigl[\frac q2\bigr]$.

In the case of orientable $\Mone$, it has been proved in~\cite{Z1965}
using the Nielsen method (see also~\cite{ZVC}
or Grigorchuk, Kurchanov and Zieschang~\cite[Proposition~1.2]{GriKurZie})
that there exists a sequence of
``elementary moves'' of the system of generators
$\xi_1,\dots,\xi_{\unfrac{q}{2}}$, $\eta_1,\dots,\eta_{\unfrac{q}{2}}$ of $F_q$,
and a corresponding sequence of ``elementary moves'' of the ``system of cuts''
on $\wwbar\Mone$ (see above), such that the resulting system of generators
$\xi'_1,\dots,\xi'_{\unfrac{q}{2}}, \eta'_1,\dots,\eta'_{\unfrac{q}{2}}$
is also canonical (this means, there exists an automorphism $\varphi$ of $F_q$
such that
$\xi'_i=\varphi(\xi_i)$,
$\eta'_i=\varphi(\eta_i)$, and 
$$\Q_\delta\bigl(\xi'_1,\dots,\xi'_{\unfrac{q}{2}},
  \eta'_1,\dots,\eta'_{\unfrac{q}{2}}\bigr)
  =\Q_\delta\bigl(\xi_1,\dots,\xi_{\unfrac{q}{2}},
  \eta_1,\dots,\eta_{\unfrac{q}{2}}\bigr)$$
in
$F_q$), and $g(\bar \xi'_i)=1$ in $F$ for all $1\le i\le\frac{q}{2}$.
Here $\bar u\in \pi_1(\wwbar\Mone,\y)$ denotes the image of $u\in F_q$ under the
projection $F_q\to F_q /
\bigllangle  \Q_\delta(\xi_1,\dots,\xi_{\frac{q}{2}},
\eta_1,\dots,\eta_{\frac{q}{2}}) \bigrrangle 
=\pi_1(\wwbar\Mone,\y)$. In the general case (that is, when $\Mone$ is not necessarily
oreintable), the existence of an automorphism $\varphi$ of $F_q$ having the
analogous properties was proved by
Ol'shanski\u{\i}~\cite[Theorem~1]{Ol'shanski}.

Since $g(\bar \xi'_i)=1$ in $F$, it follows $\bar f_\#(\bar \xi'_i)=1$ in
$\pi_1(\wwbar\Mtwo,\x)$. Hence $f_\#(\xi'_i)\in N$. This gives
$h\varphi(\xi_i)=h(\xi'_i)=j_{v_1}f_\#(\xi'_i)\in N$.

Let us prove the latter assertion of~(A). Since the automorphism $\varphi$
preserves the quadratic word
$\Q_\delta\bigl(\xi_1,\dots,\smash{\xi_{[{\scriptscriptstyle\frac{q+1}{2}}]}},
\eta_1,\dots,\smash{\eta_{[{\scriptscriptstyle\frac{q}{2}}]}}\bigr)$,
it also ``preserves'' the orientation character
$w_\delta\co F_q\to \{1,-1\}$, that is, $w_\delta=w_\delta\varphi$,
see \fullref{def:5.1}(A), \fullref{rem:or:char},
and Lyndon and Schupp~\cite[Chapter~I, Proposition~7.6]{LS}.
Now observe that the solution
$\bigl(x_1,\dots,\smash{x_{[{\scriptscriptstyle\frac{q+1}{2}}]}},
y_1,\dots,\smash{y_{[{\scriptscriptstyle\frac{q}{2}}]}}\bigr)$ is faithful if
and only if $w_\delta=w_\varepsilon h$. Similarly, %the solution
$\bigl(x'_1,\dots,\smash{x'_{[{\scriptscriptstyle\frac{q+1}{2}}]}},
y'_1,\dots,\smash{y'_{[{\scriptscriptstyle\frac{q}{2}}]}}\bigr)$ is faithful
if and only if $w_\delta=w_\varepsilon h'$.
By the above, the latter equality \vrule width 0pt height 12pt
is equivalent to $w_\delta\varphi=w_\varepsilon h\varphi$, and since $\varphi$
is an automorphism, it is equivalent to $w_\delta=w_\varepsilon h$.

(C)\qua In the case of faithful solutions, the assertion follows
from~\cite[Lemma~5.7]{GKZ1} or~\cite[Lemma~5.18(b)]{KWZ}.
If the solution is non-faithful then the map
$\bar f$ has Type~II if $\sgn_\delta(\ker \bar f_\#) = \{+1\}$, and it
has Type~III if $\sgn_\delta(\ker \bar f_\#) = \{1,-1\}$.
Since $A(\bar f)=0$, it follows from~\cite[Proposition~4.19]{KWZ} that
each sum under consideration vanishes
if $\bar f$ has Type~II, and it is even if $\bar f$ has Type III.
\end{proof}

 \begin{Cor} \label{cor:Wecken}
Under the hypothesis of \fullref{pro:roots},
the following properties hold:

{\rm (A)}\qua Suppose that $r=q=2$. If $A(\bar f) > 0$ then $\rank\,H=2$ and the
solution is faithful.
If $A(\bar f) = 0$ then $\rank\,H \le 1$, moreover
either $\delta=-1$ and $\xalpha\in N=\llangle R_\varepsilon(a_1,a_2)\rrangle $, or $\delta=1$
and $\xalpha'\in N$, for
some solution $(\xalpha',\ybeta')$ which is faithful (resp.\ non-faithful) if
$(\xalpha,\ybeta)$ is faithful (resp.\ non-faithful). If $\xalpha\in N$ then
$A(\bar f)=0$ and the following implications hold:
\begin{align}
\label {eq:faithful}
(x,y)\mbox{ is faithful} &~\iff~
\sgn_\varepsilon(\ybeta)=\delta, \\
\label {eq:rank1'}
\ell=2, c_2\mbox{ odd}  &~\implies~ \exists k\in\ZZ, &
c_1 &=
\begin{cases}
  -c_2\delta^k & \mbox{if $\bar y\ne1$ or $\delta=1$},\\
  \mbox{odd} & \mbox{otherwise},
\end{cases}\\
\notag
 & &
\bar v_1\bar v_2^{\;-1} &= \bar \ybeta^k.
\end{align}
Here $(x,y)=(x_1,y_1)$ is the solution of~\eqref{eqNeven} with $r=q=2$
corresponding to the solution $(z_1,z_2)$ of~\eqref{eqF2} via
the standard transformation of unknowns, see~\eqref{eq:change}.

{\rm (B)}\qua Suppose that either the solution $(z_1,\dots,z_q)$ is non-faithful,
or $A = 0$, or $|A|\cdot(r-2)=q-2$ (in particular, $r=q=2$). Then the map
$\bar f\co \wwbar\Mone \to \wwbar\Mtwo$ has the Wecken property for the root
problem: $\iMR[\bar f] = \iNR[\bar f]=|A(\bar f)|$.

{\rm (C)}\qua Suppose the solution is faithful and $A \ne 0$. Then
$|A| \cdot (r-1) \le q - 2 + \ell$, furthermore:

If $|A| \cdot (r-1) = q - 2 + \ell$, $\ell\ge2$ and
$\sgn_\varepsilon(\w_i)c_i \ne \sgn_\varepsilon(\w_j)c_j$ for some pair of
indices $1\le i,j\le \ell$, then $\iNR[\bar f] < \iMR[\bar f] = \ell$ and, thus,
$\bar f$ does not have the Wecken property for the root problem.
If $\ell' < |A| \cdot (r-1)-q+2$ then
$\iNR[\bar f] \le \ell' < |A| \cdot (r-1)-q+2 \le \iMR[\bar f]$,
where $\ell'$ is the maximal number of disjoint subsets of $\{ 1,\dots,\ell \}$
such that the union of the subsets is $\{ 1,\dots,\ell \}$ and the sum of
$\sgn_\varepsilon(\w_j)c_j$ over all $j$ belonging to the same subset does not
depend on the subset and, hence, equals $A/\ell'$.
 \end{Cor}

\begin{proof}
(A)\qua Let $r=q=2$. Suppose that $A(\bar f)>0$. By
\fullref{pro:roots:A>0}(B), the solution is faithful and $\bar f$
is homotopic to a covering. Therefore $\ell_H=A(\bar f)$ and
$\bar f_\# \co  \pi_1(\wwbar\Mone)\to \pi_1(\wwbar\Mtwo)$ is a monomorphism,
hence $\rank\,H=\rank\,\pi_1(\wwbar\Mone)=2$.

Suppose that $A(\bar f)=0$. By \fullref{pro:roots:A=0}(A),
$\rho=\rank\,H \le [\frac q2]=1$ and there exists an automorphism
$\varphi\in\Aut(F_2)$ such that the relator
$\xi\eta\xi^{-\delta}\eta^{-1}\in F_2$ is preserved by $\varphi$, and
$x':=h\varphi(\xi)\in N$; moreover the corresponding solutions $(x,y)$ and
$(x',y')$ are both faithful or both non-faithful.
Thus $\varphi$ is the desired automorphism when $\delta=1$. In the case
$\delta=-1$, it is well known that the cyclic subgroup $\langle \bar\xi\rangle$
of $\langle \xi,\eta\mid\xi\eta\xi\eta^{-1}\rangle$ (the fundamental group of
the Klein bottle) generated by $\bar\xi$ is characteristic, hence
$\varphi(\xi)=\xi^{\pm1}\xi_1$ for some $\xi_1\in\llangle \xi\eta\xi\eta^{-1}\rrangle $.
Since $h\varphi(\xi)\in N$ and $h(\xi_1)\in N$ (since
$h(\xi\eta\xi\eta^{-1})=xyxy^{-1}$ equals the right-hand side of the equation,
thus it belongs to $N$), it follows that $x=h(\xi)\in N$.

Suppose that $x\in N$. Then $\bar x=1$, thus $H=\langle \bar y\rangle$ and
$\rank\,H\le 1$. It follows from the above that $A(\bar f)=0$.
The property~\eqref{eq:faithful} follows by observing that the solution
$(x,y)=(h(\xi),h(\eta))$ is faithful if and only if
$\sgn_\varepsilon(h(\zeta))=\sgn_\delta(\zeta)$ for any
$\zeta\in F_2=\langle \xi,\eta\mid\rangle$ or, equivalently, for any
$\zeta\in\{\xi,\eta\}$. For $\zeta=\xi$, this equality holds, since
$h(\xi)=x\in N$ and $\sgn_\delta(\xi)=1$. For $\zeta=\eta$, the equality is
equivalent to $\sgn_\varepsilon(y)=\delta$.

Let us prove~\eqref{eq:rank1'}.
Since $\ell=2$ and $c_2$ is odd, it follows from
\fullref{pro:roots:A=0}(C) that $\bar v_1\bar v_2^{\;-1}$ belongs to
the subgroup $H=\langle\bar y\rangle$ and that $c_1+c_2$ is even.
Hence $\bar v_1\bar v_2^{\;-1}=\bar y^k$, for some $k\in\ZZ$, and~\eqref{eq:rank1'} is proved
when $\bar y=1$, $\delta=-1$.
Let us assume that $\bar y\ne1$ or $\delta=1$. If
$\bar y\ne1$ then the kernel of the induced homomorphism
 $$
\bar f_\# \co  \pi_1(\wwbar\Mone)
= \langle \xi,\eta \mid \xi\eta\xi^{-\delta}\eta^{-1}\rangle
\to \pi=\pi_1(\wwbar\Mtwo) = \langle \aa,\bb \mid B \rangle, \quad
\bar \xi\mapsto \bar\xalpha, \ \bar \eta\mapsto \bar\ybeta,
 $$
is generated by $\bar \xi$.
Since $\sgn_\delta(\bar\xi)=1$, we have $\sgn_\delta(\ker \bar f_\#) = \{+1\}$.
If $\delta=1$, the equality $\sgn_\delta(\ker \bar f_\#) = \{+1\}$ is obvious.
Since $\bar v_1\bar v_2^{\;-1}\in \langle \bar\ybeta \rangle=H$, it follows from
\fullref{pro:roots:A=0}(C) that
$w_\delta(\bar f_\#^{-1}(\bar v_1\bar v_2^{\;-1}))c_1 + c_2=0$.
On the other hand, we have $\bar v_1\bar v_2^{\;-1}=\bar\ybeta^k$, thus
 $$
w_\delta(\bar f_\#^{-1}(\bar v_1\bar v_2^{\;-1}))
=w_\delta(\bar f_\#^{-1}(\bar\ybeta^k))=(w_\delta(\bar\eta))^k = \delta^k.
 $$
This proves the equality $c_1\delta^k+c_2=0$, and thereby completes the proof
of~\eqref{eq:rank1'}.

(B)\qua If the solution is non-faithful then $\bar f$ has the Wecken property for
the root problem, due to
Propositions~\ref{pro:roots:A>0}(C) and \ref{pro:roots:A=0}(B).
Suppose that the solution is faithful. Then, by \fullref{pro:roots},
$A(\bar f) = |A|$. If $A=0$ or $|A|\cdot(r-2)=q-2$ then
$\bar f$ has the Wecken property for the root problem, due to
Propositions~\ref{pro:roots:A=0}(B) and~ \ref{pro:roots:A>0}(B).

(C)\qua As above, $A(\bar f)=|A|$. It follows from
\fullref{pro:roots:A>0}(A),~(C) that
$|A| \cdot (r-1) \le q-2+ \iMR[\bar f] \le q-2+\ell$ and
$\iNR[\bar f] = \ell_H \le \ell'$. Hence,
$\iNR[\bar f] \le \ell' < \ell = \iMR[\bar f]$ in the first case, and
$\iNR[\bar f] \le \ell' < |A| \cdot (r-1)-q+2 \le \iMR[\bar f]$ in the second case.
\end{proof}

\begin{Rem} \label{rem:alternative}
Another way of proving the property~\eqref{eq:rank1'} is given below
in \fullref{thm:first}, using the corresponding first derived equation
(which is similar to~\eqref{first:eq}), rather than
\fullref{pro:roots:A=0}(C).  Both geometric and algebraic ways of
proving Proposition 3.8(C) are given in \cite[Proposition 4.19]{KWZ}.
\end{Rem}

\subsection[Applications to the quadratic equations (1)-(4)]{Applications to the quadratic equations \eqref{eqn1}--\eqref{eqn4}}
\label {subsec:appl}

Here we apply the results of \fullref{subsec:connroot} to study the existence of
faithful, or non-faithful, solutions $(z_1,z_2)$ of~\eqref{eqF2} satisfying
the condition~\eqref{eq:cond}.
For some values of $\bar v=p_\pi(v)\in \pi=F_2/\llangle R_\varepsilon(a_1,a_2)\rrangle $, we
give some explicit faithful and non-faithful solutions
in Tables~\ref{tbl1} and~\ref{tbl2}, respectively, in terms of the new variables, as given
in~\eqref{eq:change}.
The non-existence results stated in Tables~\ref{tbl1} and~\ref{tbl2} will be based on
the results of \fullref{subsec:connroot}.

\begin{Cor} \label {cor:cond}
A solution~$(z_1,z_2)$ of~\eqref{eqF2} satisfies the condition~\eqref{eq:cond}
if and only if the absolute degree $A(\bar f)$ of the corresponding map
$\bar f\co \wwbar\Mone\to\wwbar\Mtwo$ (see \fullref{subsec:connroot})
vanishes.
\end{Cor}

\begin{proof}
Suppose that the solution $(z_1,z_2)$ does not satisfy the
condition~\eqref{eq:cond}. Then the solution is faithful and
$\vartheta\ne -w_\varepsilon(v)$, thus $A=w_\varepsilon(v)\vartheta+1\ne 0$. By
\fullref{pro:roots},
this gives $A(\bar f)=|A|>0$.

Suppose that $A(\bar f)>0$. By \fullref{cor:Wecken}(A), the solution
is faithful. By \fullref{pro:roots},
this implies $A(\bar f)=|A|=|w_\varepsilon(v)\vartheta+1|$. Since the latter
expression is positive, we must have $\vartheta \ne -w_\varepsilon(v)$.
Therefore the solution does not satisfy the condition~\eqref{eq:cond}.
\end{proof}

As in~\eqref{eqN} and~\eqref{eq:change}, let us rewrite the
equations~\eqref{eqn1}--\eqref{eqn4} in terms of the new generators $\aa,\bb$ and
the unknowns $\xalpha,\ybeta$, as given in~\eqref{eq:change}.
Thus $R_+(a,b)=[a,b]=[\aa,\bb]$, $R_-(a,b)=a^2b^2=\aa\bb\aa\bb^{-1}$, and we
obtain the equation~\eqref{eqN}, which is written in detail as follows:
\begin{align}
\tag{1$'$} [\xalpha,\ybeta] &=
  \w [\aa,\bb]^\vartheta\w^{-1}\cdot [\aa,\bb],
  \label{eqn1'}\\
\tag{2$'$} [\xalpha,\ybeta] &=
  \w (\aa\bb\aa\bb^{-1})^\vartheta\w^{-1}\cdot \aa\bb\aa\bb^{-1},
  \label{eqn2'}\\
\tag{3$'$} \xalpha\ybeta\xalpha\ybeta^{-1} &=
  \w [\aa,\bb]^\vartheta\w^{-1}\cdot [\aa,\bb],
  \label{eqn3'}\\
\tag{4$'$} \xalpha\ybeta\xalpha\ybeta^{-1} &=
  \w (\aa\bb\aa\bb^{-1})^\vartheta\w^{-1}\cdot \aa\bb\aa\bb^{-1}.
  \label{eqn4'}
\end{align}
In the new generators, the fundamental group
$\pi=\pi_\varepsilon=\pi_1(\bar M_2)$ and the projection of
$F_2=\langle \aa,\bb\mid\rangle$ to it have the form
$$
p_\pi\co F_2\to \pi=F_2/N,
\qquad \mbox{where} \quad N=\llangle  B \rrangle , \ B=\aa\bb\aa^{-\varepsilon}\bb^{-1}.
$$
As above, denote
 $$
\bar u := p_\pi(u), \qquad u\in F_2.
 $$
Every element $\bar u\in\pi$ can be written in a unique way in the following
{\it canonical form}:
\begin{equation}
\label{eq:canon}
\bar u = \bar \aa^r\bar \bb^s, \quad r,s\in\ZZ.
\end{equation}

\begin{Rem} \label{rem:rank1}
Let us apply \fullref{cor:Wecken}(A) to study the existence of
(faithful, or non-faithful) solutions of the
equations~\eqref{eqn1'}--\eqref{eqn4'}
satisfying $A(\bar f)=0$. Suppose that $(x,y)$ is such a solution.
In the case of the equations~\eqref{eqn3'} and \eqref{eqn4'}, we have $\delta=-1$;
hence $x\in N$.
In the case of the equations~\eqref{eqn1'} and \eqref{eqn2'}, we have $\delta=1$; hence
there exists
a solution $(x',y')$ with $x'\in N$, where the solutions $(x,y)$ and
$(x',y')$ are both faithful or both non-faithful.
Thus we can restrict ourselves to study the existence of solutions $(x,y)$
of~\eqref{eqn1'}--\eqref{eqn4'} satisfying $x\in N$. Such solutions have the
properties~\eqref{eq:faithful} and \eqref{eq:rank1'}, where one substitutes
$v_1=v$, $v_2=1$, $c_1=\vartheta$, $c_2=1$. Thus the
property~\eqref{eq:rank1'} has the following form for the
equation~\eqref{eqN}, or~\eqref{eqn1'}--\eqref{eqn4'}:
\begin{equation}
\label{eq:rank1}
\bar v = \bar \ybeta^k, \quad \vartheta \delta^k=-1 ,
 \qquad \mbox{for some} \quad k\in\ZZ.
\end{equation}
\end{Rem}

\begin{Not} \label{not:canon}
If $F_2=\langle t_1,t_2 \mid \rangle$ is a free group on two
generators $t_1,t_2$,
let $|u|_{t_i}$ denote the sum of the exponents of $t_i$ which appear in a
word $u\in F_2$. In the case of $\pi=\pi_-$, denote by $p_K^\aa(u)$ and
$p_K^\bb(u)$, the exponents of $\bar\aa$, $\bar\bb$, respectively, which appear
in the canonical form~\eqref{eq:canon} of the element $\bar u$, thus
$p_K^\aa(u):=r$ and $p_K^\bb(u):=s$, see~\eqref{eq:canon}.
We also denote the projection $p_{\pi}\co F_2\to\pi$ by $p_T$, or $p_K$,
in the cases when $\pi$ is the fundamental group of the 2--torus $T$
($\varepsilon=1$), or the Klein bottle $K$ ($\varepsilon=-1$), respectively.
We will say that an element $\bar u$ of an abelian group is
{\it divisible by $2$} if there exists an element $\bar u_1$ of the group such
that $2\bar u_1=\bar u$.
(Here the additive notation for the group operation is used.)
\end{Not}

The following \fullref{thm:class} summarizes the above results about the
existence of faithful, or non-faithful, solutions satisfying~\eqref{eq:cond}
of the quadratic equation~\eqref{eqN}.
It can be regarded as the ``first classification'' of values of the
conjugation parameter $v$ with respect to the property that the corresponding
quadratic equation admits a (faithful, or non-faithful) solution. These
results are also summarized in Tables~\ref{tbl1} and~\ref{tbl2}, and in the explicit solutions
given in Tables~\ref{tbl3} and~\ref{tbl4}.
The cases which are not completely solved by \fullref{thm:class} are marked as
``mixed'' cases in Tables~\ref{tbl1} and~\ref{tbl2}.

\begin{Thm} \label{thm:class}
Let $v\in F_2=\langle \aa,\bb\mid\rangle$,
$\delta,\varepsilon,\vartheta\in\{1,-1\}$.
For the quadratic equation~\eqref{eqN}, the existence of a faithful
(resp.\ non-faithful) solution satisfying~\eqref{eq:cond} is
equivalent to the existence of a faithful (resp.\ non-faithful) solution
satisfying $\xalpha\in N=\llangle  \aa\bb\aa^{-\varepsilon}\bb^{-1} \rrangle $.
The following results on the existence of such faithful and non-faithful
solutions hold, see Tables~\ref{tbl1} and~\ref{tbl2}, respectively:
\begin{enumerate}
\item The equation \eqref{eqn1'} has a faithful solution for any $v\in F_2$ and
$\vartheta=-1$, see \fullref{tbl1}(1) for a solution, while it has no non-faithful
solution for any $v\in F_2$ and $\vartheta\in\{1,-1\}$. So, in
this case, the problem of the existence of solutions satisfying~\eqref{eq:cond}
is completely solved.

\item The equation \eqref{eqn2'} with $w_-(v)=-\vartheta$ admits a faithful
solution if and only if $\vartheta=-1$, see \fullref{tbl1}(2a) for a
solution. For non-faithful solutions of~\eqref{eqn2'}, we have: 

\begin{enumerate}
\item[(a)] If $\vartheta=1$, there is no solution.
\item[(b)] If $\vartheta=-1$ and $w_-(v)=-1$ then there is a solution, see
\fullref{tbl2}(2b) for a solution.

\item[(c)] If $\vartheta=-1$, $w_-(v)=1$ (thus $p_K^\bb(v)$ is even), and
$p_K^\aa(v)\ne 0$ then there is no solution.

\item[(d)] If $\vartheta=-1$, $w_-(v)=1$ (thus $p_K^\bb(v)$ is even), and
$p_K^\aa(v)=0$ then there is an element $v_1\in p_K^{-1}(p_K(v))$,
for which the equation admits a solution, see \fullref{tbl4}(2c,2d) for a
solution.
\end{enumerate}

\item The equation~\eqref{eqn3'} has no faithful solution, while
the following properties hold for its non-faithful solutions:

\begin{enumerate}
\item[(a)] If $\vartheta=1$, there is a solution, see \fullref{tbl2}(3a) for a
solution.

\item[(b)] If $\vartheta=-1$ and $p_T(v)$ is not divisible by 2, then there
is no solution.

\item[(c)] If $\vartheta=-1$ and $p_T(v)$ is divisible by 2, then there is
an element $v_1\in p_T^{-1}(p_T(v))$, for which the equation admits
a solution, see \fullref{tbl4}(3c) for a solution.
\end{enumerate}

\item The following properties hold for faithful solutions of the
equation~\eqref{eqn4'} with $w_-(v)=-\vartheta$:
\begin{enumerate}

\item[(a)] If $w_-(v)=-1$ (thus $p_K^\bb(v)$ is odd) then there is a
solution, see \fullref{tbl1}(4a) for a solution.

\item[(b)] If $w_-(v)=1$ (thus $p_K^\bb(v)$ is even) and $p_{K}^\aa(v)\ne
0$, then there is no solution.

\item[(c)] If $w_-(v)=1$ (thus $p_K^\bb(v)$ is even) and $p_{K}^\aa(v)=0$,
then there is an element $v_1\in p_K^{-1}(p_K(v))$, for which the
equation admits a solution, see \fullref{tbl3}(4d) for a solution.

For non-faithful solutions of~\eqref{eqn4'}, the following properties hold:

\item[(d)] If $w_-(v)=-1$ (thus $p_K^\bb(v)$ is odd) then there is no
solution.

\item[(e)] If $w_-(v)=1$ (thus $p_K^\bb(v)$ is even) and $\vartheta=1$, then
there is a solution, see \fullref{tbl2}(4b) for a solution.

\item[(f)] If $w_-(v)=1$ (thus $p_K^\bb(v)$ is even), $\vartheta=-1$, and
moreover $p_K^\bb(v)$ is not divisible by 4 or $p_K^\aa(v)$ is odd,
then there is no solution.

\item[(g)] If $w_-(v)=1$ (thus $p_K^\bb(v)$ is even), $\vartheta=-1$, $p_K^\bb(v)$ is divisible by 4, and
$p_K^\aa(v)$ is even, then there is an element $v_1\in
p_K^{-1}(p_K(v))$, for which the equation admits a solution, see
\fullref{tbl4}(4d) for a solution.
\end{enumerate}

\item In each of the ``mixed'' cases 2d, 3c, 4c, 4g above, any solution with $x\in N$  satisfies \eqref{eq:faithful} and \eqref{eq:rank1}, which imply
  $\bar v\in\langle\bar y^2\rangle$.
\end{enumerate}
 \end{Thm}

\begin{proof}
By \fullref{cor:cond} and \fullref{rem:rank1},
the existence of a solution satisfying~\eqref{eq:cond} is equivalent to the
existence of a solution satisfying $\xalpha\in N$, where the solutions are both
faithful or both non-faithful.
This proves the first desired assertion.

By direct calculations in the free group $F_2$, or in the abelianised $F_2$,
one readily obtains the following cases of Tables~\ref{tbl1}
and~\ref{tbl2}:
 $$
\begin{array}{l}
\mbox{\fullref{tbl1}, cases (1), (2a), (2b), (3), (4a), and} \\
\mbox{\fullref{tbl2}, cases (1), (2a), (2b), (3a), (4b).}    \end{array}
 $$
The corresponding arguments for each of these cases are given in the footnotes
to these cases in Tables~\ref{tbl1} and~\ref{tbl2}.

The following cases of Tables~\ref{tbl1} and~\ref{tbl2} are marked as ``mixed'' cases:
\begin{equation}
\label{eq:mixed}
\mbox{\fullref{tbl1}, case (4c) \quad and \quad \fullref{tbl2}, cases (2d), (3c), (4e).}
\end{equation}
In each of these cases, an explicit value of the conjugation parameter
$v_1\in p_\pi^{-1}(v)$, together with an explicit solution of the corresponding
quadratic equation, are given in \fullref{tbl3}, case~(4d), and
\fullref{tbl4},
cases~(2c,2d), (3c), (4d), respectively. In the first of these cases, the solution
was given in~\cite[Corollary~7.2]{GKZ1}.
Other three cases are justified by direct calculations in $F_2$ (actually in
$N$). In the latter case, one also uses the following relation which is a simple
consequence of the relation $\bar\aa \bar\bb \bar\aa \bar\bb^{-1}=1$,
in the 
fundamental group $\pi=\pi_-$ of the Klein bottle:
\begin{equation}
\label{eq:evenK}
(\bar \aa^r \bar \bb^{2s})^2 = \bar \aa^{2r}\bar \bb^{4s}, \quad r,s\in\ZZ.
\end{equation}
Let us prove (5) and the non-existence results stated in the remaining cases, namely:
\begin{center}
\fullref{tbl1}, case (4b) \quad and \quad
\fullref{tbl2}, cases (2c), (3b), (4a), (4c), (4d).
\end{center}
By \fullref{cor:cond} and \fullref{rem:rank1}, we may assume that
$x\in N=\llangle B\rrangle $, and \eqref{eq:faithful}, \eqref{eq:rank1} hold.
In particular, $\bar x=1$, $\bar v=\bar y^k$, for some $k\in\ZZ$.

Consider the cases~(4b,c) of \fullref{tbl1} and the cases (2c,d) of \fullref{tbl2}.
Since the solution is faithful (resp.\ non-faithful), we have
$w_-(\bar y)=-1$, see~\eqref{eq:faithful}.
We conclude that $\bar y^2$ is a power of $\bar\bb$, by applying the canonical
form~\eqref{eq:canon} of elements in $\pi=\pi_-$:
\begin{equation}
\label{eq:oddK}
(\bar \aa^r \bar \bb^{2s+1})^2 = \bar \bb^{4s+2}, \quad r,s\in\ZZ.
\end{equation}
Since $\bar v=\bar y^k$, $w_-(\bar y)=-1$, $w_-(\bar v)=1$, the
integer $k$ must be even. Therefore, $\bar v$ is a power of $\bar y^2$,
thus also a power of $\bar\bb$.

In the cases (4a), (4c) of \fullref{tbl2}, we have $w_-(\bar v)=-1$ and
$w_-(\bar \ybeta)=1$, since the solution is non-faithful,
see~\eqref{eq:faithful}. This contradicts to $\bar v=\bar \ybeta^k$.

In the cases (3b,c), (4d,e) of \fullref{tbl2}, we have
$\delta=\vartheta=-1$. It follows from the second part
of~\eqref{eq:rank1} that $(-1)^k=1$, thus $k$ is even. This proves (5)
and finishes the proof in the case~(3b) of \fullref{tbl2}.  In the
case~(4d) of \fullref{tbl2}, we have $w_-(\bar\ybeta)=1$, since the
solution is non-faithful, see~\eqref{eq:faithful}.  Together with the
relation~\eqref{eq:evenK}, this shows that the canonical form of $\bar
v=\bar y^k$ is $\bar \aa^{2m}\bar \bb^{4n}$, for some $m,n\in\ZZ$.
\end{proof}

\fullref{thm:class} gives many cases for the values of $\bar v\in\pi$
such that all elements $v_1\in p_{\pi}^{-1}(\bar v)$ simultaneously have
(or simultaneously do not have, respectively) the following property:
the corresponding equation~\eqref{eqN} has a solution
satisfying~\eqref{eq:cond}, where the cases of faithful and
non-faithful solutions are considered separately, see Tables~\ref{tbl1}
and~\ref{tbl2},
respectively. The remaining cases listed in~\eqref{eq:mixed}
are marked in Tables~\ref{tbl1} and~\ref{tbl2} as ``mixed'' cases because of the following.

\begin{table}[ht!]
\begin{minipage}{\linewidth}
\renewcommand{\thempfootnote}{\rm(\roman{mpfootnote})}
\begin{center}
\begin{footnotesize}
\centering
\begin{tabular}{||r|c|c|c|c|c|c||}
\hhline{|t:=======:t|}
Case& \quad$\delta$\quad & \quad$\varepsilon$\quad & \quad$\vartheta$\quad &
  \multicolumn{2}{c|}{conditions on $v$} &
  faithful solution $(\xalpha,\ybeta)$ \\
\hhline{|~|~|~|~|-|-|}
& & & & \quad$\sgn_\varepsilon(v)$\quad &  & \\
\hhline{|:=:=:=:=:=:=:=:|}
(1)\ \phantom{1}\ & + & + & $-$ &$+$\hyperlink{mpfn11}{$^{\rm(i)}$}&  &
                     $(vB^{-1}v^{-1},\ v^{-1})$\hyperlink{mpfn14}{$^{\rm(iv)}$} \\
\hhline{|:=:=:=:=:=:=:=:|}
(2)\ a\           & + &$-$& $-$ & $+$     &  &
        $(vB^{-1}v^{-1},\ v^{-1})$\hyperlink{mpfn14}{$^{\rm(iv)}$} \\
\hhline{||~|~|~|-|-|-|-||}
b\              &   &   & $+$ & $-$     &  &    $\emptyset$\hyperlink{mpfn12}{$^{\rm(ii)}$} \\
\hhline{|:=:=:=:=:=:=:=:|}
(3)\ \phantom{a}\ &$-$& + &     &$+$\hyperlink{mpfn11}{$^{\rm(i)}$}&  &
$\emptyset$\hyperlink{mpfn13}{$^{\rm(iii)}$} \\
\hhline{|:=:=:=:=:=:=:=:|}
(4)\ a\ &$-$&$-$ & $+$ & $-$     &  &
           $(B,\ B^{-1}v)$\hyperlink{mpfn14}{$^{\rm(iv)}$} \\
\hhline{||~|~|~|-|-|-|-||}
b\     & &  & $-$ & $+$      & \multicolumn{1}{c|}{$p^\aa_K(v)\ne 0$} &
$\emptyset$\hyperlink{mpfn15}{$^{\rm(v)}$} \\
\hhline{||~|~|~|~|~|-|-||}
c\     & &  &     &          & $p^\aa_K(v)=0$ & ``mixed'' case, see
\fullref{tbl3} %~(4c)
\\
\hhline{|b:=======:b|}
\end{tabular}
\end{footnotesize}
\end{center}

\caption{Faithful solutions of
$\xalpha\ybeta\xalpha^{-\delta}\ybeta^{-1}=vB^\vartheta v^{-1}B$
with $B=\aa\bb\aa^{-\varepsilon}\bb^{-1}$, $w_\varepsilon(v) =
-\vartheta$}
\label{tbl1}

\end{minipage}
\end{table}

\begin{table}[ht!]
\centerline{\hspace*{1.3cm}
\begin{minipage}{\linewidth}
\renewcommand{\thempfootnote}{\rm(\roman{mpfootnote})}
\begin{center}
\begin{footnotesize}
\centering
\begin{tabular}{||r|c|c|c|c|c|c||}
\hhline{|t:=======:t|}
Case & $\delta$ & $\varepsilon$ & $\vartheta$
&\multicolumn{2}{c|}{conditions on $v$}
& non-faithful solution $(\xalpha,``\ybeta)$
\\
\hhline{||~|~|~|~|-|-|~||}
& & & & $\sgn_\varepsilon(v)$ & &
\\
\hhline{|:=:=:=:=:=:=:=:|}
(1) \phantom{a}
      &  +  &  +  &     &$+$\hyperlink{mpfn11}{$^{\rm(i)}$}& &
$\emptyset$\hyperlink{mpfn11}{$^{\rm(i)}$}
\\
\hhline{|:=:=:=:=:=:=:=:|}
(2) a &  $+$ & $-$ & $+$  & $\pm$   & & $\emptyset$\hyperlink{mpfn12}{$^{\rm(ii)}$}
\\
\hhline{||~|~|~|-|-|-|-||}
   b &     &     & $-$ & $-$     & &
             $(vB^{-1}v^{-1}, v^{-1})$\hyperlink{mpfn14}{$^{\rm(iv)}$}
\\
\hhline{||~|~|~|~|-|-|-||}
   c &     &     &     & +      & $p^\aa_K(v)\ne 0$ &
$\emptyset$\hyperlink{mpfn15}{$^{\rm(v)}$}
\\
\hhline{||~|~|~|~|~|-|-||}
   d &     &     &     &        & $p^\aa_K(v) = 0$ &
                                          mixed case, see
\fullref{tbl4}(2) \\
\hhline{|:=:=:=:=:=:=:=:|}
(3) a & $-$ &  +  & $+$ &$+$\hyperlink{mpfn11}{$^{\rm(i)}$}& &
                       $([\aa,\bb], [\bb,\aa]v)$\hyperlink{mpfn14}{$^{\rm(iv)}$}
\\
\hhline{||~|~|~|-|-|-|-||}
   b &     &     & $-$ &$+$\hyperlink{mpfn11}{$^{\rm(i)}$}& $2{\nmid} p_T(v)$ &
$\emptyset$\hyperlink{mpfn15}{$^{\rm(v)}$}
\\
\hhline{||~|~|~|~|~|-|-||}
   c &     &     &     &         & $2{\mid} p_T(v)$ &
                                           mixed case, see
\fullref{tbl4}(3)
\\
\hhline{|:=:=:=:=:=:=:=:|}
(4) a & $-$ & $-$ &  +   & $-$     & &
$\emptyset$\hyperlink{mpfn15}{$^{\rm(v)}$}
\\
\hhline{||~|~|~|~|-|-|-||}
   b &     &     &      &  +      & &
                 $(B, B^{-1}v)$\hyperlink{mpfn14}{$^{\rm(iv)}$}
\\
\hhline{||~|~|~|-|-|-|-||}
   c &      &     &  $-$   & $-$     & & $\emptyset$\hyperlink{mpfn15}{$^{\rm(v)}$}
\\
\hhline{||~|~|~|~|-|-|-||}
   d &      &      &       &  +      & $4{\nmid} p_K^\bb(v)$ or
                                     $2{\nmid} p_K^\aa(v)$ &
$\emptyset$\hyperlink{mpfn15}{$^{\rm(v)}$}
\\
\hhline{||~|~|~|~|~|-|-||}
   e &     &     &     &         & $4{\mid} p_K^\bb(v)$ and $2{\mid} p_K^\aa(v)$ &
                                       mixed case, see \fullref{tbl4}(4)
 \\
\hhline{|b:=======:b|}
\end{tabular}
\end{footnotesize}
\end{center}
\caption{Non-faithful solutions of
$\xalpha\ybeta\xalpha^{-\delta}\ybeta^{-1}=vB^\vartheta v^{-1}B$ with
$B=\aa\bb\aa^{-\varepsilon}\bb^{-1}$}
\label{tbl2}
\footnotetext[1]{\hypertarget{mpfn11}{Automatically} for $\varepsilon=+1$.}
\footnotetext[2]{\hypertarget{mpfn12}{The} right-hand side of the equation~\eqref{eqF2}
  is not in $[F_2,F_2]$.}
\footnotetext[3]{\hypertarget{mpfn13}{Automatically} for $\varepsilon=+1$, $\delta=-1$.}
\footnotetext[4]{\hypertarget{mpfn14}{Direct} calculation.}
\footnotetext[5]{\hypertarget{mpfn15}{Using}~\eqref{eq:faithful}, and
  either~\eqref{eq:rank1} or the first derived equation~\eqref{first:eq},
  see~\fullref{rem:alternative} and \fullref{thm:first}.}
\end{minipage}}
\end{table}

\begin {Def} \label{def:mixed}
A family of quadratic equations~\eqref{eqN},
with the conjugation parameter $v$ running through the set
$p_{\pi}^{-1}(\bar v_0)$, is called
{\it mixed (with respect to the property of the
existence of a faithful, respectively non-faithful, solution)} if there exist two
parameter values $v_1,v_2\in p_{\pi}^{-1}(\bar v_0)$
such that the equation with $v=v_1$ has a faithful (respectively, non-faithful)
solution,
while the equation with $v=v_2$ has no faithful (respectively, non-faithful)
solution.
\end{Def}

\subsection{Comments to Tables~\ref{tbl1} and~\ref{tbl2}}
\label{subsec:tables12}

As above, we rewrite the equation~\eqref{eqF2} in the equivalent
form~\eqref{eqN}, in terms of the new generators $\aa,\bb$ of $F_2$, and the
new unknowns $x,y$, using the transformation of variables~\eqref{eq:change}.
Thus the equations \eqref{eqn1}--\eqref{eqn4} are transformed to the
equations~\eqref{eqn1'}--\eqref{eqn4'},
see \fullref{subsec:appl}.
In Tables~\ref{tbl1} and~\ref{tbl2} above, we summarize the results of \fullref{thm:class}
on the existence of faithful and non-faithful solutions of the latter equations,
respectively.

\begin{Rem} \label{rem:tables}
The primary objective, for the remainder of this paper, is the
study of the
four cases~\eqref{eq:mixed} of Tables~\ref{tbl1} and~\ref{tbl2} (the ``mixed'' cases), which are not
completely solved by \fullref{thm:class}. These cases are described in
detail in~\fullref{sec:Quadtab}.
In Tables~\ref{tbl3} and~\ref{tbl4} below, we will show that the cases~\eqref{eq:mixed}
are indeed the ``mixed'' cases with respect to the property of the existence of
a solution, see \fullref{def:mixed}.
A complete description of all words $v_1,v_2\in p_{\pi}^{-1}(\bar v)$
as in \fullref{def:mixed}, in a mixed case,
does not seem to be an easy task.
\end{Rem}

\section{Some quotients of the normal closure of an element of a free
group} \label{sec:Quad2}

In this section, we denote by $F$ a free group of finite rank $\ge 2$,
$B\in F$, $N=\llangle B\rrangle $, and $\pi=F/N$. Thus, $N$ is the normal closure of the
element $B$, that is, the minimal normal subgroup of $F$ containing $B$, while
$\pi$ is a one-relator group. We will assume that the word $B$ is not a proper
power of any element of $F$ (although, in some of the assertions, the
hypothesis above
can be made weaker). In particular, all assertions of this section are valid if
$B$ is a strictly quadratic word in a set of free generators of $F$,
see Lyndon and Schupp~\cite[Section~I.7]{LS}.
For $F=F_r=\langle a_1,\dots,a_r\mid\rangle$, such words are automorphic images
of the words $R_+(a_1,\dots,a_r)$,
$R_-(a_1,\dots,a_r)$, $r\ge 2$, see \fullref{def:5.1}(A)
and~\cite[Chapter~I, Proposition~7.6]{LS}.

Consider the following normal subgroups of the group $N$:
 $$
N\supset[F,N]\supset N_1=[N,N],\ [F,[F,N]]\supset[F,[N,N]]\supset[N,[N,N]].
 $$
We will construct presentations of the
 %corresponding
quotients $N/N_1$, $N_1/[N,N_1]$ and $N/[N,N_1]$ (see \fullref{subsec:Nab}),
$N_1/[F,N_1]$ and $N/[F,N_1]$ (see \fullref{subsec:NF}),
and $N/[F,N]$ and $[F,N]/[F,[F,N]]$ (see \fullref{subsec:Q}). It will
follow that the first, second, fourth, and sixth quotients are free
abelian groups, the third and fifth quotients are the middle groups
of extensions of free abelian groups, while the seventh one is
isomorphic to $\pi^{\ab}=\pi/[\pi,\pi]$, the abelianised group $\pi$.
If $N$ is the commutator subgroup $[F,F]$ then the latter quotient
comes from the lower central series of the free group $F$, see also
\fullref{rem:Hilton}.

As in~\fullref{sec:Quad15}, we will denote by $\bar u\in F$ the class of an
element $u\in F$ in $\pi$.

\subsection{The groups $N/N_1$, $N_1/[N,N_1]$, and $N/[N,N_1]$}
\label{subsec:Nab}

Let us consider the short exact sequence
$$1\to N_1 \to N \to N/N_1 \to 1.$$
Here, as above,
$N=\llangle B\rrangle $, $B\in F$, and $B$ is not a proper power of any element of $F$.
By the Nielsen--Schreier subgroup theorem~\cite{Schreier}, $N$ is a free group,
since it is a subgroup of a free group. Furthermore, it follows
from Lyndon~\cite[Section~7]{L1} that $N^{\ab}=N/N_1$, the abelianised
group $N$, is isomorphic to the free abelian group which has a basis in a
bijective correspondence with $\pi=F/N$, see~\cite[Introduction]{L1}.
These results are formulated in more detail as follows.

\begin{Pro}[Lyndon \cite{L1}] \label{pro:Nab}
Suppose that the relator $B\in F$ is not a proper power of any element of a
free group $F$. Consider the short exact sequence
\begin{equation}
\label{eq:shortFN}
1 \to N \to F \stackrel{p_\pi}{\longrightarrow} \pi\to 1,
\end{equation}
where $N=\llangle B\rrangle $, the minimal normal subgroup which contains the relator $B$,
while $\pi=F/N$, a group with a single defining relation.
Then the group
$N$ is free and admits a free basis (for example, a Schreier basis) of the form $B_u$,
$u\in W$, where
$W=s(\pi)\subset F$, and $s\co \pi\to F$ is a map with $p_\pi s=\id_\pi$.
Furthermore, $N^{\ab}$, the abelianised group $N$, is isomorphic to the
abelian group $(\ZZ[\pi],+)$ of the group ring $\ZZ[\pi]$.
Moreover, there exists a short exact sequence
\begin{equation}
\label{eq:shortNab}
1\to [N,N]\stackrel{i_N}{\longrightarrow} N \stackrel{q_N}{\longrightarrow}
(\ZZ[\pi],+) \to 0,
\end{equation}
where $i_N$ is the canonical inclusion, while $q_N$ is an epimorphism sending
\begin{equation}\label {eq:Nab}
q_N\co  N \to (\ZZ[\pi],+), \qquad
\prod_{i=1}^r B_{u_i}^{n_i}
\longmapsto
\sum_{i=1}^r n_i \bar u_i \in \ZZ[\pi],
\end{equation}
for any $u_i\in F$, $n_i\in\ZZ$, where $B_u=uBu^{-1}$, $\bar
u=p_\pi(u)$, $u\in F$.
\qed
\end{Pro}

A similar assertion, for any element $B\in \pi$, was proved by Cohen and
Lyndon~\cite{CL}.
In the case when the relator $B$ is a strictly quadratic word in the free
generators $a_1,\dots,a_r$ of the group $F=\langle a_1,\dots,a_r\mid\rangle$,
$r\ge 2$, for example $B=R_\varepsilon(a_1,\dots,a_r)$, see
\fullref{def:5.1}(A), an alternative proof of
\fullref{pro:Nab} can be obtained as follows.
The subgroup $N$ is a free group, as explained above.
In the case when $B$ is a strictly
quadratic word, a free Schreier basis of $N$ was explicitely constructed
by Zieschang~\cite{Z1966}, Zieschang, Vogt and Coldewey~\cite{ZVC}; see
also Kudryavtseva, Weidmann and Zieschang~\cite[Proposition~4.9]{KWZ}. This
immediately implies \fullref{pro:Nab}, see~\cite[Corollary~4.12]{KWZ}.

\begin {Pro} \label{pro:pontr}
Under the hypothesis of \fullref{pro:Nab}, consider
the central short exact sequence
$$1\to N_1/[N,N_1] \to N/[N,N_1]\to N/N_1 \to 1.$$
Then $N_1/[N,N_1]\approx H_2(N/N_1)\approx\ZZ[J]$, a free abelian group with
basis denoted by $e_{\theta}$ where $\theta$ runs over the set
$J=(\pi\times \pi\setminus\Delta)/\Sigma_2$, and
$\Sigma_2$ is the symmetric group in two symbols, which acts on
$\pi\times \pi\setminus\Delta$ by permutations of the coordinates.
A presentation of the group $N_1/[N,N_1]$ is obtained as follows: for each
$\theta \in (\pi\times \pi\setminus\Delta)/\Sigma_2$ choose a
pair $(\xi,\eta)\in \theta$, denote $\epsilon_{(\xi,\eta)}:=e_\theta$,
$\epsilon_{(\eta,\xi)}:=-e_\theta$, and denote by $J_1$ the set of such pairs
$(\xi,\eta)$, thus $J_1\subset\pi\times\pi\setminus\Delta$.
Then there exists a short exact sequence
\begin{equation}\label {eq:shortN1NN1}
1\to [N,N_1]\stackrel{i_{N_1}}{\longrightarrow} N_1 \stackrel{q_{N_1}}
{\longrightarrow} \ZZ[J] \to 0,
\end{equation}
where $i_{N_1}$ is the canonical inclusion, while $q_{N_1}$ is an epimorphism
sending
\begin{equation}\label {eq:N1NN1}
%\varphi_{N_1}
q_{N_1}\co  N_1 \to \ZZ[J], \qquad
\left[ \prod_{i=1}^r B_{u_i}^{n_i} , \prod_{j=1}^s B_{v_j}^{m_j} \right]
\longmapsto
\sum_{i=1}^r \sum_{j=1}^s n_i m_i \epsilon_{(\bar u_i,\bar v_j)} \in \ZZ[J],
\end{equation}
where $u_i,v_j\in F$, $1\le i\le r$, $1\le j\le s$, and $B_u=uBu^{-1}$, $u\in F$.
Furthermore, the group $N/[N,N_1]$ admits the following presentations:
\begin{multline}\label{eq:presentation}
N/[N,N_1] \approx \langle x_\xi,\ \xi\in\pi \mid
[x_\xi,[x_\eta,x_\zeta]],\ \xi,\eta,\zeta\in\pi\rangle \\
\approx \left\langle \begin{array}{l}
e_\theta,\ \theta\in(\pi\times \pi\setminus\Delta)/\Sigma_2, \\
x_\xi,\ \xi\in\pi \end{array} \left| \begin{array}{l}
[e_\theta,e_{\theta'}],\ [e_\theta,x_\xi],\ \theta,\theta'\in
(\pi\times \pi\setminus\Delta)/\Sigma_2,\ \xi\in\pi, \\
{} [x_\xi,x_\eta] e^{-1}_{\{\xi,\eta\}},\ (\xi,\eta)\in J_1
\end{array} \right. \!
\right \rangle\!,
\end{multline}
where $\{\xi,\eta\}\in (\pi\times \pi\setminus\Delta)/\Sigma_2$ denotes the
class of $(\xi,\eta)\in J_1$ in $(\pi\times \pi\setminus\Delta)/\Sigma_2$.
\end{Pro}

\begin{proof}
Recall that if $1\to H\to G\to Q\to 1$ is a short exact sequence then we
have a 5--term exact sequence
\begin{equation}\label{eq:Stallings}
H_2(G) \to H_2(Q)\to H/[G,H] \to H_1(G)\to H_1(Q)\to 0,
\end{equation}
due to Stallings~\cite[Theorem~2.1]{S}. Applying this to the short exact
sequence
$$1 \to N_1 \to N \to N/N_1 \to 1$$
we obtain that
$G=N$, $H=N_1$, $Q=N/N_1$, thus the first and third homomorphisms in the
5--term sequence are trivial.
It follows
that the second homomorphism $H_2(N/N_1)\to N_1/[N,N_1]$ is an isomorphism.
Since $Q\approx (\ZZ[\pi],+)$ is a free abelian group, it follows from
Brown~\cite[Theorem~V.6.4]{Br}
that $H_2(Q;\ZZ)\approx\Lambda^2(Q)$, the subgroup of
grade 2 of the graded ring $\Lambda(Q)$ (the exterior graded ring of the group
$Q$),
where $Q$ is at grade 0. This proves the desired presentation for the group
$N_1/[N,N_1]$.

To prove that~\eqref{eq:N1NN1} defines a homomorphism, let us first
show that there exists a unique homomorphism $q_{N_1}\co N_1\to\ZZ[J]$
satisfying~\eqref{eq:N1NN1}. Denote by $p_{N_1}\co N_1\to N_1/[N,N_1]$
the canonical projection. Consider the free basis $B_u$, $u\in W$, of
$N$ given by \fullref{pro:Nab}.  It follows from Magnus, Karrass and
Solitar~\cite[Theorem~5.12]{MKS} that the group $N_1/[N,N_1]$ is a
free abelian group, where the elements
$p_{N_1}([B_{s(u)},B_{s(v)}])\in N_1/[N,N_1]$, $(u,v)\in J_1$, form a
free abelian basis. Therefore the map sending
$p_{N_1}([B_{s(u)},B_{s(v)}])\mapsto e_{\{u,v\}}$, $(u,v)\in J_1$,
uniquely extends to a homomorphism $\varphi_{N_1}\co
N_1/[N,N_1]\to\ZZ[J]$. Since $\varphi_{N_1}$ sends the above basis of
$N_1/[N,N_1]$ to a basis of $\ZZ[J]$, it is an isomorphism. The
property~\eqref{eq:N1NN1} of the obtained projection
$q_{N_1}:=\varphi_{N_1}p_{N_1}$ follows from commutator calculus,
see~\cite[Theorem~5.3]{MKS}.

Now the presentation~\eqref{eq:presentation}
follows by observing that the natural epimorphism of $N/[N,N_1]$ to the group
in the right-hand side of~\eqref{eq:presentation} sending
$B_u[N,N_1]\mapsto x_{\bar u}$, $u\in W$, is well-defined. It has a trivial
kernel, because one can easily construct its inverse.
\end{proof}

\subsection{The groups $N_1/[F,N_1]$ and $N/[F,N_1]$} \label{subsec:NF}

Here we will obtain the main results of this section, which will be applied
in~\fullref{sec:Quad3} to study the existence of solutions of the
equations~\eqref{eqn1'}--\eqref{eqn4'} in the mixed cases, see \fullref{rem:tables}.

Let us first recall some other facts from Lyndon~\cite{L1} on the homology
of one-relator groups.

\begin{Lem}[Lyndon {{\cite[Theorem~2.1]{L1}}}] \label{lem:zeroho}
For any group $\pi$, the homology group $H_i(\pi,\ZZ[\pi])$ is trivial for
$i\ge 1$, and isomorphic to $\ZZ$ for $i=0$. Here the local
coefficients $\ZZ[\pi]$ is the $\ZZ[\pi]$--module corresponding to the action
of $\pi$ on $\ZZ[\pi]$, which is given by the right multiplication.
\qed
\end{Lem}

The result above is also true if $\pi$ acts on $\ZZ[\pi]$ by the left
multiplication.

\begin{Lem} \label{lem:properties}
Under the hypothesis of \fullref{pro:Nab},
$H_i(\pi)=0$, $i\ge 3$, while $H_2(\pi)$ is either $\ZZ$ or $0$.
The latter group is $0$ if and only if $B\not\in[F,F]$.
\end{Lem}

\begin{proof}
From~\cite[Corollaries~4.2 and~11.2]{L1} it follows that $H^i(\pi)=0$
for $i\ge3$, while $H^2(\pi)$ is a cyclic group, which is finite if and only if
$B\not\in[F,F]$. Using the universal coefficient theorem,
we ontain the desired assertion, see also Brown~\cite[Example~II.4.3]{Br}.
\end{proof}

Remark that we will not use in this paper that $H_i(\pi)=0$ for
$i\ge 4$.

Now denote $N_F=[F,N_1]$ where $N_1=[N,N]$.
Similarly to~\eqref{eq:Q}, denote by $Q$ the quotient of the abelian group
$\ZZ[\pi\setminus\{1\}]$ by the system of relations $g\sim-g^{-1}$,
$g\in\pi\setminus\{1\}$:
\begin{equation}\label{eq:Qgeneral}
Q=(\ZZ[\pi\setminus\{1\}])/\langle g+g^{-1}\,|\,g\in\pi\setminus\{1\}\rangle .
\end{equation}

\begin {Pro} \label{pro:QH2}
Under the hypothesis of \fullref{pro:Nab},
consider the central short exact sequence
$$1\to N_1/N_F \to N/N_F \to N^{\ab} \to 1.$$
Then $N_1/N_F\approx H_2(F/N_1)\approx Q\approx\ZZ[I]$,
for $I=(\pi\backslash \{1\})/\sim$, where the relation $\sim$ is given by
identifying $g$ with $g^{-1}$, for $g\in\pi\setminus\{1\}$.
Moreover, there exists a short exact sequence
\begin{equation}\label {eq:shortQisom}
1\to [F,N_1]\stackrel{i_{N_F}}{\longrightarrow} N_1 \stackrel{q_{N_F}}
{\longrightarrow} Q \to 0,
\end{equation}
where $i_{N_F}$ is the canonical inclusion, while $q_{N_F}$ is an epimorphism
sending
\begin{equation}\label {eq:Qisom}
%\varphi_{N_F}
q_{N_F}\co  N_1 \to
 % N_1/[F,N_1]\approx
Q, \quad
\Biggl[ \prod_{i=1}^r B_{u_i}^{n_i}, \prod_{j=1}^s B_{v_j}^{m_j} \Biggr]
\longmapsto
p_Q\Biggl(\sum_{i=1}^r \sum_{j=1}^s n_i m_i \bar u_i^{\;-1}\bar v_j \Biggr) \in Q,
\end{equation}
where $u_i,v_i\in F$, $p_Q\co \ZZ[\pi]\to Q$ is the projection.
Furthermore, the group $N/N_F$ admits the following presentations:
\begin{equation}\label{eq:presentationNNF}
\begin{aligned}
N/N_F &\approx \langle x_\xi,\ \xi\in\pi \mid
[x_\xi,x_\eta][x_{\zeta\xi},x_{\zeta\eta}]^{-1},\
[x_\xi,[x_\eta,x_\zeta]],\ \xi,\eta,\zeta\in\pi\rangle \\
&\approx\left\langle \begin{array}{l}
e_\theta,\ \theta\in\pi\setminus\{1\}, \\
x_\xi,\ \xi\in\pi \end{array}
\left| \begin{array}{l}
\ [e_\theta,e_{\theta'}],\ [e_\theta,x_\xi],\ \theta,\theta'\in
\pi\setminus\{1\},\ \xi\in\pi, \\
\ [x_\xi,x_\eta] e^{-1}_{\xi^{-1}\eta},\ (\xi,\eta)\in\pi\times\pi\setminus\Delta
\end{array} \right.
\right \rangle.
\end{aligned}
\end{equation}
\end{Pro}

\begin{proof}
To establish the isomorphism $N_1/N_F\approx H_2(F/N_1)$, consider
 % apply~\eqref{eq:Stallings} to
the 5--term exact sequence
 $$
H_2(F) \to H_2(F/N_1) \to N_1/[F,N_1] \to F^{\ab} \to (F/N_1)^{\ab} \to 0
 $$
obtained from the short exact sequence $1 \to N_1 \to F \to F/N_1 \to 1$
by means of~\eqref{eq:Stallings}.
Since $H_2(F)=0$ and $F^{\ab}\to (F/N_1)^{\ab}$ is an isomorphism, it follows
that $H_2(F/N_1)\to N_1/[F,N_1]=N_1/N_F$ is an isomorphism.

To establish the isomorphism $H_2(F/N_1)\approx Q$,
%apply
consider the Hochschild--Serre spectral sequence~\cite{HS} (also called
the Lyndon--Hochschild--Serre spectral sequence) related to the short
exact sequence
 $
1\to N/N_1\to F/N_1\to \pi \to 1
 $.
Recall that this spectral sequence has the form
\begin{equation}\label{eq:E2}
E^2_{pq}=H_p(\pi,H_q(N/N_1)) \implies H_{p+q}(F/N_1),
\end{equation}
where the local coefficients $H_q(N/N_1)$ is the $\ZZ[\pi]$--module
corresponding to the action $\Ad_\pi^q\co \pi\to \Aut(H_q(N/N_1))$,
which is induced by the action $\Ad_\pi|_{N/N_1}\co \pi\to \Aut(N/N_1)$
given by conjugation: $(gN)\cdot(xN_1)=gxg^{-1}N_1$, $g\in F$, $x\in
N$.  By \fullref{pro:Nab}, $H_1(N/N_1)\approx(\ZZ[\pi],+)$ and the
action $\Ad_\pi^1$ is given by the left multiplication. Thus, from
\fullref{lem:zeroho}, we have that $E^2_{p1}=H_p(\pi,\ZZ[\pi])=0$ for
$p\ge 1$. By \fullref{lem:properties}, we have $E^2_{30}=H_3(\pi)=0$,
which implies $E^2_{02}=E^{\infty}_{02}$.

Let us show that $H_2(F/N_1)\approx E^2_{02}$.  If $B\not\in[F,F]$
then, by \fullref{lem:properties}, $E^2_{20}=E^{\infty}_{20}=0$, and
thus we get $H_2(F/N_1)=E^2_{02}$. Consider the remaining case,
$B\in[F,F]$.  Observe that both groups $E^2_{20}=H_2(\pi)$ and
$E^2_{01}=H_0(\pi,\ZZ[\pi])$ are isomorphic to $\ZZ$, due to
Lemmas~\ref{lem:properties} and~\ref{lem:zeroho}, respectively.  On
the other hand, the isomorphism $H_1(F/N_1)\approx
F^{ab}\approx\pi^{ab}\approx H_1(\pi)=E^2_{10}=E^\infty_{10}$ and
\eqref{eq:E2} for $p+q=1$ imply $E^\infty_{01}=0$. Thus the
differential $d^2_{20}\co E^2_{20}\to E^2_{01}$ is an isomorphism, and
it follows that $H_2(F/N_1)\approx E^2_{02}=H_0(\pi,H_2(N^{\ab}))$.

Now, due to \fullref{pro:pontr}, $H_2(N^{\ab})$ is isomorphic to
$\ZZ[J]$, where $J=(\pi\times\pi\setminus\Delta)/\Sigma_2$, and the
corresponding action $\Ad_\pi^2\co \pi \to
\Aut(H_2(N^{\ab}))\approx\Aut(\ZZ[J])$ is given by
$\zeta\cdot e_{\{\xi,\eta\}}=e_{\{\zeta\xi,\zeta\eta\}}$,
for each pair $(\xi,\eta)\in J_1$, and $\zeta \in\pi$, where
$J_1\subset \pi\times \pi\setminus\Delta$ is chosen to be invariant
under $\Ad_\pi^2$, see \fullref{pro:pontr}.  Therefore
$E^2_{02}=H_0(\pi,H_2(N^{\ab}))\approx H_0(\pi,\ZZ[J])$ is isomorphic
to the quotient of $\ZZ[J]$ by the system of relations
$e_{\{\xi,\eta\}}\sim e_{\{\zeta\xi,\zeta\eta\}}$, $(\xi,\eta)\in
J_1$, $\zeta\in\pi$.  Hence it is isomorphic to $\ZZ[I]\approx Q$.

To prove that~\eqref{eq:Qisom} defines a homomorphism $q_{N_F}$,
observe that the canonical projection $N_1/[N,N_1]\to N_1/[F,N_1]$
factors through the canonical projection of the group $N_1/[N,N_1]$
onto the quotient of $N_1/[N,N_1]$ by the system of relations
$p_{N_1}(n)\sim p_{N_1}(gng^{-1})$, $n\in N_1$, $g\in F$, where
$p_{N_1}\co N_1\to N_1/[N,N_1]$ is the canonical projection, see also
\fullref{pro:pontr}. Due to the isomorphism $N_1/[N,N_1]\to\ZZ[J]$
from \fullref{pro:pontr}, we obtain the system of relations
$e_{\{\xi,\eta\}}\sim e_{\{\zeta\xi,\zeta\eta\}}$, $(\xi,\eta)\in J_1$,
$\zeta\in\pi$, on $\ZZ[J]$.  This system of relations determines the
obvious equivalence relation $\sim$ on the basis $e_{\{\xi,\eta\}}$,
$(\xi,\eta)\in J_1$, of $\ZZ[J]$.  Thus the desired quotient of
$\ZZ[J]$ is the free abelian group $\ZZ[J/\sim]$, where the
equivalence classes of $\sim$ form a basis. This gives the desired
isomorphism $N_1/[F,N_1]\approx \ZZ[J/\sim]=\ZZ[I]$.
Now~\eqref{eq:Qisom} follows by observing that the equivalence class
of $e_{\{\bar u_i,\bar v_j\}}=q_{N_1}([B_{u_i},B_{v_j}])$ in $J$
corresponds to $p_Q(\bar u_i^{\;-1}\bar v_j)=q_{N_F}([B_{u_i},B_{v_j}])$
under the isomorphism $\ZZ[I]\approx Q$.

Now the presentation~\eqref{eq:presentationNNF}
follows by observing that the natural epimorphism of $N/[F,N_1]$ to the group
in the right-hand side of~\eqref{eq:presentationNNF} sending
$B_u[F,N_1]\mapsto x_{\bar u}$, $u\in W$, is well-defined. It has a trivial
kernel, because one can easily construct its inverse.
\end{proof}

\begin{Rem} \label {rem:Hilton}
Our first derived equation is the ``projection" of the equation~\eqref{eqN}
in $N$ to the quotient $N^{\ab}=N/[N,N]$, see \fullref{thm:first}. Our
second derived equation is the ``projection" of the equation~\eqref{eqN} to
the quotient $[N,N]/[F,[N,N]]$, via choosing suitable representatives of the
solutions of the first derived equation, see \fullref{thm:second}.
Observe that the subgroup $[N,N]$ is the second term $\Gamma^2(N)$ of the lower
central series
 $$
\Gamma^1(N)=N, \quad \Gamma^{i+1}(N)=[N,\Gamma^i(N)], \quad i\geq 1,
 $$
of the group $N$, while the subgroup $[F,[N,N]]$ is the third term
$\Gamma^3_F(N)$ of the lower central series
 $$
\Gamma^1_F(N)=N, \quad \Gamma^{i+1}_F(N)=[F,\Gamma^i(N)], \quad i\geq 1,
 $$
with respect to the action of the group $F$ on $N$ by conjugation, that is
$g\cdot x=gxg^{-1}$, $g\in F$, $x\in N$.
We recall (see Hilton~\cite{Hil}, or Hilton, Mislin and
Roitberg~\cite[Section~II.2]{HMR})
that if a group $G$ acts on a group $H$ then
the {\it lower central series with respect to the action of $G$ on $H$} is
defined as
 $$
\Gamma_G^1(H)=H, \quad \Gamma_G^{i+1}(H)=\gr\{(g\cdot x)yx^{-1}y^{-1}
\mid g\in G, \ x\in \Gamma^i(H), \ y\in H\}, \ i\geq 1.
 $$
Here $g\cdot x$ means the action of the automorphism defined by $g$ on the
element $x$, $\Gamma^i(H)$ is the usual lower central series of the group
$H$, and $\gr S$ denotes the minimal subgroup of $H$ containing a subset
$S\subset H$.
\end{Rem}

\subsection{The groups $N/[F,N]$ and $[F,N]/[F,[F,N]]$} \label{subsec:Q}

Here we study the quotients corresponding to the subgroups
$N\supset [F,N]\supset [F,[F,N]]$. One can apply them to study existence of
solutions of equations in free groups. However, the results of this subsection
are not used in our applications, and can be skipped in the first reading.

\begin{Pro} \label{pro:NF}
Under the hypothesis of \fullref{pro:Nab},
the group $N/[F,N]$ is isomorphic to $\ZZ$.
Moreover, there exists a short exact sequence
 $$
1\to [F,N] \stackrel{i}{\longrightarrow} N
\stackrel{\varepsilon q_N}{\longrightarrow} \ZZ \to 0,
 $$
where $i$ is the canonical inclusion,
$q_N\co  N\to (\ZZ[\pi],+)$ is the epimorphism given by~\eqref{eq:Nab},
while $\varepsilon\co \ZZ[\pi]\to\ZZ$ is the augmentation.
\end{Pro}

\begin{proof}
The first assertion follows in a straightforward way from the 5--term
exact sequence obtained from
the short exact sequence~\eqref{eq:shortFN}, namely
$1\to N \to F \to \pi \to 1$, by means of~\eqref{eq:Stallings}.
In detail, if $B\in[F,F]$ then $H_2(F)=0$,
$F^{\ab}\to\pi^{\ab}$ is an isomorphism, and $H_2(\pi)\approx\ZZ$, due to
\fullref{lem:properties}. Therefore
$H_2(\pi)\to N/[F,N]$ is an isomorphism, hence $N/[F,N]\approx\ZZ$.
Suppose that $B\not\in[F,F]$. Then $\ker(F^{\ab}\to\pi^{\ab})\approx\ZZ$, and
$H_2(\pi)=0$, due to \fullref{lem:properties}.
Therefore $N/[F,N]\approx \ker(F^{\ab}\to\pi^{\ab})\approx\ZZ$.

To prove the second assertion, observe that the composition
$\varepsilon q_N\co N\to \ZZ$ sends $B\mapsto 1$,
hence it is an epimorphism. Next we show that the kernel of $\varepsilon q_N$
equals $[F,N]$. The inclusion $\ker(\varepsilon q_N) \supset [F,N]$ follows
from the fact that $[F,N]$ is generated by the elements $[u,B_v]\in[F,N]$,
$u,v\in F$,
due to commutator calculus, while $[u,B_v]=B_{uv}B^{-1}_v$ is mapped to $1-1=0$
under $\varepsilon q_N$, thus $[u,B_v]\in\ker(\varepsilon q_N)$.
The converse inclusion follows by observing that any element
$u\in \ker(\varepsilon q_N)$ has the form
$u=B_{u_1}^{c_1}\ldots B_{u_r}^{c_r}$, for some
$r,c_1,\dots,c_r\in\ZZ$, $r\ge 0$, $u_1,\dots,u_r\in F$, where
$c_1+\ldots+c_r=0$. Clearly, the projection of $u$ to the quotient
$N/[F,N]$ equals the projection of the element
$B^{c_1}\ldots B^{c_r}=B^0=1$ to $N/[F,N]$, thus $u\in [F,N]$.
\end{proof}

\begin{Pro} \label {pro:FFN}
Under the hypothesis of \fullref{pro:Nab},
% there are group isomorphisms
$$[F,N]/[F,[F,N]]\approx H_2(F/[F,N])\approx
H_1(\pi)\approx\pi^{\ab}.$$
Moreover, there exists a short exact sequence
 $$
1\to [F,[F,N]] \stackrel{i_F}{\longrightarrow} [F,N]
\stackrel{q_F}{\longrightarrow} \pi^{\ab} \to 0,
 $$
where $i_F$ is the canonical inclusion, $q_F$ is an epimorphism
sending $q_F\co [n,g]\mapsto \varepsilon q_N(n) \cdot
p_{\ab}(\bar g)\in\pi^{\ab}$, $n\in N$, $g\in F$. Here one uses an
additive notation for the group operation in $\pi^{\ab}$,
$p_{\ab}\co \pi\to \pi^{\ab}$ denotes the canonical projection,
$q_N\co N\to (\ZZ[\pi],+)$ is the epimorphism defined by~\eqref{eq:Nab}, $\varepsilon \co \ZZ[\pi]\to \ZZ$ is the
augmentation.
\end{Pro}

\begin{proof}
Let us prove the first assertion. Consider the 5--term exact sequence
 $$
H_2(F) \to H_2(F/[F,N]) \to [F,N]/[F,[F,N]] \to F^{\ab} \to (F/[F,N])^{\ab} \to 0
 $$
obtained from the short exact sequence $1\to[F,N]\to F\to F/[F,N]\to1$
by means of~\eqref{eq:Stallings}. Since $H_2(F)=0$ and
$F^{\ab}\to (F/[F,N])^{\ab}$ is an isomorphism, it follows that
$H_2(F/[F,N])\to[F,N]/[F,[F,N]]$ is an isomorphism.

Similarly to the proof of \fullref{pro:QH2},
consider the Hochschild--Serre spectral sequence related to
the short exact sequence $1\to N/[F,N] \to F/[F,N] \to \pi\to 1$.
Since $N/[F,N] \approx \ZZ$ by \fullref{pro:NF},
the spectral sequence has the form
 $$
E^2_{pq}=H_p(\pi,H_q(N/[F,N])) \approx H_p(\pi,H_q(\ZZ))
 \implies H_{p+q}(F/[F,N]),
 $$
similarly to~\eqref{eq:E2}, where the local coefficients $H_q(\ZZ)$ is the
trivial $\ZZ[\pi]$--module. Since $H_0(\ZZ)\approx H_1(\ZZ)\approx \ZZ$ and
%the homology of $\ZZ$ vanishes in dimensions different from 0 and 1,
$H_q(\ZZ)=0$ for $q\ge 2$, while the homology of $\pi$
vanishes in dimension 3 (due to \fullref{lem:properties}),
the only possible non-vanishing $E^2$ terms are those with $q=0,1$ and $p\ne 3$.
In particular, $E^2_{02}=0$, $E^2_{30}=H_3(\pi)=0$, and therefore
$E^\infty_{11}=E^2_{11}=H_1(\pi)$.

Let us show that $E^\infty_{20}=E^3_{20}=0$.
If $B\not\in[F,F]$ then, by \fullref{lem:properties},
$E^2_{20}=H_2(\pi)=0$.
Suppose that $B\in[F,F]$. Then $E^2_{20}=H_2(\pi)\approx\ZZ$
by \fullref{lem:properties}, moreover the projection
$(F/[F,N])^{\ab}\to H_1(\pi)=E^2_{10}$ is an isomorphism. Therefore
$d^2_{20}\co E^2_{20}\to E^2_{01}$ is an isomorphism, hence
$E^\infty_{20}=E^3_{20}=0$.

Since $E^2_{02}=0$ and $E^3_{20}=0$, we have the desired isomorphism
$H_2(F/[F,N])\approx E^\infty_{11}=E^2_{11}=H_1(\pi)$.

Let us prove the second assertion. We shall represent elements of
the quotient $[F,N]/[F,[F,N]]$ by elements of $[F,N]$, identified
under the congruence relation $g_1\equiv g_2$ {\it modulo}
$[F,[F,N]]$, and shall write $g_1\equiv g_2$ whenever
$g_1g_2^{-1}\in [F,[F,N]]$. Observe that every element $w\in[F,N]$
can be written in the form $w\equiv [B,u]$, for some $u\in F$, due
to the following congruences: $[B_u,v]\equiv [B,u^{-1}vu]$ and
$[B,uv]=[B,u][u,[B,v]][B,v]\equiv[B,u][B,v]$ for any $u,v\in F$.

Let us show that there exists an epimorphism
$p\co \pi^{\ab}\to [F,N]/[F,[F,N]]$ sending
$p_{\ab}(\bar u)\mapsto p_F([B,u])$, $u\in F$, where
$p_F\co [F,N]\to [F,N]/[F,[F,N]]$ is the canonical projection. To show that
such a map $p$ is well-defined, we use commutator
calculus and the following observations.
Using $[B,B_v]=[B,[v,B]]\in[N,[F,N]]\subset [F,[F,N]]$, $v\in F$, one shows
that $[N,N]\subset[F,[F,N]]$, which implies
$[B,un]\equiv[B,u][B,n]\equiv[B,u]$ for any $u\in F$ and $n\in N$.
Furthermore, using one of the Witt--Hall identities (see Magnus, Karrass
and Solitar~\cite[Theorem~5.1, (11)]{MKS})
one can show that $[N,[F,F]]\subset[F,[F,N]]$, which implies
$[B,uf']\equiv[B,u][B,f']\equiv[B,u]$ for any $u\in F$ and $f'\in[F,F]$.
The map $p$ is a homomorphism, since $[B,uv]\equiv[B,u][B,v]$ for any
$u,v\in F$, see above. Therefore the map $p$ is an epimorphism.

Since $\pi^{\ab}\approx[F,N]/[F,[F,N]]$, and $\pi^{\ab}$ is a finitely-generated
abelian group, it follows that any epimorphism $\pi^{\ab}\to [F,N]/[F,[F,N]]$ is
an isomorphism. Therefore the epimorphism $p$ is an isomorphism. It follows
that the composition $p^{-1}p_F\co [F,N]\to\pi^{\ab}$
is an epimorphism and satisfies the desired properties.
\end{proof}

\section{Derived equations in $\ZZ[\pi]$ and $\ZZ[\pi\setminus\{1\}]/\sim$}
\label{sec:Quad3}

The quadratic equations under consideration are the equations
\eqref{eqn1'}--\eqref{eqn4'}
of \fullref{subsec:appl}
with two unknowns $\xalpha\in N$, $\ybeta\in F_2$ in the free group
$F_2=\langle \aa,\bb\mid\rangle$ of rank~2, see \fullref{thm:class}.
Actually these equations are in the subgroup $N=\llangle B\rrangle $ where
$B=\aa\bb\aa^{-\varepsilon}\bb^{-1}$.
To prove some further non-existence results, we will apply the algebraic
approach developed in~\fullref{sec:Quad2}. For each of the
equations~\eqref{eqn2'}, \eqref{eqn3'} and~\eqref{eqn4'} in $N$,
we will construct two {\it derived equations}, which are in fact ``projections''
of the equation to the abelian quotients $N/N_1$ and $N_1/[F_2,N_1]$,
respectively, described in~\fullref{sec:Quad2},
see Propositions~\ref{pro:Nab} and~\ref{pro:QH2},
where $N_1=[N,N]$. The first derived equation is an equation in the group ring
$\ZZ[\pi]$ of the fundamental group $\pi=\pi_\varepsilon=F_2/N$ of the
corresponding target surface (this group ring, as an abelian group, is
isomorphic to the abelianised group $N$, see \fullref{pro:Nab}).
The second derived equation is an equation
in the quotient $Q$ of $\ZZ[\pi]$, see~\eqref{eq:Q}, and it is
obtained by ``projecting'' the equation to this quotient (actually, to
$N_1/[F_2,N_1]\approx Q$, see \fullref{pro:QH2}), via choosing
suitable representatives of the
solutions (if there exists any) of the first derived equation.

\subsection{The first derived equation} \label{subsec:1der}

Here we will construct the {\it first derived equation} for each of the
equations \eqref{eqn2'}, \eqref{eqn3'} and~\eqref{eqn4'} of
\fullref{subsec:appl}.
Due to \fullref{thm:class}, or Corollaries~\ref{cor:Wecken}(A)
and~\ref{cor:cond},
we can assume, without loss of generality, that $\xalpha\in N$, for a solution
$(\xalpha,\ybeta)$ of~\eqref{eqN}.
So, the left-hand side of the equation~\eqref{eqN} is the product of $\xalpha$
and $\ybeta \xalpha^{-\delta} \ybeta^{-1}$ where both elements belong to $N$.
The right-hand side is also the product of two elements of $N$, whose
projections to $N^{\ab}=N/[N,N]\approx(\ZZ[\pi],+)$ are $\vartheta \bar v$ and $1$,
respectively, where $\bar v=p_\pi(v)$ and $p_\pi\co F_2\to\pi$ is the
projection, see~\eqref{eq:shortNab}, \eqref{eq:Nab}.
So, we can project both sides of the equation to $N^{\ab}=N/[N,N]$, and we get:

\begin{Thm} \label{thm:first}
Suppose that $(\xalpha,\ybeta)$ is a solution of the equation~\eqref{eqN}
with $\xalpha\in N=\llangle  \aa\bb\aa^{-\varepsilon}\bb^{-1}\rrangle $,
$\ybeta\in F_2=\langle \aa,\bb\mid\rangle$. Let
$\tilde\xalpha=q_N(\xalpha)\in \ZZ[\pi]$,
$\bar\ybeta=p_\pi(\ybeta)\in \pi=F_2/N$
be the images of $\xalpha$, $\ybeta$ under the projections
$q_N\co N\to (\ZZ[\pi],+)\approx N^{\ab}=N/[N,N]$ and $p_\pi\co F_2\to\pi$,
respectively. Here
the natural identification of $N^{\ab}$, the abelianised group $N$,
with the group $(\ZZ[\pi],+)$ is given by~\eqref{eq:shortNab}, \eqref{eq:Nab}.
Then the pair $(\tilde\xalpha,\bar\ybeta)$ satisfies the following equation
called the {\it first derived equation}:
\begin{equation}\label {first:eq}
(1-\delta\bar\ybeta) \tilde\xalpha = 1 + \vartheta \bar v
\end{equation}
in the group ring $\ZZ[\pi]$,
with the ``unknowns'' $\bar\ybeta\in\pi$ and $\tilde\xalpha\in\ZZ[\pi]$.
Moreover, the properties~$\eqref{eq:faithful}$, $\eqref{eq:rank1}$ are valid.
Furthermore, any solution $(\tilde x,\bar y)$ of~$\eqref{first:eq}$
satisfies~$\eqref{eq:rank1}$.
\end{Thm}

\begin{proof}
The group $N^{\ab}$
is isomorphic to the abelian group $(\ZZ[\pi],+)$, see \fullref{pro:Nab}.
Under this isomorphism, the element $B=\aa\bb\aa^{-\varepsilon}\bb^{-1}\in N$
is identified with $1\in\pi\subset\ZZ[\pi]$, and $B_u$ with
$\bar u\in\pi\subset\ZZ[\pi]$, thus the right-hand side of the equation is
identified with $1+\vartheta \bar v$. Moreover, the conjugation of
$B_u=uBu^{-1}$, $u\in F_2$,
by an element $z\in F_2$ equals $B_{zu}$, which is identified with
$\bar z\bar u\in\pi\subset\ZZ[\pi]$.
It follows that the projection of the left-hand side to $N^{\ab}$ equals
$\tilde\xalpha-\delta\bar\ybeta\tilde\xalpha=(1-\delta\bar\ybeta)\tilde\xalpha$,
which gives~\eqref{first:eq}. The properties~\eqref{eq:faithful},
\eqref{eq:rank1} are due to \fullref{rem:rank1}.

Let us derive the property~\eqref{eq:rank1} from~\eqref{first:eq}. Suppose
that $(\tilde x,\bar y)\in\ZZ[\pi]\times\pi$
is a solution of~\eqref{first:eq}. Consider the left action of the infinite
cyclic group $\langle t\rangle\approx\ZZ$ on $\pi$ via $t\cdot g=\bar y g$,
$g\in\pi$. Consider the orbits $\L_g=\ZZ\cdot g$, $g\in\pi$. For any
$\tilde x\in\ZZ[\pi]$, denote by
$\tilde x_g\in\ZZ[\L_g]$ the image of $\tilde x$ under the projection
$\ZZ[\pi]\to\ZZ[\L_g]$, $g\in\pi$. It follows from~\eqref{first:eq} that
$(1-\delta\bar y)\tilde x_1=1+\vartheta\bar v$ if $\bar v\in\L_1$, and
$(1-\delta\bar y)\tilde x_1=1$ if $\bar v\not\in\L_1$.
Since the augmentation of the left-hand side is even, this implies
$\bar v\in\L_1=\langle\bar y\rangle$, thus $\bar v=\bar y^k$ for some $k\in\ZZ$.
If $\bar y=1$ and $\delta=-1$, the property~\eqref{eq:rank1} is now obvious,
since it is equivalent to $1=1^k$ and $(-1)^k\vartheta=-1$, for some $k\in\ZZ$.
In the remaining case ($\bar y\ne1$ or $\delta=1$), consider the homomorphism
$\chi\co \langle\bar y\rangle\to\ZZ^*=\{1,-1\}\subset\ZZ$
sending $\bar y\mapsto\delta$ (it is well-defined, since $\pi$ is a torsion
free group). By extending $\chi$ linearly to the group ring
$\ZZ[\langle\bar y\rangle]$, one obtains the {\it $\chi$--twisted augmentation}
$\varepsilon_\chi\co \ZZ[\langle\bar y\rangle]\to\ZZ$. From above, we have
$(1-\delta\bar y)\tilde x_1=1+\vartheta\bar y^k$, where
$\tilde x_1\in\ZZ[\langle\bar y\rangle]$, a Laurent polynomial in
$\bar y$. Since $\chi$--twisted augmentation of the left-hand side vanishes, we
have $0=\varepsilon_\chi(1+\vartheta\bar y^k)=1+\vartheta\delta^k$. This completes
the derivation of~\eqref{eq:rank1} from~\eqref{first:eq}.
\end{proof}

\subsection{Solutions of the first derived equation in the ``mixed'' cases}
\label{subsec:sol1}

Here we study separately the solutions of the first derived
equation~\eqref{first:eq} in the {\it mixed cases} described in
\fullref{rem:tables} and \fullref{def:mixed}, see also
Tables~\ref{tbl1} and~\ref{tbl2}.  Recall that, for any
solution~$(\tilde x,\bar y)\in (\ZZ[\pi])\times\pi$ of the first
derived equation~\eqref{first:eq} in a mixed case, $\bar v$ belongs to
the cyclic subgroup of $\pi=\pi_\varepsilon$ generated by
$\bar\ybeta^2$, see \fullref{thm:class}(5). We will call a solution
$(\tilde x,\bar y)$ of \eqref{first:eq} {\it faithful} if $w_\varepsilon(\bar
y)=\delta$, see \eqref{eq:faithful}.

\subsubsection*{Case of the equation \eqref{eqn2'}} Here $\delta=1$, $\varepsilon=-1$.
We will only consider non-faithful solutions of~\eqref{first:eq}
for $\vartheta=-1$, $v\in F_2$ such that
$\bar v=p_K(v)=\bar \bb^{2n}$, $n\in\ZZ$, see~\fullref{rem:tables}.
Denote $c_L=\bb\aa^{-L}$, $L\in\ZZ$.

\begin{Lem} \label {lem:sol21} % {lem:allsonf}
For the equation~\eqref{eqn2'} with $\vartheta=-1$, the non-faithful solutions of
the first derived equation~$\eqref{first:eq}$ are described by
\begin{equation}\label{eqn2_1}
(1-\bar\ybeta) \tilde\xalpha = 1 - \bar v, \quad \mbox{where} \quad
w_-(\bar\ybeta)=-1,
\tag{$2_1$}
\end{equation}
in $\ZZ[\pi]$, with the unknowns $\tilde\xalpha\in\ZZ[\pi]$,
$\bar\ybeta\in\pi$, where $\pi=\pi_-$.
For $v$ satisfying $\bar v=p_K(v)=\bar \bb^{2n}$, $n\in\ZZ$, the
solutions of the equation~\eqref{eqn2_1} are given by
 $$
\bar \ybeta=\bar c_L^\ell=(\bar \aa^L \bar \bb)^\ell, \quad
\tilde \xalpha = \frac {1 - \bar c_L^{2n}}{1 - \bar c_L^\ell}
= \left\{ \begin{array}{ll}
1+\bar c_L^\ell+\bar c_L^{2\ell}+\ldots+\bar c_L^{2n-\ell}, & n/\ell > 0, \\
0, & n=0, \\
-\bar c_L^{\;-\ell}-\bar c_L^{\;-2\ell}-\ldots-\bar c_L^{2n}, & n/\ell < 0,
\end{array}\right.
 $$
where $L\in\ZZ$ is arbitrary, and $\ell$ runs over the set of all odd divisors
of $n$. For $n=0$ we assume that $\ell$ is any odd number.
\end{Lem}

\begin{proof}
The equation~\eqref{eqn2_1} follows from \fullref{thm:first}.
Suppose $\bar \ybeta=\bar\aa^L\bar\bb^\ell$.
Because the solution is non-faithful, it follows that $\ell$ is odd. Since
$\bar v=\bar \bb^{2n}$ belongs to the subgroup  generated by $\bar\ybeta$, it
follows that $\bar v=\bar\ybeta^k=(\bar \aa^{L}\bar \bb^{\ell})^k$, for some
$k\in\ZZ$. This implies that $\ell$ is a divisor of $2n$.
Thus $\bar\ybeta$ has the form given by the second part of the
Lemma. It follows by a straightforward calculation that all values of
$(\tilde\xalpha,\bar\ybeta)$ given by Lemma are solutions. That they are the
only solutions follows from the fact that the
group ring $\ZZ[\pi_-]$ has no zero divisors, since $\pi_-$ is a solvable
torsion free group, see Kropholler, Linnell and Moody~\cite[Theorem~1.4]{KLM}.
\end{proof}

\begin{Rem} \label {rem:repres2}
Later, the following representatives
$(\xalpha_{L,\ell},\ybeta_{L,\ell})\in N\times F_2$ of the
solutions $(\tilde\xalpha,\bar\ybeta)$ of \eqref{eqn2_1} from \fullref{lem:sol21}
will be used:
\begin{align*}
\ybeta_{L,\ell}&=c_L^\ell=(\bb\aa^{-L})^\ell, \\
\xalpha_{L,\ell}&= \begin{cases}
B_{c_L^{2n-\ell}}B_{c_L^{2n-2\ell}}B_{c_L^{2n-3\ell}} \ldots B_{c_L^\ell}B,
& n/\ell >0, \\
1, & n=0, \\
B^{-1}_{c_L^{2n}}B^{-1}_{c_L^{2n+\ell}}B^{-1}_{c_L^{2n+2\ell}} \ldots
B^{-1}_{c_L^{-2\ell}} B^{-1}_{c_L^{-\ell}},
& n/\ell<0,
\end{cases}
\end{align*}
where $B=\aa\bb\aa\bb^{-1}$, $\ell\ne0$ is any odd number if $n=0$, or any odd
divisor of $n$ if $n\ne 0$, thus the number of factors in the expression for
$\xalpha_{L,\ell}$ is even and equal to $2|n/\ell|$.
\end{Rem}

\subsubsection*{Case of the equation \eqref{eqn3'}}
Here $\delta=-1$, $\varepsilon=1$, and
all solutions are non-faithful. We will consider only the case
where $\vartheta=-1$ and $\bar v=p_T(v)=\bar \aa^{2m}\bar \bb^{2n}$,
see \fullref{rem:tables}.

If $|m|+|n|>0$, let us denote $d=\gcd(m,n)$ and $c=\aa^{m/d} \bb^{n/d}$.

\begin{Lem} \label {lem:sol31} % {lem:allso}
For the equation~\eqref{eqn3'} with $\vartheta=-1$, the solutions of
the first derived equation~$\eqref{first:eq}$
are described by
\begin{equation}
\label{eqn3_1}
(1+\bar\ybeta) \tilde\xalpha = 1 - \bar v
\tag{$3_1$}
\end{equation}
in $\ZZ[\pi]$, with the unknowns $\tilde\xalpha\in\ZZ[\pi]$,
$\bar\ybeta\in\pi$, where $\pi=\pi_+$.
For $v$ satisfying $\bar v=p_T(v)=\bar \aa^{2m}\bar \bb^{2n}$, $m,n\in\ZZ$,
$|m|+|n|>0$, all solutions of this equation are given by
\begin{align*}
\bar \ybeta&=\bar c^\ell , \\
\tilde \xalpha &= \frac {1 - \bar c^{2d}}{1 + \bar c^\ell}
= \begin{cases}
1-\bar c^\ell+\bar c^{2\ell}-\ldots
         +\bar c^{2d-2\ell}-\bar c^{2d-\ell}, & \ell > 0, \\
\bar c^{\;-\ell}-\bar c^{\;-2\ell}+\ldots
         -\bar c^{2d+2\ell}+\bar c^{2d+\ell}-\bar c^{2d}, & \ell < 0,
\end{cases}
\end{align*}
where $\ell\ne0$ is any divisor of $d=\gcd(m,n)$,
$\bar c=\bar \aa^{m/d}\bar \bb^{n/d}$.
If $v$ satisfies $\bar v=p_T(v)=1$ then all solutions are given by
$\tilde\xalpha=0$ and $\bar\ybeta\in\pi$ is any element.
\end{Lem}

\begin{proof}
The equation~\eqref{eqn3_1} follows from \fullref{thm:first}.
Let $\bar v=\bar \aa^{2m}\bar \bb^{2n}$ and $\bar\ybeta=\bar \aa^r\bar \bb^s$,
$m,n,r,s\in\ZZ$. Suppose $|m|+|n|>0$.
Since $\bar v$ belongs to the subgroup generated by $\bar\ybeta^2$, it
follows that $\bar v=\bar\ybeta^{2k}=\bar \aa^{2kr}\bar \bb^{2ks}$, for some
$k\in\ZZ$. This implies $kr=m$, $ks=n$, thus $k$ is a divisor of $d$, and
$\bar\ybeta=\bar\aa^{m/k}\bar \bb^{n/k}=\bar c^\ell$ where $\ell=d/k$.
Thus $\bar\ybeta$ has the form given by the second part of the Lemma.
It follows by a straightforward calculation that all values of
$(\tilde\xalpha,\bar\ybeta)$ given by Lemma are solutions. That they are the
only solutions follows from the fact that the group ring $\ZZ[\pi_+]$ has no
zero divisors.

Suppose $m=n=0$, thus $\bar v=1$, and the right-hand side of the
equation~\eqref{eqn3_1} vanishes. Since $1+\bar\ybeta\ne0$ in $\ZZ[\pi]$ for any
$\bar\ybeta\in\pi$, it follows that $\tilde\xalpha=0$, since the group
ring $\ZZ[\pi_+]$ has no zero divisors, since it is a polynomial ring.
\end{proof}

\begin{Rem} \label{rem:repres3}
Suppose that $\bar v\ne1$, thus $\bar v=\bar \aa^{2m}\bar \bb^{2n}$ with
$|m|+|n|>0$. Denote $d=\gcd(m,n)$, $c=\aa^{m/d}\bb^{n/d}$,
$B=\aa\bb\aa^{-1}\bb^{-1}$.
Later, the following representatives $(\xalpha_\ell,\ybeta_\ell)\in N\times F_2$
of the solutions $(\tilde\xalpha,\bar\ybeta)$ of \eqref{eqn3_1} from
\fullref{lem:sol31} will be used:
\begin{align*}
\ybeta_\ell&=c^\ell, \\
\xalpha_\ell&= \begin{cases}
B_{c^{2d-2\ell}}B_{c^{2d-4\ell}}{\ldots}B_{c^{2\ell}}B
B^{-1}_{c^\ell}B^{-1}_{c^{3\ell}}{\ldots}B^{-1}_{c^{2d-3\ell}}B^{-1}_{c^{2d-\ell}},
& \ell>0, \\
B^{-1}_{c^{2d}}B^{-1}_{c^{2d+2\ell}}{\ldots}B^{-1}_{c^{-4\ell}}B^{-1}_{c^{-2\ell}}
B_{c^{-\ell}}B_{c^{-3\ell}}{\ldots}B_{c^{2d+3\ell}} B_{c^{2d+\ell}},
& \ell<0,
\end{cases}
\end{align*}
where $\ell\ne0$ is any divisor of $d$, thus the number of factors in the
expression for $\xalpha_\ell$ is even and equal to $2d/|\ell|$.

For $\bar v=1$, we will use the representatives
$\xalpha_{L,\ell}=1$ and $\ybeta_{L,\ell}=\aa^L \bb^\ell$, where $L,\ell\in\ZZ$.
Actually $L,\ell$ coincide with the exponents in the canonical form of
$\bar\ybeta\in\pi_+$, see~\eqref{eq:canon}.
\end{Rem}

\subsubsection*{Case of the equation \eqref{eqn4'}}
Here $\delta=\varepsilon=-1$.
First we consider the case of faithful solutions where $\vartheta=-1$ and
$\bar v=p_K(v)=\bar \bb^{2n}$, see \fullref{rem:tables}.

As above, we denote $c_L=\bb\aa^{-L}$.

\begin{Lem} \label{lem:sol41f}
For the equation~\eqref{eqn4'} with $\vartheta=-1$, the faithful solutions of
the first derived equation~$\eqref{first:eq}$
are described by
\begin{equation}
\label{eqn4_1f}
(1+\bar\ybeta) \tilde\xalpha = 1 - \bar v, \quad \mbox{where} \quad
w_-(\bar\ybeta)=-1,
\tag{$4_1^\f$}
\end{equation}
in $\ZZ[\pi]$, with the unknowns $\tilde\xalpha\in\ZZ[\pi]$,
$\bar\ybeta\in\pi$, where $\pi=\pi_-$.
For $v$ satisfying $\bar v=p_K(v)=\bar \bb^{2n}$, $n\in\ZZ$, the
solutions of this equation are given by
\begin{align*}
\bar \ybeta&=\bar c_L^\ell=(\bar \aa^L \bar \bb)^\ell, \\
 \tilde \xalpha &= \frac {1 - \bar c_L^{2n}}{1 + \bar c_L^\ell} =
\begin{cases}
1-\bar c_L^\ell+\bar c_L^{2\ell}-\ldots+\bar
c_L^{2n-2\ell}-\bar c_L^{2n-\ell},
   & n/\ell > 0, \\
0, & n=0, \\
\bar c_L^{\;-\ell}-\bar c_L^{\;-2\ell}+\ldots+\bar c_L^{2n+\ell}-\bar c_L^{2n},
& n/\ell<0,
\end{cases}
\end{align*}
where $L\in\ZZ$ is arbitrary, and $\ell$ runs over the set of all odd divisors
of $n$. For $n=0$ we assume that $\ell$ is any odd number.
Compare \fullref{lem:sol21}.
\end{Lem}

\begin{proof}
Similar to that of \fullref{lem:sol21}.
\end{proof}

\begin{Rem} \label {rem:repres4f}
Later, the following representatives
$(\xalpha_{L,\ell},\ybeta_{L,\ell})\in N\times F_2$ of the
solutions $(\tilde\xalpha,\bar\ybeta)$ of \eqref{eqn4_1f} from
\fullref{lem:sol41f} will be used:
\begin{align*}
\ybeta_{L,\ell}&=c_L^\ell, \\
\xalpha_{L,\ell}&= \begin{cases}
B_{c_L^{2d-2\ell}}B_{c_L^{2d-4\ell}} \ldots B_{c_L^{2\ell}}B
B^{-1}_{c_L^\ell}B^{-1}_{c_L^{3\ell}}\ldots
B^{-1}_{c_L^{2d-3\ell}}B^{-1}_{c_L^{2d-\ell}},
& n/\ell >0, \\
1, & n=0, \\
B^{-1}_{c_L^{2d}}B^{-1}_{c_L^{2d+2\ell}} \dots
B^{-1}_{c_L^{-4\ell}}B^{-1}_{c_L^{-2\ell}}
B_{c_L^{-\ell}}B_{c_L^{-3\ell}}\ldots B_{c_L^{2d+3\ell}}
B_{c_L^{2d+\ell}}, & n/\ell<0,
\end{cases}
\end{align*}
where $c_L=\bb\aa^{-L}$,
$B=\aa\bb\aa\bb^{-1}$, $\ell\ne0$ is any odd number if $n=0$, or any odd
divisor of $n$ if $n\ne 0$, thus the number of factors in the expression for
$\xalpha_{L,\ell}$ is even and equal to $2|n/\ell|$.
Compare Remarks~\ref{rem:repres2} and~\ref{rem:repres3}.
\end{Rem}

Now consider the case of non-faithful solutions where $\vartheta=-1$ and
$\bar v=p_K(v)=\bar \aa^{2m}\bar \bb^{4n}$, see \fullref{rem:tables}.
If $|m|+|n|>0$, we denote $d=\gcd(m,n)$, $c=\aa^{m/d} \bb^{2n/d}$.

\begin{Lem}\label{lem:sol41nf}
For the equation~\eqref{eqn4'} with $\vartheta=-1$, the non-faithful solutions of
the first derived equation~\eqref{first:eq} are described by
\begin{equation}
\label{eqn4_1nf}
(1+\bar\ybeta) \tilde\xalpha = 1 - \bar v, \quad \mbox{where} \quad
w_-(\bar\ybeta)=1,
\tag{$4_1^\nf$}
\end{equation}
in $\ZZ[\pi]$, with the unknowns $\tilde\xalpha\in\ZZ[\pi]$,
$\bar\ybeta\in\pi$, where $\pi=\pi_-$.
For $v$ satisfying $\bar v=p_K(v)=\bar \aa^{2m}\bar \bb^{4n}$, $m,n\in\ZZ$,
$|m|+|n|>0$,
all non-faithful solutions of this equation are given by the same formulae
as in \fullref{lem:sol31}:
\begin{align*}
\bar \ybeta&=\bar c^\ell, \\
\tilde \xalpha &= \frac {1 - \bar c^{2d}}{1 + \bar c^\ell}
= \begin{cases}
1-\bar c^\ell+\bar c^{2\ell}-\ldots
         +\bar c^{2d-2\ell}-\bar c^{2d-\ell}, & \ell > 0, \\
\bar c^{\;-\ell}-\bar c^{\;-2\ell}+\ldots
         -\bar c^{2d+2\ell}+\bar c^{2d+\ell}-\bar c^{2d}, & \ell < 0,
\end{cases}
\end{align*}
where $\ell\ne0$ is any divisor of $d=\gcd(m,n)$,
$\bar c=\bar \aa^{m/d}\bar \bb^{2n/d}$.
If $v$ satisfies $\bar v=p_K(v)=1$ then all non-faithful solutions are given
by: $\tilde\xalpha=0$ and $\bar\ybeta\in\pi$ is any element satisfying
$w_-(\bar\ybeta)=1$.
\end{Lem}

\begin{proof}
Similar to that of \fullref{lem:sol31} (see also the end of the
proof of \fullref{lem:sol21}).
\end{proof}

\begin{Rem} \label {rem:repres4nf}
Suppose that $\bar v\ne1$, thus $\bar v=\bar \aa^{2m}\bar \bb^{4n}$ with
$|m|+|n|>0$. Denote $d=\gcd(m,n)$, $c=\aa^{m/d}\bb^{2n/d}$,
$B=\aa\bb\aa\bb^{-1}$. We will later use the following representatives
$(\xalpha_\ell,\ybeta_\ell)\in N\times F_2$ of the solutions
$(\tilde\xalpha,\bar\ybeta)$ of \eqref{eqn4_1nf} from \fullref{lem:sol41nf}.
We define these representatives by the same formulae as in
\fullref{rem:repres3}.

For $\bar v=1$, we will use the following representatives:
$\xalpha_{L,\ell}=1$, $\ybeta_{L,\ell}=\aa^L \bb^{2\ell}$, where $L,\ell\in\ZZ$.
Actually $L,2\ell$ coincide with the exponents in the canonical form of
$\bar\ybeta\in\pi_-$, see~\eqref{eq:canon}.
\end{Rem}

\subsection{The second derived equation in the ``mixed'' cases}
\label{subsec:2der}

In order to find further properties of the solutions
of the equations \eqref{eqn2'}, \eqref{eqn3'} and
\eqref{eqn4'} in the ``mixed'' cases (see \fullref{subsec:appl},
Tables~\ref{tbl1} and~\ref{tbl2},
\fullref{rem:tables}, and \fullref{def:mixed}),
we will construct the {\it second derived equation} \eqref{eqn2_2}
(resp.\ \eqref{eqn3_2} or \eqref{eqn4_2f}, \eqref{eqn4_2nf})
for the equation \eqref{eqn2'} (resp.\ \eqref{eqn3'}, or \eqref{eqn4'}).
More specifically, for every solution
of one of the first derived equations~\eqref{eqn2_1}, \eqref{eqn3_1},
\eqref{eqn4_1f}, and~\eqref{eqn4_1nf}
(see Lemmas~\ref{lem:sol21}, \ref{lem:sol31}, \ref{lem:sol41f}
and~\ref{lem:sol41nf})
we will construct an equation in the free abelian group $Q$, see~\eqref{eq:Q},
which is the quotient
\begin{equation}\label{eq:Qpm}
Q=Q_\varepsilon
= (\ZZ[\pi\setminus\{1\}])/\langle g+g^{-1}\mid g\in\pi\setminus\{1\}\rangle
 , \ \mbox{with}\ \pi=\pi_\varepsilon=F_2/N,
\end{equation}
of the free abelian group $\ZZ[\pi\setminus\{1\}]$ by the system of
relations $g\sim-g^{-1}$, $g\in\pi\setminus\{1\}$, where
$N:=\llangle \aa\bb\aa^{-\varepsilon}\bb^{-1}\rrangle $. (This quotient is
isomorphic to $[N,N]/[F_2,[N,N]]$, see \fullref{pro:QH2}.)
Consequently, we will obtain (in \fullref{thm:second}) an
equation, which we will call the {\it second derived equation}, in
two unknown ``polynomials'' $X\in Q$, $Y\in \ZZ[\pi]$, and some
integer unknowns which enumerate the solutions of the first derived
equation.

As an application of the second derived equation, we will obtain
the non-existence results stated in Tables~\ref{tbl3} and~\ref{tbl4}, see~\fullref{sec:Quadtab}.
For this, we will use the following property of the derived equations,
which follows from Theorems~\ref{thm:first} and~\ref{thm:second}:
the non-existence of a solution of either the first or the
second derived equation implies the non-existence of a (faithful or
non-faithful) solution of the corresponding quadratic equation \eqref{eqN}
in $N$.

\subsubsection*{Case of the equation \eqref{eqn2'}}
Here $\delta=1$, $\varepsilon=-1$,
and we may assume that $\vartheta=-1$ and $\bar v=\bar\bb^{2n}\in\pi$,
$n\in\ZZ$, $\pi=\pi_-$, see \fullref{tbl2} and \fullref{rem:tables}.
We consider the following pair of derived equations
(corresponding to non-faithful solutions).
The first derived equation is~\eqref{eqn2_1} in $\ZZ[\pi]$, $\pi=\pi_-$, with
the unknowns $\bar\ybeta\in\pi$ and $\tilde\xalpha\in\ZZ[\pi]$,
see \fullref{lem:sol21}. By this Lemma, the solutions have the form
 $
\bar\ybeta=\bar\ybeta_{L,\ell} = \bar\aa^L \bar\bb^\ell , \
\tilde\xalpha=\tilde\xalpha_{L,\ell}
= \frac{1-\bar \bb^{2n}}{1-\bar \aa^L \bar \bb^\ell}
 $
for all $L,\ell\in\ZZ$ such that
\begin{equation}\label {ell}
\ell\mid n \quad\mbox{if}\quad n\ne0, \quad\mbox{and} \quad \ell \mbox{ is odd}
\end{equation}
(the latter condition corresponds to the fact that a solution to be found is
non-faithful).
Our {\it second derived equation} will be the following equation
in the quotient $Q=Q_-$, see~\eqref{eq:Qpm}:
\begin{equation}
\label{eqn2_2}
p_Q\left( {\displaystyle \frac{1-\bar\bb^{-2n}}{1-\bar\bb^{-\ell}}}
\cdot \varphi^L(Y) \right) = p_Q \left( \varphi^L(V)
 -{\displaystyle\frac{1-\bar\bb^{-2n}}{1-\bar\bb^2} \bar\bb
        \frac{1-\bar\aa^L}{1-\bar\aa}
+ \frac{1-\bar\bb^{-2n}}{1-\bar\bb^{\ell}} } \right) \! , \!\!\!\!\!
\tag{$2_2$}
\end{equation}
where $p_Q\co \ZZ[\pi]\to\ZZ[\pi\setminus\{1\}]\to Q$ is the
projection, $\varphi\in\Aut(F_2)$ denotes the automorphism sending
$\aa\mapsto \aa$, $\bb\mapsto \bb\aa$, and
$B=B_-=\alpha\beta\alpha\beta^{-1}\mapsto B$, as well as the induced
automorphism of $Q$. The parameter $V\in\ZZ[\pi]$ of the
equation~\eqref{eqn2_2} is defined via
\begin{equation}\label{eq:**}
v=v_0 \prod B^{n_i}_{v_i}, \quad V=\sum n_i \bar v_i,
\quad v_0=\bb^{2n},
\end{equation}
see \eqref{eq:Nab}, while 
the unknowns are $(L,\ell,X,Y)$ with
$L,\ell\in\ZZ$ as in~\eqref{ell}, and
$X\in Q$, $Y\in\ZZ[\pi]$. Remark that the unknown $X$ does not contribute to
the equation \eqref{eqn2_2}, thus $X$ can be arbitrary.

In the special case $n=0$,
the second derived equation~\eqref{eqn2_2} has the form
$0=p_Q(V)$ with the unknowns $L,\ell\in\ZZ$, $\ell$ odd, and $X\in Q$,
$Y\in\ZZ[\pi]$. Since no unknown
contributes to this equation, a solution exists if and only if $p_Q(V)=0$,
moreover if $p_Q(V)=0$ then arbitrary values of the unknowns determine a
solution.

\subsubsection*{Case of the equation \eqref{eqn3'}}
Here $\delta=-1$, $\varepsilon=1$, and we may assume that $\vartheta=-1$ and
$\bar v=\bar \aa^{2m}\bar \bb^{2n}\in\pi$, $m,n\in\ZZ$, $\pi=\pi_+$,
see \fullref{tbl2} and \fullref{rem:tables}.
For the equation~\eqref{eqn3'} (it has only non-faithful solutions)
we consider the following pair of derived equations.
The first derived equation is~\eqref{eqn3_1}
in $\ZZ[\pi]$, $\pi=\pi_+$, with
the unknowns $\bar\ybeta\in\pi$ and $\tilde\xalpha\in\ZZ[\pi]$,
see \fullref{lem:sol31}. By this Lemma, for $|m|+|n|>0$ the solutions have
the form
 $
 \bar\ybeta = \bar\ybeta_\ell = \bar c^\ell, \
 \tilde\xalpha = \tilde\xalpha_\ell= \frac{1-\bar c^{2d}}{1+\bar c^\ell}
 $
where $d=\gcd(m,n)$, $\ell\in\ZZ$ such that $\ell\mid d$, and
$c=\aa^{m/d}\bb^{n/d}\in\pi$, while for $m=n=0$ the solutions
have the form
 $
 \bar\ybeta = \bar\ybeta_{L,\ell} = \bar\aa^L\bar\bb^\ell, \
 \tilde\xalpha = \tilde\xalpha_{L,\ell}=0
 $
where $L,\ell\in\ZZ$.
Our {\it second derived equation} will be the following equation
in the quotient
$Q=Q_+$, see~\eqref{eq:Qpm}:
\begin{equation}
\label{eqn3_2}
\left\{ \begin{array}{rcll}
2X
- {\displaystyle p_Q\left( \frac{1-\bar c^{\;-2d}}{1+\bar c^{\;-\ell}} \cdot Y
\right)}
 \!\!\!
 &=& \!\!\! {\displaystyle
 p_Q\left( V +
\frac{1-\bar c^{\;-2d}}{1-\bar c^{2\ell}} \right) }
         & \mbox{if } |m|+|n|>0, \!\!\!\!\!\!\!\! \\
  2X\!\!\!&=&\!\!\!p_Q(V) & \mbox{if } m=n=0, \end{array}\right.
\tag{$3_2$}
\end{equation}
where $p_Q\co \ZZ[\pi]\to\ZZ[\pi\setminus\{1\}]\to Q$ is the
projection. The parameter $V\in\ZZ[\pi]$ of the
equation~\eqref{eqn3_2} is defined similarly to above, with
$\pi=\pi_+$, $B=B_+=\alpha\beta\alpha^{-1}\beta^{-1}$, via
\begin{equation}\label{eq:***}
v=v_0 \prod B^{n_i}_{v_i}, \quad V=\sum n_i \bar v_i,
\quad v_0=\left\{\begin{array}{ll}c^{2d},& |m|+|n|>0,\\1,&m=n=0,\end{array}\right.
\end{equation}
while the unknowns are either $(\ell,X,Y)$ with $\ell\in\ZZ$, $\ell\mid d$,
$X\in Q$, $Y \in \ZZ[\pi]$ if $|m|+|n|>0$, or $(L,\ell,X,Y)$
with $L,\ell\in\ZZ$, $X\in Q$, $Y\in\ZZ[\pi]$ if $m=n=0$.

In the special case $m=n=0$,
the second derived equation~\eqref{eqn3_2} has the form
$2X=p_Q(V)$ with the unknowns $L,\ell\in\ZZ$, $X\in Q$, $Y\in\ZZ[\pi]$. It
admits
a solution if and only if $2\mid p_Q(V)$, moreover for $2\mid p_Q(V)$ the value
of $X$ is uniquily determined, while the unknowns $(L,\ell,Y)$ take arbitrary
values, in order to determine a solution.

\subsubsection*{Case of the equation \eqref{eqn4'}, non-faithful solutions}
Here $\delta=\varepsilon=-1$, and we may assume that $\vartheta=-1$ and
$\bar v=\bar \aa^{2m}\bar \bb^{4n}$, $m,n\in\ZZ$, $\pi=\pi_-$, see
\fullref{tbl2} and
\fullref{rem:tables}. We consider the following pair of derived equations
(corresponding to non-faithful solutions). The first derived equation
is~\eqref{eqn4_1nf} in $\ZZ[\pi]$, $\pi=\pi_-$, with the unknowns
$\bar\ybeta\in\pi$ and $\tilde\xalpha\in\ZZ[\pi]$ such that $w_-(\bar\ybeta)=1$,
see \fullref{lem:sol41nf}. By this Lemma,
for $|m|+|n|>0$ the solutions have the form
 $
\bar\ybeta = \bar\ybeta_\ell = \bar c^\ell, \
\tilde\xalpha = \tilde\xalpha_\ell = \frac{1-\bar c^{2d}}{1+\bar c^\ell}
 $
where $d=\gcd(m,n)$, $\ell\in\ZZ$ such that $\ell\mid d$, and
$c=\aa^{m/d}\bb^{2n/d}\in\pi$, while for $m=n=0$ the solutions
have the form
 $
 \bar\ybeta = \bar\ybeta_{L,\ell} = \bar\aa^L\bar\bb^{2\ell}, \
 \tilde\xalpha = \tilde\xalpha_{L,\ell}=0
 $
where $L,\ell\in\ZZ$.
Our {\it second derived equation} will be the following equation
in the quotient
$Q=Q_-$, see~\eqref{eq:Qpm}:
\begin{equation}
\label{eqn4_2nf}
\left\{ \begin{array}{rcll}
2X - p_Q{\displaystyle\left(
\frac{1-\bar c^{\;-2d}}{1+\bar c^{\;-\ell}} \cdot Y \right)}
 \!\!\!&=&\!\!\!
 p_Q{\displaystyle\left( V +
\frac{1-\bar c^{\;-2d}}{1-\bar c^{2\ell}} \right) }
         & \mbox{if } |m|+|n|>0,\!\!\!\!\!\!\!\\
  2X\!\!\!&=& \!\!\!p_Q(V) & \mbox{if } m=n=0. \end{array}\right.
\tag{$4_2^\nf$}
\end{equation}
Here the projection $p_Q$
and the polynomial $V\in \ZZ[\pi]$ are defined as in~\eqref{eq:***}
with $\pi=\pi_-$, 
$B=B_-=\alpha\beta\alpha\beta^{-1}$,
$c=\aa^{m/d}\bb^{2n/d}$,
while the unknowns are either $(\ell,X,Y)$ with $\ell\in\ZZ$, $\ell\mid d$,
$X\in Q$, $Y \in \ZZ[\pi]$ if $|m|+|n|>0$, or $(L,\ell,X,Y)$ with
$L,\ell\in\ZZ$, $X\in Q$, $Y\in\ZZ[\pi]$ if $m=n=0$.

In the special case $m=n=0$,
the second derived equation~\eqref{eqn4_2nf} has the form $2X=p_Q(V)$ with the
unknowns $L,\ell\in\ZZ$, $X\in Q$, $Y\in\ZZ[\pi]$. As above, it admits
a solution if and only if $2\mid p_Q(V)$, moreover for $2\mid p_Q(V)$ the value
of $X$ is uniquily determined, while the unknowns $(L,\ell,Y)$ take arbitrary
values, in order to determine a solution.

\subsubsection*{Case of the equation \eqref{eqn4'}, faithful solutions}
Here $\delta=\varepsilon=-1$, and we may assume that $\vartheta=-1$ and
$\bar v=\bar \bb^{2n}$, $n\in\ZZ$, $\pi=\pi_-$, see \fullref{tbl1} and
\fullref{rem:tables}. We consider the following pair of derived equations
(corresponding to faithful solutions). The first derived equation
is~\eqref{eqn4_1f}
in $\ZZ[\pi]$,
$\pi=\pi_-$, with the unknowns $\bar\ybeta\in\pi$ and $\tilde\xalpha\in\ZZ[\pi]$
such that $w_-(\bar\ybeta)=-1$, see \fullref{lem:sol41f}.
By this Lemma,
the solutions have the form
 $
\bar\ybeta = \bar\ybeta_{L,\ell} = \bar\aa^L \bar\bb^\ell , \
\tilde\xalpha = \tilde\xalpha_{L,\ell}
= \frac{1-\bar\bb^{2n}}{1+\bar\aa^L\bar\bb^\ell }
 $
where $L,\ell\in\ZZ$ satisfy~\eqref{ell} (the latter condition in~\eqref{ell}
corresponds to the fact that a solution to be found is faithful).
Our {\it second derived equation} will be the following equation
in the quotient $Q=Q_-$, see~\eqref{eq:Qpm}:
\begin{multline}
\label{eqn4_2f}
2\varphi^L(X) - p_Q{\displaystyle\biggl(
\frac{1-\bar\bb^{-2n}}{1-\bar\bb^{-\ell}} \cdot \varphi^L(Y) \biggr)} \\
 = p_Q \biggl( {\displaystyle \varphi^L(V)
 - \frac{1-\bar\bb^{-2n}}{1-\bar\bb^2} \bar\bb \frac{1-\bar\aa^L}{1-\bar\aa}
  + \frac{1-\bar\bb^{-2n}}{1-\bar\bb^{2\ell}}
       }
       \biggr).
\tag{$4_2^\f$}
\end{multline}
Here the projection $p_Q$, the automorphism $\varphi$ of $Q$, and the
polynomial $V\in \ZZ[\pi]$ are defined as in~\eqref{eq:**}, while
the unknowns are $(L,\ell,X,Y)$ with $L,\ell\in\ZZ$ as in~\eqref{ell}, and
$X\in Q$, $Y \in \ZZ[\pi]$.

In the special case $n=0$,
the second derived equation~\eqref{eqn4_2f} has the
form $2X=p_Q(V)$ with the unknowns $L,\ell\in\ZZ$, $\ell$ odd, and $X\in Q$,
$Y\in\ZZ[\pi]$. As above, it admits
a solution if and only if $2\mid p_Q(V)$. Moreover, if $2\mid p_Q(V)$ then the
value of $X$ is uniquily determined, while the unknowns $(L,\ell,Y)$ take
arbitrary values, in order to determine a solution.

\begin{Thm} \label {thm:second}
Under the hypothesis of \fullref{thm:first}, suppose that
$v_0\in F_2$ is the representative of $\bar v\in\pi$, as in~$\eqref{eq:**}$
or~$\eqref{eq:***}$, and $(\xalpha,\ybeta)$ is a solution of one of the
equations~\eqref{eqn2'}, \eqref{eqn3'} or~\eqref{eqn4'} from \fullref{subsec:appl}, in a
``mixed'' case, see \fullref{rem:tables} and Tables~\ref{tbl1}
and~\ref{tbl2}. Let
$(\xalpha_{L,\ell},\ybeta_{L,\ell})$ be
the corresponding representative, given by
Remarks~\ref{rem:repres2}, \ref{rem:repres3}, \ref{rem:repres4f}
and~\ref{rem:repres4nf}, of the solution
$(\tilde\xalpha,\bar\ybeta)\in (\ZZ[\pi])\times\pi$ of the corresponding first
derived equation
\eqref{eqn2_1}, \eqref{eqn3_1}, \eqref{eqn4_1f} or \eqref{eqn4_1nf} (see
Lemmas~\ref{lem:sol21}, \ref{lem:sol31}, \ref{lem:sol41f}
and~\ref{lem:sol41nf})
where the subscript $L$ is not necessarily present. Let
$X\in Q$, $Y,V\in \ZZ[\pi]$
be the images of the elements
$\xalpha_{L,\ell}^{-1}\xalpha\in[N,N]$,
$\ybeta_{L,\ell}^{-1}\ybeta,\ v_0^{-1}v\in N$
under the projections $q_{N_F}\co [N,N]\to Q\approx [N,N]/[F_2,[N,N]]$ and
$q_N\co N\to (\ZZ[\pi],+) \approx N^{\ab}$, respectively:
 $$
X=q_{N_F}(\xalpha_{L,\ell}^{-1}\xalpha)\in Q, \quad
Y=q_N(\ybeta_{L,\ell}^{-1}\ybeta)\in \ZZ[\pi],\quad V=q_N(v_0^{-1}v)
\in \ZZ[\pi],
 $$
where the natural identifications $N^{\ab}\approx(\ZZ[\pi],+)$ and
$[N,N]/[F_2,[N,N]] \approx Q$ are given by~\eqref{eq:shortNab},
\eqref{eq:Nab}, and~\eqref{eq:shortQisom}, \eqref{eq:Qisom}.
Then the quadruple $(L,\ell,X,Y)$ (or the triple $(\ell,X,Y)$,
respectively) satisfies
the corresponding equation
\eqref{eqn2_2}, \eqref{eqn3_2}, \eqref{eqn4_2f} or \eqref{eqn4_2nf}, described above, called the
{\it second derived equation}.
\end{Thm}

\subsection{Derivation of the second derived equation} \label{subsec:der2}

Here we give a proof of \fullref{thm:second}, that is we derive the
equations \eqref{eqn2_2} and~\eqref{eqn3_2} from the equations
\eqref{eqn2'} and~\eqref{eqn3'}, respectively, and the
equations \eqref{eqn4_2f} and~\eqref{eqn4_2nf} from the equation \eqref{eqn4'}, in the ``mixed''
cases, see \fullref{subsec:appl} and \fullref{rem:tables}.

The following three technical Lemmas
will be useful for deriving the second derived
equations~\eqref{eqn2_2} and~\eqref{eqn4_2f} from the equations~\eqref{eqn2'} and~\eqref{eqn4'},
respectively.

\begin{Lem} \label{lem:xyxy}
In the free group $F_2=\langle \aa,\bb\mid\rangle$, put $B=\aa\bb\aa\bb^{-1}$
and denote $B_u=uBu^{-1}$, $u\in F_2$. Then, for any $L\in\ZZ$,
$$
\aa^L\bb\aa^L\bb^{-1}=\left\{ \begin{array}{ll}
B_{\aa^{L-1}}B_{\aa^{L-2}}\dots B_\aa B, & L\geq 0, \\
B_{\aa^{L}}^{-1}B_{\aa^{L+1}}^{-1}\dots B_{\aa^{-1}}^{-1}, & L<0.
 \end{array} \right.
$$
If $N=\llangle B\rrangle $ and $\pi=\pi_-=F_2/N$ then, under the
projection $q_N\co N\to$\break $(\ZZ[\pi],+) \approx N^{\ab}=N/[N,N]$,
see~\eqref{eq:shortNab}, \eqref{eq:Nab}, the element $\aa^L \bb \aa^L
\bb^{-1}$ is mapped to\break $q_N (\aa^L\bb \aa^L\bb^{-1}) =
\frac{1-\bar \aa^L}{1-\bar \aa}$.
\end{Lem}

\begin{proof}
Let us calculate $\aa^L\bb\aa^L\bb^{-1}$. For $L\ge0$ we prove the
formula by induction. For $L=0,1$ the formula is obviously true. From the
formula for $L\ge1$ we get the formula for $L+1$ as follows:
$$
\aa^{L+1}\bb\aa^{L+1}\bb^{-1} = \aa (\aa^L\bb\aa^L\bb^{-1})\aa^{-1}
 \aa\bb\aa\bb^{-1}
= \aa (B_{\aa^{L-1}}B_{\aa^{L-2}}\dots B_\aa B) \aa^{-1} \cdot B
 %\aa\bb\aa\bb^{-1}
$$
$$
= (B_{\aa^L}B_{\aa^{L-1}}\dots B_{\aa^2} B_\aa) B
= B_{\aa^L}B_{\aa^{L-1}}\dots B_\aa B.
$$
Using the above formula, we get the formula for $L<0$:
$$
\aa^L\bb\aa^L\bb^{-1}
= \aa^L (\aa^{-L}\bb\aa^{-L}\bb^{-1})^{-1} \aa^{-L}
= \aa^L (B_{\aa^{-L-1}}B_{\aa^{-L-2}}\dots B_\aa B)^{-1} \aa^{-L}
$$
$$
= \aa^L (B^{-1} B^{-1}_\aa \dots B^{-1}_{\aa^{-L-2}} B^{-1}_{\aa^{-L-1}} )\aa^{-L}
= B^{-1}_{\aa^L} B^{-1}_{\aa^{L+1}} \dots B^{-1}_{\aa^{-2}} B^{-1}_{\aa^{-1}}.
$$
In the abelianised group $N$, which is identified with $(\ZZ[\pi_-],+)$, see
\fullref{pro:Nab}, we have
$$
q_{N}(\aa^L\bb\aa^L\bb^{-1})
= q_{N}(B_{\aa^{L-1}}B_{\aa^{L-2}}\dots B_\aa B)
= \bar \aa^{L-1} + \bar \aa^{L-2} + \ldots + \bar \aa + 1
= \frac{1-\bar\aa^L}{1-\bar \aa}
$$
if $L\ge0$, and
$$
q_{N}(\aa^L\bb\aa^L\bb^{-1})
= q_{N}(B_{\aa^{L}}^{-1}B_{\aa^{L+1}}^{-1}\dots B_{\aa^{-1}}^{-1})
= - \bar \aa^L - \bar \aa^{L+1} - \ldots - \bar \aa^{\;-1}
= \frac{1-\bar \aa^L}{1-\bar \aa}
$$
if $L<0$.
\end{proof}

\begin{Rem} \label{rem:xyxy}
Under the assumptions of \fullref{lem:xyxy}, one can prove the following
generalization of the formulae from this Lemma, for arbitrary $L,\ell\in\ZZ$
where $\ell$ is odd:
$$
\aa^L\bb^\ell \aa^L\bb^{-\ell} {=} \begin{cases}
\prod_{k=1}^{L} \Bigl[ \Bigl(
 \prod_{i=1}^{(\ell-1)/2} B^{-1}_{\aa^{L+1-k}\bb^{\ell-2i}} \Bigr)
 \prod_{j=0}^{(\ell-1)/2}\! B_{\aa^{L-k}\bb^{2j}} \Bigr], & \!\!\!\!\! \ell>0,
 \\
\prod_{k=1}^{L} \Bigl[ \Bigr(
 \prod_{i=0}^{(-\ell-1)/2} B_{\aa^{L+1-k}\bb^{\ell+2i}} \Bigr)
 \prod_{j=1}^{(-\ell-1)/2}\! B^{-1}_{\aa^{L-k}\bb^{-2j}} \Bigr], & \!\!\!\! \ell<0
\end{cases}
$$
if $L\ge 0$, and
$$
\aa^L\bb^\ell \aa^L\bb^{-\ell} = \begin{cases}
\!\prod_{k=L}^{-1} \left[ \left( \prod_{j=(1-\ell)/2}^{0}
B^{-1}_{\aa^{k}\bb^{-2j}} \right)
 \prod_{i=(1-\ell)/2}^{-1}B_{\aa^{k+1}\bb^{\ell+2i}} \right], & \ell >0,
 \smallskip
 \\
\!\prod_{k=L}^{-1} \left[ \left( \prod_{j=(1+\ell)/2}^{-1} B_{\aa^{k}\bb^{2j}} \right)
 \prod_{i=(1+\ell)/2}^{0}B^{-1}_{\aa^{k+1}\bb^{\ell-2i}} \right], & \ell<0
\end{cases}
$$
if $L<0$. Observe that the formulae for $L<0$ can be easily obtained from the
formulae for $L>0$ via the identity $\aa^L\bb^\ell \aa^L\bb^{-\ell} =
\aa^L(\aa^{-L}\bb^\ell \aa^{-L}\bb^{-\ell})^{-1}\aa^{-L}$.
The above formulae (for $L>0$) can be proved either by straightforward
calculations, or geometrically, by identifying the subgroup $N=\llangle B\rrangle $ with the
fundamental group of a suitable covering of the punctured Klein bottle, and
interpreting elements of $N$ as based loops on this covering, considered up to
the based homotopy. In more detail, we consider a punctured Klein bottle
$K^*=K\setminus \Int D$ with base point $P\in\partial D$, where $K$ is the
Klein bottle, and $D\subset K$ a closed disk. We interprete $K$ as the quotient
of the Euclidean plane $\tilde K$ by the free action of the group
$\pi=F_2/N$ on $\tilde K$ by isometries of the plane, in a usual way.
We can also identify $\pi_1(K,P)=\pi$, $\pi_1(K^*,P)=F_2$, and the element
$B=\aa\bb\aa\bb^{-1}\in F_2$ with the homotopy class of the (suitably
oriented) boundary circle $\partial K^*$. Consider the covering $\tilde K^*$
of $K^*$ corresponding to the subgroup $N=\llangle B\rrangle $. It is a punctured plane with
infinitely many punctures, moreover the inclusion $K^*\hookrightarrow K$ lifts
to an inclusion $\tilde K^*\hookrightarrow\tilde K$.
Let us consider a based loop $\gamma$ on $K^*$, whose homotopy class equals
$[\gamma]=\aa^L\bb^\ell \aa^L\bb^{-\ell}\in F_2$. Since $[\gamma]\in N$, this
loop lifts to the covering $\tilde K^*$.
The obtained based loop $\tilde \gamma$ on $\tilde K^*$ can be considered as a
rectangle of ``width'' $L$ and ``height'' $\ell$ on the plane $\tilde K$.
Representing the elements $B_u$, $u\in F_2$, by suitable based loops on
$\tilde K^*$, one can decompose the element $[\tilde \gamma]\in N$ into the
product of $B_u$, $u\in F_2$, in many different ways. One can check that
the above formulae give one of the ways for such a decomposition.
\end{Rem}

\begin{Lem} \label{lem:yc}
Suppose $n,L\in\ZZ$, $n\ne0$, $c_L=\bb\aa^{-L}$, $B=\aa\bb\aa\bb^{-1}$. Then
$$
\bb^{-2n}c_L^{2n}
= \begin{cases}
\prod_{j=0}^{n-1} \bb^{1-2n+2j} (\aa^{-L}\bb\aa^{-L}\bb^{-1}) \bb^{2n-2j-1},
& n>0, \\
\prod_{j=1}^{-n} \bb^{1-2n-2j} (\aa^{-L}\bb\aa^{-L}\bb^{-1})^{-1} \bb^{2n+2j-1},
& n<0 , \end{cases}
 $$
thus the element $\bb^{-2n}c_L^{2n}\in F_2$ belongs to the subgroup $N=\llangle B\rrangle $.
If $\pi=\pi_-=F_2/N$ then, under the projection
$q_N\co N\to (\ZZ[\pi],+) \approx N^{\ab}=N/[N,N]$, see~\eqref{eq:shortNab},
\eqref{eq:Nab},
the element $\bb^{-2n}c_L^{2n}$ is mapped to $q_N (\bb^{-2n}c_L^{2n})
=-\bar\bb \frac{1-\bar \bb^{-2n}}{1-\bar \bb^2}
    \cdot \frac{1-\bar \aa^{\;-L}}{1-\bar \aa}$.
\end{Lem}

\begin{proof}
Suppose $n>0$. Then
\begin{align*}
\bb^{-2n}c_L^{2n}
&= \bb^{1-2n}{\cdot}(\aa^{-L}\bb\aa^{-L}\bb^{-1}){\cdot}
\bb^2(\aa^{-L}\bb\aa^{-L}\bb^{-1})\bb^{-2}{\cdot}
\bb^4(\aa^{-L}\bb\aa^{-L}\bb^{-1})\bb^{-4} \\
&\quad\cdots
\bb^{2n-4}(\aa^{-L}\bb\aa^{-L}\bb^{-1})\bb^{4-2n}\cdot
\bb^{2n-2}(\aa^{-L}\bb\aa^{-L}\bb^{-1})\bb^{2-2n}\cdot \bb^{2n-1} \\
&= \bb^{1-2n}\cdot \left(
\prod_{j=0}^{n-1} \bb^{2j} (\aa^{-L}\bb\aa^{-L}\bb^{-1}) \bb^{-2j}\right)
\cdot \bb^{2n-1} \\
&= \prod_{j=0}^{n-1} \bb^{1-2n+2j} (\aa^{-L}\bb\aa^{-L}\bb^{-1}) \bb^{2n-2j-1}\in N
\end{align*}
by \fullref{lem:xyxy}. In the abelianised group $N$, which is identified
with $(\ZZ[\pi_-],+)$, see \fullref{pro:Nab}, we obtain
$$q_{N}(\bb^{-2n}c_L^{2n})
= \bar\bb^{1-2n}\cdot \sum_{j=0}^{n-1} \bar\bb^{2j}
  q_N(\aa^{-L}\bb\aa^{-L}\bb^{-1})$$
which, by \fullref{lem:xyxy}, equals
$$
\bar\bb^{1-2n}\cdot \sum_{j=0}^{n-1} \bar\bb^{2j} \frac{1-\bar \aa^{\;-L}}{1-\bar \aa}
= \bar\bb^{1-2n}\cdot \frac{1-\bar \bb^{2n}}{1-\bar \bb^2} \cdot
\frac{1-\bar \aa^{\;-L}}{1-\bar \aa}
=-\bar\bb\frac{1-\bar \bb^{-2n}}{1-\bar \bb^2} \cdot \frac{1-\bar \aa^{\;-L}}{1-\bar \aa}.
 $$
Suppose $n<0$. Then we have
\begin{align*}
\bb^{-2n}c_L^{2n} &= \bb^{-2n}(\bb\aa^{-L})^{2n} =
\bb^{-2n}(\aa^L\bb^{-1})^{-2n} \\
&=
\bb^{1-2n}{\cdot}\bb^{-2}(\bb\aa^{L}\bb^{-1}\aa^{L})\bb^2{\cdot}
\bb^{-4}(\bb\aa^{L}\bb^{-1}\aa^{L})\bb^4{\cdot}
\bb^{-6}(\bb\aa^{L}\bb^{-1}\aa^{L})\bb^6 \\
&\quad\cdots
\bb^{2n+2}(\bb\aa^{L}\bb^{-1}\aa^{L})\bb^{-2n-2}\cdot
\bb^{2n}(\bb\aa^{L}\bb^{-1}\aa^{L})\bb^{-2n}\cdot \bb^{2n-1} \\
&= \bb^{1-2n}\cdot \left(
\prod_{j=1}^{-n} \bb^{-2j} (\aa^{-L}\bb\aa^{-L}\bb^{-1})^{-1} \bb^{2j}\right)
\cdot \bb^{2n-1} \\
&=\prod_{j=1}^{-n} \bb^{1-2n-2j} (\aa^{-L}\bb\aa^{-L}\bb^{-1})^{-1}
\bb^{2n+2j-1} \in N
\end{align*}
by \fullref{lem:xyxy}. In the abelianised group $N$, which is identified
with $(\ZZ[\pi_-],+)$, see \fullref{pro:Nab}, we obtain
 $$
q_{N}(\bb^{-2n}c_L^{2n})
= - \bar\bb^{1-2n}\cdot \sum_{j=1}^{-n} \bar\bb^{-2j} q_N(\aa^{\;-L}\bb\aa^{-L}\bb^{-1})
 $$
which, by \fullref{lem:xyxy}, equals
$$- \bar\bb^{1-2n}\cdot \sum_{j=1}^{-n} \bar\bb^{-2j} \frac{1-\bar \aa^{\;-L}}{1-\bar \aa}
= \bar\bb^{1-2n}\cdot \frac{1-\bar \bb^{2n}}{1-\bar \bb^2}
  \cdot \frac{1-\bar \aa^{\;-L}}{1-\bar \aa}
=-\bar\bb \frac{1-\bar \bb^{-2n}}{1-\bar \bb^2} \cdot \frac{1-\bar \aa^{\;-L}}{1-\bar \aa}.
\proved
 $$
\end{proof}

\subsubsection*{Derivation of \eqref{eqn3_2}}
Here $B=\aa\bb\aa^{-1}\bb^{-1}$, $N=\llangle  B\rrangle $. As in \fullref{lem:sol31}, we
assume that $\bar v=\bar\aa^{2m}\bar\bb^{2n}$, $m,n\in\ZZ$.

Suppose $|m|+|n|>0$, thus $v=c^{2d}P_v$ where $c=\aa^{m/d}\bb^{n/d}$,
$d=\gcd(m,n)$, $P_v\in N$, thus
$P_v=\prod B_{v_i}^{n_i}=\prod_{i=1}^r B_{v_i}^{n_i}$, see~\eqref{eq:***}.
It follows from \fullref{lem:sol31} that any
solution $(\xalpha,\ybeta)$ of~\eqref{eqn3'} has the form
$\xalpha=\xalpha_\ell\ttheta$, $\ybeta=\ybeta_\ell\vvarphi$,
for some $\ell\in\ZZ$ with $\ell\mid d$, $\ttheta\in [N,N]$, and $\vvarphi\in N$,
where $\xalpha_\ell,\ybeta_\ell$ are given by \fullref{rem:repres3}.
The equation \eqref{eqn3'} has the form
 $$
\xalpha\ybeta\xalpha\ybeta^{-1} = c^{2d} P_v \ B^{-1} \ P_v^{-1} c^{-2d} \ B.
 $$
Thus, the equation has the following form in the new unknowns
$\ell,\ttheta,\vvarphi$:
\begin{equation}\label{eq:*}
\xalpha_\ell \ttheta \ \ybeta_\ell \vvarphi \ \xalpha_\ell \ttheta \ \vvarphi^{-1} \ybeta_\ell^{-1}
= c^{2d} P_v \ B^{-1} \ P_v^{-1} c^{-2d} \ B.
\end{equation}
We will start by analyzing both sides of this equality modulo
$[F_2,[N,N]]$, and we will complete by using the
presentation~\eqref{eq:Qisom} of $[N,N]/[F_2,[N,N]]$.  We shall represent
elements of $N/[F_2,[N,N]]$ by elements of $N$, identified under the
congruence relation $g_1\equiv g_2$ {\it modulo} $[F_2,[N,N]]$, and
shall write $g_1\equiv g_2$ whenever $g_1g_2^{-1}\in [F_2,[N,N]]$.

The right-hand side of~\eqref{eq:*} modulo $[F_2,[N,N]]$ equals
\begin{align*}
c^{2d} P_v \ B^{-1} \ P_v^{-1} c^{-2d} \ B
&= c^{2d} [P_v, B^{-1}] B^{-1} c^{-2d} B \\
&\equiv c^{2d} B^{-1} c^{-2d} B [P_v, B^{-1}]
= B^{-1}_{c^{2d}} B [P_v, B^{-1}] .
\end{align*}
The left-hand side of~\eqref{eq:*} modulo $[F_2,[N,N]]$ equals
\begin{equation}\label{eq:LHS}
\ttheta^2 \xalpha_\ell \ybeta_\ell \xalpha_\ell \ybeta_\ell^{-1} \cdot
\ybeta_\ell[\xalpha_\ell^{-1},\vvarphi]\ybeta_\ell^{-1}
\equiv
\ttheta^2 \xalpha_\ell \ybeta_\ell \xalpha_\ell \ybeta_\ell^{-1}
[\xalpha_\ell^{-1},\vvarphi],
\end{equation}
since the elements $\ttheta$, $[\xalpha_\ell^{-1},\vvarphi]$ belong to $[N,N]$
and, hence, they commute with any element of $F_2$ in the quotient
$F_2/[F_2,[N,N]]$.
Let us calculate $\xalpha_\ell \ybeta_\ell \xalpha_\ell \ybeta_\ell^{-1}$ in
the quotient $F_2/[F_2,[N,N]]$.
We have, by \fullref{rem:repres3},
\begin{align}
\notag \xalpha_\ell \ybeta_\ell \xalpha_\ell \ybeta_\ell^{-1}
&= \bigl(
B_{c^{2d-2\ell}}B_{c^{2d-4\ell}} \ldots B_{c^{2\ell}}B \cdot
B^{-1}_{c^\ell}B^{-1}_{c^{3\ell}}\ldots B^{-1}_{c^{2d-3\ell}}B^{-1}_{c^{2d-\ell}}
\bigr) \cdot c^\ell \\
\notag &\quad\cdot \bigl(
B_{c^{2d-2\ell}}B_{c^{2d-4\ell}} \ldots B_{c^{2\ell}}B \cdot
B^{-1}_{c^\ell}B^{-1}_{c^{3\ell}}\ldots B^{-1}_{c^{2d-3\ell}}B^{-1}_{c^{2d-\ell}}
\bigr) \cdot c^{-\ell} \\
\notag &= B_{c^{2d-2\ell}}B_{c^{2d-4\ell}} \ldots B_{c^{2\ell}}B
\cdot
B^{-1}_{c^{2\ell}}B^{-1}_{c^{4\ell}}\ldots
  B^{-1}_{c^{2d-2\ell}}B^{-1}_{c^{2d}} \\
\label{eq:LHS+}
&\equiv
B^{-1}_{c^{2d}}B \cdot \prod_{j=1}^{\frac d\ell} [B,B^{-1}_{c^{2j\ell}}]
\qquad \mbox{if} \quad \ell>0; \\
\notag \xalpha_\ell \ybeta_\ell \xalpha_\ell \ybeta_\ell^{-1}
&= \bigl(
B^{-1}_{c^{2d}}B^{-1}_{c^{2d+2\ell}} \dots B^{-1}_{c^{-4\ell}}B^{-1}_{c^{-2\ell}}
\cdot B_{c^{-\ell}}B_{c^{-3\ell}}\ldots B_{c^{2d+3\ell}} B_{c^{2d+\ell}}
\bigr)\cdot c^\ell \\
\notag &\quad\cdot \bigl(
B^{-1}_{c^{2d}}B^{-1}_{c^{2d+2\ell}} \dots B^{-1}_{c^{-4\ell}}B^{-1}_{c^{-2\ell}}
\cdot B_{c^{-\ell}}B_{c^{-3\ell}}\ldots B_{c^{2d+3\ell}} B_{c^{2d+\ell}}
\bigr)\cdot c^{-\ell} \\
\notag &=
B^{-1}_{c^{2d}}B^{-1}_{c^{2d+2\ell}} \dots B^{-1}_{c^{-4\ell}}B^{-1}_{c^{-2\ell}}
\cdot
BB_{c^{-2\ell}}\ldots B_{c^{2d+4\ell}} B_{c^{2d+2\ell}} \\
\label{eq:LHS-}
&\equiv
B^{-1}_{c^{2d}}B \cdot \prod_{j=1}^{-\frac d{\ell}-1} [B^{-1}_{c^{-2j\ell}},B]
\qquad \mbox{if} \quad \ell<0.
\end{align}
Therefore, after cancelling the common factor $B^{-1}_{c^{2d}} B$ from the both
sides, the equation has the following form in the quotient $F_2/[F_2,[N,N]]$:
\begin{align*}
\ttheta^2
\cdot \left(\prod_{j=1}^{\frac d\ell} [B,B^{-1}_{c^{2j\ell}}]\right)
\cdot [\xalpha_\ell^{-1},\vvarphi]
&\equiv
[P_v, B^{-1}] &\mbox{for}\quad \ell>0, \\
\ttheta^2
\cdot \left(\prod_{j=1}^{-\frac d{\ell}-1} [B^{-1}_{c^{-2j\ell}},B]\right)
\cdot [\xalpha_\ell^{-1},\vvarphi]
&\equiv
[P_v, B^{-1}] & \mbox{for}\quad \ell<0.
\end{align*}
Both sides of the latter equation belong to $N_1=[N,N]$.
After identification of $N_1/[F_2,N_1]$ with $Q=(\ZZ[\pi\setminus\{1\}])/\sim$,
see \fullref{pro:QH2}, and denoting
$X=q_{N_F}(\ttheta)\in Q \approx N_1/[F_2,N_1]$,
$Y=q_N(\vvarphi)\in \ZZ[\pi]$, $V=q_N(P_v)\in \ZZ[\pi]$,
we get, using~\eqref{eq:Qisom} and
\fullref{lem:sol31}, the equation
\begin{equation}\label{eq:32}
\left\{ \begin{array}{rl}
2X + {\displaystyle p_Q\left( \sum_{j=1}^{\frac d\ell} \bar c^{\;-2j\ell}
- \frac {1-\bar c^{\;-2d}}{1+\bar c^{\;-\ell}} \cdot Y \right)}
= p_Q(V) , & \ell>0, \\
2X + {\displaystyle p_Q\left( - \sum_{j=0}^{-\frac d\ell-1} \bar c^{2j\ell}
- \frac {1-\bar c^{\;-2d}}{1+\bar c^{\;-\ell}} \cdot Y \right)}
= p_Q(V) , & \ell<0, \end{array}\right.
\end{equation}
which coincides with the desired equation~\eqref{eqn3_2} for $|m|+|n|>0$.

Consider the case $m=n=0$, that is $\bar v=1$. We have $v=P_v\in N$ where
$P_v=\prod B_{v_i}^{n_i}=\prod_{i=1}^r B_{v_i}^{n_i}$.
It follows from \fullref{lem:sol31} that any
solution $(\xalpha,\ybeta)$ of~\eqref{eqn3'} has the form
$\xalpha=\xalpha_{L,\ell}\ttheta$, $\ybeta=\ybeta_{L,\ell}\vvarphi$
with $\xalpha_{L,\ell}=1$, $\ybeta_{L,\ell}=\aa^L\bb^\ell$,
for some $L,\ell\in\ZZ$, $\ttheta\in [N,N]$, and $\vvarphi\in N$, see
\fullref{rem:repres3}.
Similarly to above, we obtain the equation~\eqref{eq:*} where
$\xalpha_\ell,\ybeta_\ell,c^{2d}$ are replaced by
$\xalpha_{L,\ell}=1$, $\ybeta_{L,\ell}$, 1, respectively.
It follows from $\xalpha_{L,\ell}=1$ that
$\xalpha_{L,\ell} \ybeta_{L,\ell} \xalpha_{L,\ell} \ybeta_{L,\ell}^{-1}=1$
and $[\xalpha_{L,\ell}^{-1},\vvarphi]=1$,
hence the left-hand side modulo $[F_2,[N,N]]$ equals $\ttheta^2$.
As above, the right-hand side modulo $[F_2,[N,N]]$ equals $[P_v, B^{-1}]$.
After identification of $N_1/[F_2,N_1]$ with $Q=\ZZ[\pi\setminus\{1\}]/\sim$,
and denoting $X=q_{N_F}(\ttheta)\in Q \approx N_1/[F_2,N_1]$,
$Y=q_N(\vvarphi)\in \ZZ[\pi]$, $V=q_N(P_v)\in \ZZ[\pi]$,
we get, using~\eqref{eq:Qisom},
the desired equation
 $$
2X = p_Q(V).
 $$
This finishes the derivation of~\eqref{eqn3_2} from~\eqref{eqn3'}.

\subsubsection*{Derivation of \eqref{eqn4_2nf}}
Here $B=\aa\bb\aa\bb^{-1}$, $N=\llangle  B\rrangle $. As in \fullref{lem:sol41nf}, we
assume that $\bar v=\bar\aa^{2m}\bar\bb^{4n}$, $m,n\in\ZZ$.

Suppose $|m|+|n|>0$, thus $v=c^{2d}P_v$ where $c=\aa^{m/d}\bb^{2n/d}$,
$d=\gcd(m,n)$, $P_v\in N$, thus
$P_v=\prod B_{v_i}^{n_i}=\prod_{i=1}^r B_{v_i}^{n_i}$.
It follows from \fullref{lem:sol41nf} that any non-faithful solution
$(\xalpha,\ybeta)$ of~\eqref{eqn4'} has the form
$\xalpha=\xalpha_\ell\ttheta$, $\ybeta=\ybeta_\ell\vvarphi$
for some $\ell\in\ZZ$ with $\ell\mid d$, $\ttheta\in [N,N]$, and $\vvarphi\in N$,
where $\xalpha_\ell,\ybeta_\ell$ are given by \fullref{rem:repres4nf}.

The rest of the derivation is similar to that of~\eqref{eqn3_2}.

\subsubsection*{Derivation of \eqref{eqn2_2}}
Here $B=\aa\bb\aa\bb^{-1}$, $N=\llangle  B\rrangle $. As in \fullref{lem:sol21}, we
assume that
$v=\bb^{2n}P_v$, where $P_v=\prod B_{v_i}^{n_i}$, $n,n_i\in\ZZ$, $v_i\in F_2$.
By this Lemma,
any non-faithful solution $(\xalpha,\ybeta)$ of~\eqref{eqn2'}
has the form $\xalpha=\xalpha_{L,\ell}\ttheta$, $\ybeta=\ybeta_{L,\ell}\vvarphi$
for $L,\ell\in\ZZ$ as in~\eqref{ell}, $\ttheta\in [N,N]$, and $\vvarphi\in N$,
where $\xalpha_{L,\ell},\ybeta_{L,\ell}$ are given by
\fullref{rem:repres2}. The equation~\eqref{eqn2'} has the form
 $$
\xalpha\ybeta\xalpha^{-1}\ybeta^{-1}
= \bb^{2n} P_v \ B^{-1} \ P_v^{-1} \bb^{-2n} \ B.
 $$
Thus, the equation has the following form in the new unknowns
$L,\ell,\ttheta,\vvarphi$:
 $$
\xalpha_{L,\ell} \ttheta \ \ybeta_{L,\ell} \vvarphi \
\ttheta^{-1} \xalpha_{L,\ell}^{-1} \ \vvarphi^{-1} \ybeta_{L,\ell}^{-1}
= \bb^{2n} P_v \ B^{-1} \ P_v^{-1} \bb^{-2n} \ B.
 $$
As above, we will analyze both sides of this equality modulo $[F_2,[N,N]]$, and
will write $g_1\equiv g_2$ whenever $g_1g_2^{-1}\in [F_2,[N,N]]$.

The right-hand side modulo $[F_2,[N,N]]$ equals
\begin{align}
\notag \bb^{2n} P_v \ B^{-1} \ P_v^{-1} \bb^{-2n} \ B
&= \bb^{2n} [P_v, B^{-1}] B^{-1} \bb^{-2n} B \\
\label{eq:RHS}
\equiv \bb^{2n} B^{-1} \bb^{-2n} B [P_v, B^{-1}]
&= B^{-1}_{\bb^{2n}} B [P_v, B^{-1}] .
\end{align}
The left-hand side modulo $[F_2,[N,N]]$ equals
$$
[\xalpha_{L,\ell}, \ybeta_{L,\ell}] \cdot
\ybeta_{L,\ell}[\xalpha_{L,\ell},\vvarphi]\ybeta_{L,\ell}^{-1}
\equiv
[\xalpha_{L,\ell}, \ybeta_{L,\ell}] \cdot [\xalpha_{L,\ell},\vvarphi],
$$
since the elements $\ttheta$, $[\xalpha_{L,\ell},\vvarphi]$ belong to $[N,N]$
and, hence, they commute with any element of $F_2$ in the quotient
$F_2/[F_2,[N,N]]$.
Let us calculate $[\xalpha_{L,\ell}, \ybeta_{L,\ell}]$ in $F_2/[[N,N],F_2]$.
Denote $c_L=\bb\aa^{-L}$, thus
$\tilde\xalpha_{L,\ell}=\smash{\frac{1-\bar c_L^{2n}}{1-\bar c_L^\ell}}$
and $\bar\ybeta_{L,\ell}=\bar c_L^\ell$. For $n/\ell\ge0$
we have, by \fullref{rem:repres2},
\begin{align*}
[\xalpha_{L,\ell}, \ybeta_{L,\ell}]
&= \bigl(
B_{c_L^{2n-\ell}}B_{c_L^{2n-2\ell}} \ldots B_{c_L^\ell}B
\bigr)\cdot c_L^\ell \cdot \bigl(
B^{-1}B^{-1}_{c_L^\ell}\ldots B^{-1}_{c_L^{2n-2\ell}}B^{-1}_{c_L^{2n-\ell}}
\bigr) \cdot c_L^{-\ell} \\
&= B_{c_L^{2n-\ell}}B_{c_L^{2n-2\ell}} \ldots B_{c_L^\ell}B
\cdot
B^{-1}_{c_L^\ell}B^{-1}_{c_L^{2\ell}}\ldots
B^{-1}_{c_L^{2n-\ell}}B^{-1}_{c_L^{2n}} \\
&\equiv
B^{-1}_{c_L^{2n}}B \cdot \prod_{j=1}^{2\frac n\ell} [B,B^{-1}_{c_L^{j\ell}}] ,
\end{align*}
while for $n/\ell<0$ we have, by \fullref{rem:repres2},
\begin{align*}
[\xalpha_{L,\ell}, \ybeta_{L,\ell}]
&= \bigl(
B^{-1}_{c_L^{2n}}B^{-1}_{c_L^{2n+\ell}} \ldots B^{-1}_{c_L^{-2\ell}}B^{-1}_{c_L^{-\ell}}
\bigr)\cdot c_L^\ell \cdot \bigl(
B_{c_L^{-\ell}}B_{c_L^{-2\ell}}\ldots B_{c_L^{2n+\ell}} B_{c_L^{2n}}
\bigr) \cdot c_L^{-\ell} \\
&= B^{-1}_{c_L^{2n}}B^{-1}_{c_L^{2n+\ell}} \ldots
B^{-1}_{c_L^{-2\ell}}B^{-1}_{c_L^{-\ell}}
\cdot
BB_{c_L^{-\ell}}\ldots B_{c_L^{2n+2\ell}} B_{c_L^{2n+\ell}} \\
&\equiv
B^{-1}_{c_L^{2n}}B \cdot\prod_{j=1}^{-2\frac n\ell-1} [B^{-1}_{c_L^{-j\ell}},B].
\end{align*}
Therefore, after multiplying the both sides by $B^{-1}B_{c_L^{2n}}$,
the equation has the following form in the quotient $F_2/[F_2,[N,N]]$:
\begin{align*}
\Biggl(\prod_{j=1}^{2\frac n\ell} [B,B^{-1}_{c_L^{j\ell}}] \Biggr)
\cdot [\xalpha_{L,\ell},\vvarphi]
&\equiv
B^{-1}_{\bb^{2n}} B_{c_L^{2n}} \cdot [P_v, B^{-1}]
& \mbox{for}\quad n/\ell\ge0, \\
\Biggl(\prod_{j=1}^{-2\frac n\ell-1} [B^{-1}_{c_L^{-j\ell}},B]\Biggr)
\cdot [\xalpha_{L,\ell},\vvarphi]
&\equiv
B^{-1}_{\bb^{2n}} B_{c_L^{2n}} \cdot [P_v, B^{-1}]
& \mbox{for}\quad n/\ell<0.
\end{align*}
Observe that
\begin{equation}\label{eq:BB}
B^{-1}_{\bb^{2n}} B_{c_L^{2n}} = \bb^{2n}[B^{-1},\bb^{-2n}c_L^{2n}] \bb^{-2n}
\equiv [B^{-1},\bb^{-2n}c_L^{2n}]\in [N,N],
\end{equation}
due to \fullref{lem:yc}.
In particular, both sides of the obtained equation belong to $N_1=[N,N]$.
After identification of $N_1/[F_2,N_1]$ with $Q=\ZZ[\pi\setminus\{1\}]/\sim$,
see~\eqref{eq:shortQisom}, \eqref{eq:Qisom},
and denoting $X=q_{N_F}(\ttheta)\in Q \approx N_1/[F_2,N_1]$,
$Y=q_N(\vvarphi)\in \ZZ[\pi]$, $V=q_N(P_v)\in \ZZ[\pi]$,
we get, using \fullref{lem:yc} and~\eqref{eq:Qisom}, the equation
\begin{equation}\label{eq:22}
{\displaystyle p_Q \left( - \frac {1-\bar c_L^{\;-2n}}{1-\bar c_L^{\ell}}
+ \frac {1-\bar c_L^{\;-2n}}{1-\bar c_L^{\;-\ell}} \cdot Y \right)}
= {\displaystyle p_Q\left( V
+ \bar\bb\frac{1-\bar\bb^{-2n}}{1-\bar\bb^2}
\cdot \frac{1-\bar\aa^{\;-L}}{1-\bar\aa}\right)}.
\end{equation}
Since the automorphism $\varphi^L$ sends
$\bar c_L=\bar\aa^L\bar\bb\mapsto \bar\bb$,
$\bar\bb \mapsto \bar\bb \bar\aa^L$,
and leaves fixed $\bar\aa$ and $\bar\bb^2$, we obtain, after
applying the automorphism $\varphi^L$ to both sides of the latter equation, the
desired equation~\eqref{eqn2_2}.

\subsubsection*{Derivation of \eqref{eqn4_2f}}
Here $B=\aa\bb\aa\bb^{-1}$, $N=\llangle  B\rrangle $. As in \fullref{lem:sol41f}, we
assume that
$v=\bb^{2n}P_v$, where $P_v=\prod B_{v_i}^{n_i}$, $n\in\ZZ$. By this Lemma,
any faithful solution $(\xalpha,\ybeta)$ of~\eqref{eqn4'}
has the form $\xalpha=\xalpha_{L,\ell}\ttheta$, $\ybeta=\ybeta_{L,\ell}\vvarphi$
for $L,\ell\in\ZZ$ as in~\eqref{ell}, $\ttheta\in [N,N]$, and $\vvarphi\in N$,
where $\xalpha_{L,\ell},\ybeta_{L,\ell}$ are given by
\fullref{rem:repres4f}. The equation~\eqref{eqn4'} has the left-hand side
similar to that of~\eqref{eqn3'} and the right-hand side as in~\eqref{eqn2'}:
 $$
\xalpha\ybeta\xalpha\ybeta^{-1}
= \bb^{2n} P_v \ B^{-1} \ P_v^{-1} \bb^{-2n} \ B.
 $$
Thus, the equation has the following form in the new unknowns
$L,\ell,\ttheta,\vvarphi$:
 $$
\xalpha_{L,\ell} \ttheta \ \ybeta_{L,\ell} \vvarphi \
\xalpha_{L,\ell} \ttheta \ \vvarphi^{-1} \ybeta_{L,\ell}^{-1}
= \bb^{2n} P_v \ B^{-1} \ P_v^{-1} \bb^{-2n} \ B.
 $$
As above, we will analyze both sides of this equality modulo $[F_2,[N,N]]$, and
will write $g_1\equiv g_2$ whenever $g_1g_2^{-1}\in [F_2,[N,N]]$.

As in~\eqref{eq:RHS}, the right-hand side modulo $[F_2,[N,N]]$
equals $B^{-1}_{\bb^{2n}} B [P_v, B^{-1}]$. Similarly to~\eqref{eq:LHS}, one shows that the left-hand side modulo $[F_2,[N,N]]$
is equal to $\ttheta^2 \xalpha_{L,\ell} \ybeta_{L,\ell} \xalpha_{L,\ell}
\ybeta_{L,\ell}^{-1} [\xalpha_{L,\ell}^{-1},\vvarphi]$. Moreover,
using \fullref{rem:repres4f}, we have, similarly to~\eqref{eq:LHS+}
and~\eqref{eq:LHS-},
 $$
\xalpha_{L,\ell} \ybeta_{L,\ell} \xalpha_{L,\ell} \ybeta_{L,\ell}^{-1}
\equiv \begin{cases}
B^{-1}_{c_L^{2n}}B \cdot \prod_{j=1}^{\frac n\ell} [B,B^{-1}_{c_L^{2j\ell}}],
& n/\ell\ge0, \\ \displaystyle
B^{-1}_{c_L^{2n}}B \cdot\prod_{j=1}^{-\frac n\ell-1} [B^{-1}_{c_L^{-2j\ell}},B],
& n/\ell<0. \end{cases}
 $$
Therefore, after multiplying the both sides by $B^{-1}B_{c_L^{2n}}$,
the equation has the following form in the quotient $F_2/[F_2,[N,N]]$:
\begin{align*}
\ttheta^2 \cdot
\Biggl(\prod_{j=1}^{\frac n\ell} [B,B^{-1}_{c_L^{2j\ell}}]\Biggr)
\cdot [\xalpha_{L,\ell}^{-1},\vvarphi]
&\equiv
B^{-1}_{\bb^{2n}} B_{c_L^{2n}} \cdot [P_v, B^{-1}]
\qquad\mbox{for}\quad n/\ell\ge0, \\
\ttheta^2 \cdot
\Biggl(\prod_{j=1}^{-\frac n\ell-1} [B^{-1}_{c_L^{-2j\ell}},B]\Biggr)
\cdot [\xalpha_{L,\ell}^{-1},\vvarphi]
&\equiv
B^{-1}_{\bb^{2n}} B_{c_L^{2n}} \cdot [P_v, B^{-1}]
\qquad\mbox{for}\quad n/\ell<0.
\end{align*}
By~\eqref{eq:BB}, both sides of the obtained equation belong to $N_1=[N,N]$.
After identification of $N_1/[F_2,N_1]$ with $Q=(\ZZ[\pi\setminus\{1\}])/\sim$,
see \fullref{pro:QH2}, and denoting
$X=q_{N_F}(\ttheta)\in Q \approx N_1/[F_2,N_1]$,
$Y=q_N(\vvarphi)\in \ZZ[\pi]$, $V=q_N(P_v)\in \ZZ[\pi]$,
we get, similarly to~\eqref{eq:32} for the left-hand side, and
to~\eqref{eq:22} for the right-hand side, the equation
 $$
2X - {\displaystyle p_Q \left( \frac {1-\bar c_L^{\;-2n}}{1-\bar c_L^{2\ell}}
+ \frac {1-\bar c_L^{\;-2n}}{1+\bar c_L^{\;-\ell}} \cdot Y \right)}
= {\displaystyle p_Q\left( V
+ \bar\bb\frac{1-\bar\bb^{-2n}}{1-\bar\bb^2}
   \cdot \frac{1-\bar\aa^{\;-L}}{1-\bar\aa}\right)}.
 $$
Since
the automorphism $\varphi^L$ sends $\bar c_L=\bar\aa^L\bar\bb \mapsto \bar\bb$,
$\bar\bb \mapsto \bar\bb \bar\aa^L$,
and leaves fixed $\bar\aa$ and $\bar\bb^2$, we obtain, after applying the automorphism
$\varphi^L$ to both sides of the latter equation, the desired
equation~\eqref{eqn4_2f}.

This finishes the proof of \fullref{thm:second}.

\section{Solutions of the second derived equations} \label{sec:quad4}

In this section we give a necessary and sufficient condition for each of
the second derived equations \eqref{eqn2_2}, \eqref{eqn3_2},
\eqref{eqn4_2f} and \eqref{eqn4_2nf},
see \fullref{subsec:2der}, to have a solution. As a consequence, we will
describe, in each of the mixed cases,
many infinite families of $v$'s for which the equation~\eqref{eqN} has no
solution, see \fullref{rem:tables} and Tables~\ref{tbl3} and~\ref{tbl4}.
Unfortunately it is not true that if the second derived equation has a
solution then the original equation also has a solution,
see \fullref{ex:Wicks}. As we noticed in \fullref{rem:tables},
for a given $\bar v\in\pi$ which corresponds to a mixed case, it is not an
easy task to classify all the elements in $p^{-1}(\bar v)$ with respect to the
property that the corresponding equation~\eqref{eqN} has a solution or has no
solution.
In fact we do not know $\bar v$ for which the answer is completely known.

In the following three assertions, we list some identities in the quotient
$Q=\ZZ[\pi\setminus\{1\}]/\sim$, see~\eqref{eq:Qpm} and~\eqref{eq:Qgeneral},
which will be used for solving the second derived equations
\eqref{eqn2_2}, \eqref{eqn3_2}, \eqref{eqn4_2f} and \eqref{eqn4_2nf}.

As above, $\pi=\pi_\pm$ denotes the group
$\pi_\varepsilon=\langle \aa,\bb \mid \aa\bb\aa^{-\varepsilon}\bb^{-1}\rangle$,
$\varepsilon\in\{1,-1\}$, and $\bar u\in\pi$ denotes the class of an element
$u\in F_2=\langle \aa,\bb \mid \rangle$ in $\pi$. Consider
the natural projection $p_{Q}\co \ZZ[\pi]\to Q$,
see~\eqref{eq:Q}. It has the kernel
\begin{equation}\label{eq:K}
K=\ker p_{Q}
=\ZZ[\{1\}]\oplus\langle \{g+g^{-1}\mid g\in\pi\setminus\{1\}\}\rangle,
\end{equation}
where $\langle S \rangle$ denotes the minimal abelian subgroup of
$(\ZZ[\pi],+)$ containing a subset $S\subset\ZZ[\pi]$.
We will represent elements of
$Q$ by elements of $\ZZ[\pi]$, identified under the congruence
relation $X_1\equiv X_2$ {\it modulo} $K$, and shall write
$X_1\equiv X_2$ whenever $X_1-X_2\in K$.

\begin{Lem} \label{lem:ab}
For any $x\in\pi$, $k\in\ZZ$, the following congruences in $\ZZ[\pi]$ hold
modulo $K$:
 $$
\mbox{\rm (a)}\ \ \frac {1-x^{2k}}{1-x} x^{1-k} \equiv x^k, \qquad
\mbox{\rm (b)}\ \ \frac {1-x^{2k}}{1-x^2} x^{1-k} \equiv 0, \qquad
\mbox{\rm (c)}\ \ \frac {1-x^{2k}}{1-x^2} x^{-k} \equiv x^{-k}.
 $$
\end{Lem}

\begin{proof}
(a)\qua The difference of the left-hand side and the right-hand side equals
\begin{align}
\notag \frac {x^{1-k}-x^k}{1-x}
&= \frac {x^{1-k}-x}{1-x} + \frac {x-1}{1-x} + \frac {1-x^k}{1-x}
= \frac {x^{-k}-1}{x^{-1}-1} - 1 + \frac {1-x^k}{1-x}
\equiv 0; \\
\tag{b}  \frac{1-x^{2k}}{1-x^2} x^{1-k}
&= \begin{cases}
x^{1-k}+x^{3-k}+\ldots+x^{k-3}+x^{k-1}\equiv 0, & k>0, \\
0, & k=0,\\
-x^{-1-k}-x^{-3-k}-\ldots-x^{k+3}-x^{k+1}\equiv 0, & k<0;
 \end{cases} \\
\tag{c}  \frac {1-x^{2k}}{1-x^2} x^{-k}
&= \begin{cases}
x^{-k}+x^{2-k}+\ldots+x^{k-4}+x^{k-2}\equiv x^{-k}, & k>0, \\
0, & k=0,\\
-x^{-2-k}-x^{-4-k}-\ldots-x^{k+2}-x^{k}\equiv x^{-k}, & k<0.
 \end{cases}
\end{align}
This completes the proof.
\end{proof}

\begin{Cor} \label{cor:ab}
For any $x\in\pi=\pi_-$ and $n,L,\ell,k,m\in\ZZ$ with $\ell\mid n$ and $\ell$
odd, the following congruences in $\ZZ[\pi_-]$ hold modulo $K$:
\begin{align}
\tag{a} \bar\bb^n
&\equiv \frac {1-\bar\bb^{2n}}{1-\bar\aa^L\bar\bb^\ell} \bar\aa^L\bar\bb^{\ell-n}
\quad \mbox{if $n$ is even}, \\
\tag{b} \frac {1-x^{2k}}{1-x^2} x^{2m}
&\equiv \frac {1-x^{2k}}{1+x} \cdot \frac {x^{2m}-x^{1-k}}{1-x}
\equiv \frac {1-x^{2k}}{1-x} \cdot \frac {x^{2m}+(-1)^kx^{1-k}}{1+x} , \\
\tag{c}\frac {1-\bar\bb^{2n}}{1-\bar\bb^{2\ell}} \bar\bb^{2k\ell}\bar\aa^m
&\equiv
\frac {1-\bar\bb^{2n}}{1-\bar\aa^L\bar\bb^\ell} \cdot
\frac {\bar\bb^{2k\ell}+\bar\aa^L\bar\bb^{\ell-n}}{1+\bar\aa^L\bar\bb^\ell}
\bar\aa^m \quad \mbox{if $n$ is even}.
\end{align}
\end{Cor}

\begin{Lem} \label{lem:S2}
For any $n,L,\ell\in\ZZ$ with $\ell\mid n$ and $\ell$ odd, there exists
$Z_1\in\ZZ[\pi_-]$ satisfying the following congruence in
$\ZZ[\pi_-]$ modulo $K$, for any $m\in\ZZ$:
$$\frac {1-\bar\bb^{-2n}}{1-\bar\bb^2} \bar\bb \bar\aa^m
\equiv
(1-\bar\bb^{2n}) \cdot Z_1 \cdot \bar\aa^m
+ \begin{cases}
0, & n \mbox{ even}, \\
\bar\bb^n \bar\aa^m
, & n \mbox{ odd}. \end{cases}
 $$
\end{Lem}

\begin{proof}
If $n$ is even, we put
$Z_1:=-\frac {1-\bar\bb^n}{1-\bar\bb^2} \bar\bb^{1-2n}$; then
 $$
(1-\bar\bb^{2n}) \cdot Z_1 \cdot \bar\aa^m
= \frac {1-\bar\bb^{-2n}}{1-\bar\bb^2} (1-\bar\bb^n)\bar\bb \bar\aa^m
\equiv
\frac {1-\bar\bb^{-2n}}{1-\bar\bb^2} \bar\bb \bar\aa^m,
 $$
where the latter congruence is due to \fullref{lem:ab}(b).
If $n$ is odd, we put
$Z_1:= -\frac {1-\bar\bb^{n-1}}{1-\bar\bb^2} \bar\bb^{1-2n}$; then
$$
(1-\bar\bb^{2n}) \cdot Z_1 \cdot \bar\aa^m
= \frac {1-\bar\bb^{-2n}}{1-\bar\bb^2} (1-\bar\bb^{n-1})\bar\bb \bar\aa^m
\equiv
\frac {1-\bar\bb^{-2n}}{1-\bar\bb^2} \bar\bb \bar\aa^m - \bar\bb^n \bar\aa^m,
 $$
where the latter congruence is due to \fullref{lem:ab}(c).
\end{proof}

Denote $Q'=Q\otimes\ZZ_2$, and consider the natural projection
\begin{equation}\label{eq:Q'}
p_{Q'}\co \ZZ_2[\pi]\to Q' \approx
(\ZZ_2[\pi\setminus\{1\}])/\langle g+g^{-1}\mid g\in\pi\setminus\{1\}\rangle,
\end{equation}
compare~\eqref{eq:Q}.
In this section, we will only consider the unsolved case $\bar v\ne 1$.

\subsubsection*{Case of the equations \eqref{eqn3_2} and \eqref{eqn4_2nf}}
Observe that these
equations have similar form, where \eqref{eqn3_2} is in $Q_+$, while
\eqref{eqn4_2nf} is in
$Q_-$, see~\fullref{subsec:2der}.

More specifically, for the equation~\eqref{eqn3_2}, we have $B=[\aa,\bb]$,
$\pi=\pi_+$, $Q=Q_+$,
$\bar v=\bar\aa^{2m}\bar\bb^{2n}=\bar c^{2d}\in\pi$, where
$\bar c=\bar \aa^{m/d}\bar \bb^{n/d}$, $m,n\in\ZZ$, $|m|+|n|>0$, $d=\gcd(m,n)$.
For the equation \eqref{eqn4_2nf}, we have $B=\aa\bb\aa\bb^{-1}$, $\pi=\pi_-$,
$Q=Q_-$, $\bar v=\bar\aa^{2m}\bar\bb^{4n}=\bar c^{2d}\in\pi$, where
$\bar c=\bar \aa^{m/d}\bar \bb^{2n/d}$, $m,n\in\ZZ$, $|m|+|n|>0$, $d=\gcd(m,n)$,
thus $\bar c$ is orientation-preserving, and is not a proper power of an
orientation-preserving element of $\pi=\pi_-$.

Observe that the existence of a solution $(\ell,X,Y)$
of the equation~\eqref{eqn3_2} in $Q=Q_+$ is equivalent to the existence of a solution
$(\ell,Z')$
of the following equation in $Q'=Q\otimes\ZZ_2$, with the same $\ell\mid d$
and $V':=V\mod 2\in \ZZ_2[\pi_+]$, $Z':=Z\mod 2\in\ZZ_2[\pi_+]$ where
$Z=\bar c^{\ell-2d}\cdot Y+\frac {\bar c^{\;-2d}-\bar c^{\ell-d}}{1-\bar c^\ell}$:
\begin{equation}
\label{eqn3_2bar}
p_{Q'} {\displaystyle \left(
\frac{1-\bar c^{2d}}{1+\bar c^{\ell}} \cdot Z' \right)} =
p_{Q'} (V'),
\tag{$\bar 3_2$}
\end{equation}
due to \fullref{cor:ab}(b).

Similarly, the existence of a solution $(\ell,X,Y)$
of the equation~\eqref{eqn4_2nf} in $Q=Q_-$ is equivalent to
the existence of a solution $(\ell,Z')$
of the following equation in $Q'=Q\otimes\ZZ_2$, with the same $\ell\mid d$
and $V':=V\mod 2\in\ZZ_2[\pi_-]$, $Z':=Z\mod 2\in\ZZ_2[\pi_-]$ where
$Z=\bar c^{\ell-2d}\cdot Y+\frac {\bar c^{\;-2d}-\bar c^{\ell-d}}{1-\bar c^\ell}$:
\begin{equation}
\label{eqn4_2nfbar}
p_{Q'} {\displaystyle \left(
\frac{1-\bar c^{2d}}{1+\bar c^{\ell}} \cdot Z' \right)} =
p_{Q'} (V').
\tag{$\bar 4_2^{\;\nf}$}
\end{equation}
In the following \fullref{thm:partit34nf} and
\fullref{pro:orbits34nf}, we will formulate necessary and sufficient
conditions for each of the equations \eqref{eqn3_2bar} and
\eqref{eqn4_2nfbar} to have a
solution, when $\bar v\ne1$.

Denote $\bar u=\bar c^d\in\pi=\pi_\varepsilon$, thus $\bar v=\bar u^2$.
Consider the left actions on $\pi$ of the free groups
$G=\langle t,i \mid \rangle$, $\hat G=\langle \hat t,\hat i \mid \rangle$
of rank 2,
where the actions of the generators $t,i$ and $\hat t,\hat i$ are defined by
\begin{align}\label{eq:G}
t \cdot g &= \bar c g, & i \cdot g &= g^{-1}, & g&\in \pi, \\
\label{eq:hatG}
\hat t \cdot g &= \bar u g=\bar c^d g, & \hat i \cdot g &= g^{-1},
& g&\in \pi.
\end{align}
Clearly, $\hat G$ can be considered as a subgroup of $G$, with the
inclusion map $\hat G\hookrightarrow G$, $\hat t\mapsto t^d$, $\hat i\mapsto i$.
Denote $\L_g:=G\cdot g$ and $\hat \L_g:=\hat G\cdot g$, the orbits of an
element $g\in\pi$ under the actions of $G$ and $\hat G$, respectively.
Clearly $\hat \L_h\subset \L_g$ for any $g\in\pi$, $h\in \L_g$. Define the
{\it $\hat G$--augmentation}
\begin{equation}\label{eq:hatvarepsilong}
\hat\varepsilon_g\co  \ZZ_2[\hat\L_g] \to \ZZ_2, \quad
\sum_{k=1}^r m_k \bar u_k\mapsto \sum_{k=1}^r m_k, \quad
m_k\in\ZZ_2,\ \bar u_k\in \hat\L_g, \ g\in\pi,
\end{equation}
the restriction of the usual augmentation
$\ZZ_2[\pi]\to\ZZ_2$ to the subgroup
$\ZZ_2[\hat\L_g]\subset\ZZ_2[\pi]$.

\begin{Thm} \label {thm:partit34nf}
Suppose $\bar c\in\pi$, $d\in\NN$, $V\in\ZZ[\pi]$ are defined by the element
$v\in F_2$, $\bar v=\bar c^{2d}\ne 1$,
as in~$\eqref{eq:***}$.
Consider the actions~$\eqref{eq:G}$, $\eqref{eq:hatG}$ of the groups
$G,\hat G$ on $\pi$. Each of the equations \eqref{eqn3_2bar},
\eqref{eqn4_2nfbar} has
the following properties:

{\rm (A)}\qua
For every fixed $\ell\mid d$, the corresponding equation with the unknown
$Z'\in\ZZ_2[\pi]$ splits into the system of independent
equations in the subspaces $(\ZZ_2[\L_g\setminus\{1\}])/\sim$ with the
unknowns $Z'_g\in\ZZ_2[\L_g]$, where $g\in\pi$.

{\rm (B)}\qua The following conditions are pairwise equivalent:
\begin{enumerate}
\item[\rm(i)] the equation admits a solution;
\item[\rm(ii)] the equation admits a solution with $\ell=d$;
\item[\rm(iii)] for every $h\in\pi\setminus\hat \L_1$, the projection
$\hat V'_h$ of the element $V':=V\mod 2\in\ZZ_2[\pi]$ to the
subspace $\ZZ_2[\hat \L_h]$ has vanishing $\hat G$--augmentation:
$\hat\varepsilon_h(\hat V'_h)=0$.
\end{enumerate}
\end{Thm}

\begin{proof}
(A)\qua Clearly, the equivalence $g\sim g^{-1}$,
$g\in\pi\setminus\{1\}$, on $\pi\setminus\{1\}$ induces an equivalence
relation on $\L_g\setminus\{1\}$, for each orbit $\L_g$.  Moreover,
two elements of $\ZZ_2[\pi\setminus\{1\}]$ are equivalent if and only
if their projections to each subspace $\ZZ_2[\L_g\setminus\{1\}]$ are
equivalent.  Since $\frac{1-\bar c^{2d}}{1+\bar c^{\ell}} \cdot Z'$
belongs to $\ZZ_2[\L_g]$ whenever $Z'\in \ZZ_2[\L_g]$, the induced
equations in the quotients of $\ZZ_2[\L_g\setminus\{1\}]$ by $\sim$
are pairwise independent (for every fixed $\ell$).

(B)\qua
Consider the natural projection $p_{Q'}\co \ZZ_2[\pi]\to Q'=Q\otimes\ZZ_2$,
see~\eqref{eq:Q'}. It has the kernel
 $$
K'=\ker p_{Q'}
=\ZZ_2[\{1\}]\oplus\langle \{g+g^{-1}\mid g\in\pi\setminus\{1\}\}\rangle,
 $$
where $\langle S \rangle$ denotes the minimal abelian subgroup of
$(\ZZ_2[\pi],+)$ containing a subset $S\subset\ZZ_2[\pi]$,
compare~\eqref{eq:K}. Similarly to \fullref{lem:ab}, \fullref{cor:ab} and \fullref{lem:S2},
we will represent elements of $Q'$ by elements of $\ZZ_2[\pi]$, identified
under the congruence relation $X_1\equiv X_2$ {\it modulo} $K'$, and shall
write $X_1\equiv X_2$ whenever $X_1-X_2\in K'$.

(i)$\implies$(ii)\qua Suppose that $(\ell,Z')$ is a solution. Then
the left-hand side equals
 $$
\frac {1-\bar c^{2d}}{1+\bar c^\ell}\cdot Z'
=\frac {1-\bar c^{2d}}{1+\bar c^d}\cdot \frac {1+\bar c^d}{1+\bar c^\ell} \cdot Z'.
 $$
Since the right-hand sides of~\eqref{eqn3_2bar} and \eqref{eqn4_2nfbar} do not depend
on $\ell$, the pair $(d,\frac {1+\bar c^d}{1+\bar c^\ell}\cdot Z')$ is a
solution.

(ii)$\implies$(iii)\qua Suppose $(d,Z')$ is a solution, thus
 $$
(1-\bar c^d)\cdot Z'\equiv V'.
 $$
It follows that $V'=U'+W'$, where $U'$ is a linear combination of the
elements of the form $(1-\bar c^d)g_1$, $g_1\in\pi$, while $W'\in K$ is a
linear combination of the elements of the form $g_2+g_2^{-1}$ and $g_3$,
$g_2\in\pi\setminus\{1\}$, $g_3=1\in\pi$.

Take any $h\in\pi\setminus\hat \L_1$. It follows that $\hat V'_h$ is a linear
combination of $(1-\bar c^d)g_1$, $g_2+g_2^{-1}$, and $g_3$, where
$g_1\in\hat \L_h$, $g_2\in\hat \L_h\setminus\{1\}$, $g_3=1\in\pi\cap \hat \L_h$.
Since $g_3=1\not\in\hat \L_h$, the coefficient at $g_3$ in this linear
combination vanishes. Therefore the augmentation of this linear combination
vanishes, thus $\hat\varepsilon_h(\hat V'_h)=0$.

(iii)$\implies$(i)\qua Suppose $\hat\varepsilon_h(\hat V'_h)=0$ for any
$h\in\pi\setminus\hat \L_1$.
Since $\hat V'_h\in\ZZ_2[\hat \L_h]$, and $\hat \L_h$ is an orbit with respect
to the action of the group $\hat G$ on $\pi$, it follows from
$\hat\varepsilon_h(\hat V'_h)=0$ that $\hat V'_h$ is a linear
combination of the elements of the form $(1-\bar c^d)g_1$ and $g_2+g_2^{-1}$,
where $g_1\in\hat \L_h$, $g_2\in\hat \L_h\setminus\{1\}$.
Similarly, since one of the elements $\hat V'_1,\hat V'_1+1\in\ZZ_2[\hat \L_1]$
has vanishing $\hat G$--augmentation, it follows that $\hat V'_1$ is a linear
combination
of the elements of the form $(1-\bar c^d)g_1$, $g_2+g_2^{-1}$, and $g_3$,
where $g_1\in\hat \L_1$, $g_2\in\hat \L_1\setminus\{1\}$, $g_3=1\in\hat \L_1$.

This immediately gives $\hat V'_h \equiv (1-\bar c^d)\cdot \hat Z_h'$, for some
$\hat Z_h'\in\ZZ_2[\hat \L_h]$, for every $h\in\pi$.
Since $V'$ equals the sum of $\hat V_h'\in\ZZ_2[\hat \L_h]$ over all
$\hat G$--orbits $\hat \L_h\subset\pi$, we obtain the desired decomposition
$V' \equiv (1-\bar c^d)\cdot Z'$, for some $Z'\in\ZZ_2[\pi]$. Hence
$(d,Z')$ is a solution.
\end{proof}

\begin{Pro} \label {pro:orbits34nf}
Suppose $\bar u\in\pi$, $w_\varepsilon(\bar u)=1$, $\bar v=\bar u^2$, where
$\pi=\pi_\varepsilon=\langle\aa,\bb\mid\aa\bb\aa^{-\varepsilon}\bb^{-1}\rangle$.
Consider the corresponding action~$\eqref{eq:hatG}$ of the group $\hat G$ on
$\pi$. Then the orbits $\hat \L_h$, $h\in\pi$, under this action have the
following form:

{\rm (A)}\qua Suppose $\varepsilon=1$ and $\bar u=\bar\aa^m\bar\bb^n$,
$m,n\in\ZZ$. Then, for $h=\bar\aa^p\bar\bb^q$, $p,q\in\ZZ$, one has
 $$
\hat \L_h
= \{\bar u^k h^{\pm1}\mid k\in \ZZ\}
= \{\bar\aa^{p+km}\bar\bb^{q+kn}\mid k\in \ZZ\} \cup
\{\bar\aa^{\;-p+km}\bar\bb^{-q+kn}\mid k\in \ZZ\} .
 $$
{\rm (B)}\qua Suppose $\varepsilon=-1$, thus $\bar u=\bar\aa^m\bar\bb^{2n}$,
$m,n\in\ZZ$. If $w_-(h)=1$ then $h=\bar\aa^p\bar\bb^{2q}$, for some
$p,q\in\ZZ$, and
$$\hat \L_h
= \{\bar u^k h^{\pm1}\mid k\in \ZZ\}
= \{\bar\aa^{p+km}\bar\bb^{2q+2kn}\mid k\in \ZZ\} \cup
\{\bar\aa^{\;-p+km}\bar\bb^{-2q+2kn}\mid k\in \ZZ\}.$$
If $w_-(h)=-1$ then $h=\bar\aa^p\bar\bb^{2q+1}$, for some $p,q\in\ZZ$, and
\begin{align*}
\hat \L_h
 &= \{\bar\aa^{km}\bar\bb^{(2k+4r)n} h^{\pm1} \mid k,r\in\ZZ\} \\
&= \{\bar\aa^{p+km}\bar\bb^{2q+1+(2k+4r)n}\mid k,r\in \ZZ\}
 \cup \{\bar\aa^{p+km}\bar\bb^{-(2q+1)+(2k+4r)n}\mid k,r\in\ZZ\} ,
\end{align*}
moreover, in the latter case, the set of all such orbits is in one-to-one
correspondence with the set
$\ZZ_{|m|}\oplus\ZZ_{|n|}$,
where one denotes $\ZZ_0=\ZZ$, $\ZZ_1=\{0\}$;
in particular, the number of such orbits is either $|mn|$ if $mn\ne0$, or
infinite if $mn=0$.
\end{Pro}

\begin{proof}
(A)\qua Suppose $h=\bar\aa^p\bar\bb^q$, and denote
$\bar \L_h=\{\bar u^k h^{\pm1}\mid k\in \ZZ\}$. Obviously, $h\in\bar \L_h$, and
$\bar \L_h$ is invariant under the action of $\hat G$ (since $\pi=\pi_+$ is
abelian), hence $\hat \L_h \subset \bar \L_h$. The converse inclusion follows
from the fact that any element of $\bar \L_h$ is obtained from $h$ or $h^{-1}$
by the left multiplication by $\bar u^k$, for some $k\in\ZZ$.
This proves $\hat \L_h=\bar \L_h$.

The equality of the two presentations for the set $\hat \L_h$ follows from
the fact that the group $\pi=\pi_+$ is abelian.

(B)\qua Suppose $w_-(h)=1$, thus $h=\bar\aa^p\bar\bb^{2q}$. Denote
$\bar \L_h=\{\bar u^k h^{\pm1}\mid k\in \ZZ\}$. Obviously, $h\in\bar \L_h$.
Since $w_-(\bar u)=w_-(h)=1$, the elements $\bar u$ and $h$ commute, therefore
$\bar \L_h$ is invariant under the action of $\hat G$, hence
$\hat \L_h \subset \bar \L_h$. The converse inclusion follows from the fact that
any element of $\bar \L_h$ is obtained from $h$ or $h^{-1}$ by the left
multiplication by $\bar u^k$, for some $k\in\ZZ$.
This proves $\hat \L_h=\bar \L_h$.

The equality of the two presentations for the set $\hat \L_h$ follows from the
fact that the subgroup of $\pi=\pi_-$ generated by $\bar\aa,\bar\bb^2$ is
abelian.

Suppose now $w_-(h)=-1$, thus $h=\bar\aa^p\bar\bb^{2q+1}$. Denote
$$
\bar \L_h =
 \{\bar\aa^{p+km}\bar\bb^{2q+1+(2k+4r)n}\mid k,r\in \ZZ\} \cup
 \{\bar\aa^{p+km}\bar\bb^{-(2q+1)+(2k+4r)n}\mid k,r\in \ZZ\} .
 $$
Obviously $h\in\bar \L_h$. Let us show that $\bar \L_h$ is invariant
under the action of $\hat G$. Since $\bar\bb^2$ commutes with any
element of $\pi=\pi_-$, and $\bar\aa^p\bar\bb=\bar\bb\bar\aa^{\;-p}$,
we have, for any $s\in\ZZ$,
\begin{align*}
\bar u^s \cdot \bar\aa^{p+km}\bar\bb^{2q+1+2(k+2r)n}
 &= (\bar\aa^m\bar\bb^{2n})^s \cdot \bar\aa^{p+km}\bar\bb^{2q+1+2(k+2r)n} \\
 &= \bar\aa^{sm}\bar\bb^{2sn} \cdot \bar\aa^{p+km}\bar\bb^{2q+1+2(k+2r)n} \\
 &= \bar\aa^{p+(s+k)m}\bar\bb^{2q+1+2(s+k+2r)n} \in \bar \L_h, \\
(\bar\aa^{p+km} \bar\bb^{2q+1+(2k+4r)n})^{-1}
 &= \bar\aa^{p+km} \bar\bb^{-(2q+1)-(2k+4r)n} \in \bar \L_h,
\end{align*}
and similarly
 $\bar u^s \cdot \bar\aa^{p+km}\bar\bb^{-(2q+1)+2(k+2r)n}\in \bar \L_h$,
 $(\bar\aa^{p+km} \bar\bb^{-(2q+1)+(2k+4r)n})^{-1} \in \bar \L_h$.
Therefore $\hat \L_h \subset \bar \L_h$. The converse inclusion follows by
observing that
 $$
 \bar\aa^{p+km}\bar\bb^{2q+1+(2k+4r)n}
 = \hat t^{k} \cdot \bar\aa^p\bar\bb^{2q+1+4rn}
 = \hat t^{k} \hat i\hat t^{-r}\hat i\hat t^{r}
\cdot \bar\aa^p\bar\bb^{2q+1} \in \hat\L_h,
 $$
therefore any element of $\bar \L_h$ belongs to the orbit $\hat\L_h$ of
$h=\bar\aa^p\bar\bb^{2q+1}$
under the action of $\hat G$. This proves $\hat \L_h\supset\bar \L_h$ and,
hence, $\hat \L_h=\bar \L_h$.

The equality of the two presentations for the set $\hat \L_h$ follows from the
identities
\begin{align*}
 \bar\aa^{km}\bar\bb^{(2k+4r)n} \cdot h 
 &= \bar\aa^{km}\bar\bb^{(2k+4r)n} \cdot \bar\aa^p\bar\bb^{2q+1} \\
 &= \bar\aa^{p+km}\bar\bb^{2q+1+(2k+4r)n}, \\
\text{and}\qquad \bar\aa^{km}\bar\bb^{(2k+4r)n} \cdot h^{-1}
 &= \bar\aa^{km}\bar\bb^{(2k+4r)n} \cdot \bar\bb^{-(2q+1)}\bar\aa^{\;-p} \\
 &= \bar\aa^{p+km}\bar\bb^{-(2q+1)+(2k+4r)n}.
\end{align*}
This completes the proof.
\end{proof}

\subsubsection*{Case of the equations \eqref{eqn2_2} and \eqref{eqn4_2f}}
For each of the equations~\eqref{eqn2_2} and~\eqref{eqn4_2f}, see \fullref{subsec:2der}, we have
$B=\aa\bb\aa\bb^{-1}$, $\pi=\pi_-$, $Q=Q_-$, $\bar v=\bar \bb^{2n}$, $n\in\ZZ$.
Both equations have the unknowns $(L,\ell,X,Y)$ with $L,\ell\in\ZZ$ as
in~\eqref{ell}, $X\in Q$, $Y\in\ZZ[\pi]$.

Observe that the existence of a solution~$(L,\ell,X,Y)$
of the equation~\eqref{eqn2_2} in $Q=Q_-$ is equivalent to the existence of a solution
$(L,\ell,Z)$ of the following equation in $Q$, with the same
$V\in\ZZ[\pi_-]$, $L,\ell\in\ZZ$ satisfying~\eqref{ell}, and with
$Z=\bar c_L^{\ell-2n}Y+\bar c_L^{\;-2n}+C$, 
$\bar c_L=\bar\bb\bar\aa^{\;-L}=\varphi^{-L}(\bar\bb)$,
$C\in\ZZ[\pi_-]$:
\begin{equation}
\label{eqn2_2bar}
p_Q\left( \frac{1-\bar\bb^{2n}}{1-\bar\bb^\ell} \cdot \varphi^L(Z) \right)
=
p_Q(\varphi^L(V) )
+ \begin{cases} 0, & n \mbox{ even}, \\
- p_Q \left( \bar\bb^n \frac{1-\bar\aa^L}{1-\bar\aa} \right),
                               & n \mbox{ odd}, \end{cases}
\tag{$\bar 2_2$}
\end{equation}
where $C$ is determined by \fullref{lem:S2}.
As in~\eqref{eqn2_2}, this equation is equivalent to $0=p_Q(V)$ if $n=0$.

Similarly, the existence of a solution $(L,\ell,X,Y)$ of the
equation~\eqref{eqn4_2f} in $Q=Q_-$ is equivalent to the existence of a solution
$(L,\ell,Z')$ of the following equation in $Q'=Q\otimes\ZZ_2$, with
the same $L,\ell\in\ZZ$ satisfying~\eqref{ell}, with
$V':=V\mod 2\in\ZZ_2[\pi_-]$, $Z':=Z\mod 2\in\ZZ_2[\pi_-]$, and with
$Z=\bar c_L^{\ell-2n}Y
+\frac {\bar c_L^{\;-2n}+(-1)^n\bar c_L^{\ell-n}}{1+\bar c_L^\ell}+C$,
$C\in\ZZ[\pi_-]$ from above:
\begin{equation}
\label{eqn4_2fbar}
p_{Q'} \left(\frac{1-\bar\bb^{2n}}{1-\bar\bb^{\ell}}
\cdot \varphi^L(Z') \right)
=
p_{Q'}(\varphi^L(V') )
+ \left\{ \begin{array}{ll} 0, & n \mbox{ even}, \\
- p_{Q'} \left( \bar\bb^n \frac{1-\bar\aa^L}{1-\bar\aa} \right),
                               & n \mbox{ odd}, \end{array} \right.
\tag{$\bar 4_2^\f$}
\end{equation}
due to \fullref{cor:ab}(b).
This equation is equivalent to $0=p_{Q'}(V')$ if $n=0$.

Below (see \fullref{thm:partit24f} and \fullref{pro:orbits24f}),
we will formulate necessary and sufficient conditions for each of the equations
\eqref{eqn2_2bar}, \eqref{eqn4_2fbar} to have a solution, when $\bar v\ne1$.

From now on, for the remainder of this section, let us fix an integer $n\ne 0$,
and denote by $s$ the exponent of $2$ in the prime factorization of $|n|$; put
$\mu=\frac {n}{|n|}2^s$, $\ell_{\max}=\frac {|n|}{2^s}=\frac {n}{\mu}$, the
greatest odd divisor of $n$.
Consider the left actions on $\pi=\pi_-$ of the groups
\begin{align*}
G_L&:=\bigl\langle t_L,i \mid i^2,\ (t_Lit_L)^2,\ (it_L)^4  \bigr\rangle , \\
\hat G_L&:=\bigl\langle \hat t_L,\hat i \mid \hat i^2,\ (\hat t_L\hat i\hat t_L)^2,\
(\hat i\hat t_L)^4  \bigr\rangle , \\
\ttilde G_L&:=\bigl\langle \ttilde t,\ttilde i,\ttilde j_L
\mid \ttilde i^2,\ \ttilde j_L^2,\ (\ttilde i\ttilde j_L)^2,\
(\ttilde i\ttilde t)^2,\ (\ttilde j_L\ttilde t)^2 \bigr\rangle
\approx \ZZ \rtimes_{\psi_2} (\ZZ_2\oplus\ZZ_2) , \\
&\qquad\psi_2\co \ZZ_2\oplus\ZZ_2\to\Aut(\ZZ), \quad
\psi_2(\ttilde i)(\ttilde t):=\ttilde t^{-1}=:\psi_2(\ttilde j_L)(\ttilde t), \\
\hhat G&:=\bigl\langle \hhat t,\hhat i \mid \hhat i^2,\ (\hhat i\hhat t)^2
\bigr\rangle
\approx \ZZ \rtimes_{\psi_1} \ZZ_2 , \qquad
\psi_1\co \ZZ_2\to\Aut(\ZZ), \quad \psi_1(\hhat i)(\hhat t):=\hhat t^{\;-1},
\end{align*}
where $L\in\ZZ$, and the actions of the generators
$\ttilde t,\ttilde i,\ttilde j_L$, $\hat t_L,\hat i$, and $t_L,i$ are defined by
\begin{align}
\label{eq:G'}
t_L \cdot g &= \bar\aa^L\bar\bb g, & i \cdot g &= g^{-1}, && & g&\in \pi, \\
\label{eq:hatG'}
\hat t_L \cdot g &= \bar\aa^L\bar\bb^{\ell_{\max}} g,
& \hat i \cdot g &= g^{-1}, && & g&\in \pi, \\
\label{eq:ttildeG}
\ttilde t \cdot g &= \bar\bb^{2n} g, & \ttilde i \cdot g &= g^{-1}, &
\ttilde j_L \cdot g 
 &= \bar\aa^L\bar\bb^{\ell_{\max}}(\bar\aa^L\bar\bb^{\ell_{\max}} g)^{-1},
& g&\in \pi.
\end{align}
Clearly, we have the inclusions
$\hhat G\subset \ttilde G_L \hookrightarrow \hat G_L\hookrightarrow G_L$
with $\ttilde t \mapsto \hat t_L^{\;2\mu}$, $\hat t_L \mapsto t_L^{\ell_{\max}}$,
$\ttilde i\mapsto \hat i\mapsto i$, $\ttilde j_L\mapsto \hat t_L\hat i\hat t_L$,
which respect the actions.

This provides the following alternative approach for defining the
groups $\hhat G$, $\ttilde G_L$, $\hat G_L$. We will henceforth
identify these groups with the corresponding subgroups of the group
$G_L$ by denoting
$$\ttilde t=\hat t_L^{2\mu}=t_L^{2n}=:t, \quad
  \hat t_L=t_L^{\ell_{\max}}, \quad
\ttilde i=\hat i=i, \quad \ttilde j_L=\hat t_L i\hat t_L=:j_L.$$
Thus the subgroups $\hhat G\subset \ttilde G_L \subset \hat G_L\subset G_L$
admit the following presentations by means of generators and defining relations:
\begin{align*}
\hat G_L :=&\langle \hat t_L,i \mid i^2,\ (\hat t_Li\hat t_L)^2,\
(i\hat t_L)^4  \rangle , \\
\ttilde G_L :=&\langle t,i,j_L \mid i^2,\ j_L^2,\ (ij_L)^2,\ (it)^2,\ (j_Lt)^2
\rangle \approx \ZZ \rtimes_{\psi_2} (\ZZ_2\oplus\ZZ_2) , \\
\hhat G =&\langle t,i \mid i^2,\ (it)^2 \rangle
\approx \ZZ \rtimes_{\psi_1} \ZZ_2 .
\end{align*}
Observe that the defined in this way group $\hhat G$ depends on $L$. However the
above presentations of the groups by means of generators and defining relations
provide an obvious group isomorphism $G_L\approx G_{L'}$
for $L,L'\in\ZZ$. Although this isomorphism does not respect the actions of
$G_L,G_{L'}$ on $\pi$ if $L\ne L'$ (since these actions determine different
orbits), the induced isomorphism of the corresponding subgroups
$\hhat G\subset G_L$ and $\hhat G\subset G_{L'}$ respects the actions. This
gives the natural identification of different subgroups $\hhat G\subset G_L$,
respecting their actions on $\pi$.

One easily checks that
\begin{align}
\label{eq:actions}
i \cdot (\bar\bb^{2q}\bar\aa^p) &= \bar\bb^{-2q}\bar\aa^{\;-p}, &
j_L \cdot (\bar\bb^{2q}\bar\aa^p) &= \bar\bb^{-2q}\bar\aa^p, \\
\notag
i \cdot (\bar\bb^{2q+1}\bar\aa^p) &= \bar\bb^{-(2q+1)}\bar\aa^p, &
j_L \cdot (\bar\bb^{2q+1}\bar\aa^p) &= \bar\bb^{-(2q+1)}\bar\aa^{\;-p-2L}, \\
\notag
t \cdot (\bar\bb^q\bar\aa^p) &= \bar\bb^{q+2n}\bar\aa^p, &
\hat t_L \cdot (\bar\bb^{2q}\bar\aa^p) &= \bar\bb^{2q+\ell_{\max}}\bar\aa^{p-L},
\\
\notag
t_L \cdot (\bar\bb^{2q}\bar\aa^p) &= \bar\bb^{2q+1}\bar\aa^{p-L}.
\end{align}
Denote $\L_{g,L}:=G_L\cdot g$, $\hat\L_{g,L}:=\hat G_L\cdot g$,
$\ttilde \L_{g,L}:=\ttilde G_L\cdot g$, and $\hhat \L_g:=\hhat G\cdot g$,
the orbits of an element $g\in\pi$ under the actions of
$G_L$, $\hat G_L$, $\ttilde G_L$, and $\hhat G$, respectively. Clearly
$\hhat \L_h\subset \ttilde \L_{g,L}\subset \hat\L_{f,L}\subset \L_{e,L}$
for any $e\in\pi$, $f\in \L_{f,L}$, $g\in\hat\L_{f,L}$, $h\in \ttilde\L_{g,L}$.

An element $g\in\pi$ (together with its orbit $\hhat\L_g$) is called
$\hhat G$--{\it regular} if $g$ has a trivial stabilizer with respect to the
action of $\hhat G$ on $\pi$ (thus $\Stab_{\hhat G}(g)=\{1\}$, so the natural
map $\hhat G\to \hhat G\cdot g$ is bijective). Otherwise $g$ (together with its
orbit $\hhat\L_g$) is called $\hhat G$--{\it singular}.

\begin{Lem} \label{lem:reg}
An element $g\in\pi$ is $\hhat G$--singular if and only if either $w_-(g)=1$ and
$g=\bar\bb^{nk}$ (thus $nk$ is even), or $w_-(g)=-1$ and
$g=\bar\bb^{nk}\bar\aa^m$ (thus $nk$ is odd), for some $k,m\in\ZZ$.
Moreover, the stabilizer $\Stab_{\hhat G}(g)$ of a $\hhat G$--singular element
$g=\bar\bb^{nk}\bar\aa^m$ under the action of $\hhat G$ is the cyclic subgroup
of $\hhat G$ generated by the element $t^ki\in\hhat G$.
Here the element $t^ki$ is conjugate in $\hhat G$ either to the element $i$
if $k$ is even, or to the element $ti$ if $k$ is odd.
\qed
\end{Lem}

In the case of the equation \eqref{eqn4_2fbar}, we
define the augmentations
\begin{equation}\label{eq:varepsilong}
\hhat\varepsilon_g\co \ZZ_2[\hhat\L_g]\to\ZZ_2, \quad
\ttilde\varepsilon_{g,L}\co \ZZ_2[\ttilde\L_{g,L}]\to\ZZ_2, \quad
\hat\varepsilon_{g,L}\co \ZZ_2[\hat\L_{g,L}]\to\ZZ_2,
\end{equation}
called the {\it $\hhat G$--augmentation},
{\it $\ttilde G_L$--augmentation},
and {\it $\hat G_L$--augmentation}, respectively,
as the restrictions of the usual augmentation
$\ZZ_2[\pi]\to\ZZ_2$ to $\ZZ_2[\hhat\L_g]$,
$\ZZ_2[\ttilde\L_{g,L}]$, and $\ZZ_2[\hat\L_{g,L}]$, respectively,
for every $g\in\pi$.

In order to define similar augmentations
in the case of the equation~\eqref{eqn2_2bar},
the following constructions will be useful. Consider the character
 $$
\chi_L \co  G_L \to \ZZ^*=\{1,-1\}, \quad t_L\mapsto -1,\ i\mapsto -1,
 $$
thus $\hat t_L=t_L^{\ell_{\max}}\mapsto -1$, $t=t_L^{2n}\mapsto 1$. Denote
$\chi:=\chi_L|_{\hhat G}$.
For every $\hhat G$--regular element $g\in\pi$, define the {\it $\chi$--twisted
$\hhat G$--augmentation}
\begin{equation}\label{eq:hhatvarepsilong}
\hhat\varepsilon_g\co  \ZZ[\hhat\L_g] \to \ZZ, \quad
r\cdot g \mapsto \chi(r), \quad r\in\hhat G,
\end{equation}
by the linear extension of the latter formula.
The $\chi$--twisted $\hhat G$--augmentation $\hhat\varepsilon_g$ is well-defined
for any $\hhat G$--regular element $g\in\pi$, since the equality
$r_1\cdot g=r_2\cdot g$ implies $r_1^{-1}r_2\in\Stab_{\hhat G}(g)=\{1\}$, hence
$r_1=r_2$. We also have $\hhat\varepsilon_{r\cdot g}=\chi(r)\hhat\varepsilon_g$,
for any $r\in\hhat G$, and for any $\hhat G$--regular element $g\in\pi$.

An element $g\in\pi$ (together with its $\ttilde G_L$--orbit) is called
$\ttilde G_L$--{\it defective}, or simply
{\it defective}, if there exists $r\in\Stab_{\ttilde G_L}(g)$ with
$\chi_L(r)=-1$. In other words,
$\chi_L(\Stab_{\ttilde G_L}(g))=\{1,-1\}$ for defective
$g$, and $\chi_L(\Stab_{\ttilde G_L}(g))=\{1\}$ for non-defective $g$.
For every element $g\in\pi$,
define the {\it $\chi_L$--twisted $\ttilde G_L$--augmentation}
\begin{equation}\label {eq:defec}\hspace{-.5cm}
\ttilde\varepsilon_{g,L}\co  \ZZ[\ttilde\L_{g,L}] \to
 \! \left\{ \begin{array}{ll} \!\! \ZZ, & \!\!\! g\ \mbox{non-defective},\\
                        \!\! \ZZ_2, & \!\!\! g\ \mbox{defective}, \end{array} \right.
\quad
 r\cdot g \mapsto \! \left\{ \begin{array}{r}
 \!\! \chi_L(r)\in\ZZ, \\
 \!\! 1\in\ZZ_2, \end{array} \right.
\ r\in\ttilde G_L, \!\!\!\!\!\!\!\!\!
\end{equation}
by the linear extension of the latter formula.
If $n$ is odd, we similarly define the {\it $\chi_L$--twisted
$\hat G_L$--augmentation}
\begin{equation}\label {eq:defec:hat}\hspace{-.5cm}
 \hat\varepsilon_{g,L}\co  \ZZ[\hat\L_{g,L}] \to
 \! \left\{ \begin{array}{ll} \!\! \ZZ, & \!\!\! g\ \mbox{non-defective},\\
                            \!\! \ZZ_2, & \!\!\! g\ \mbox{defective}, \end{array} \right.
 \quad
 r\cdot g \mapsto \! \left\{ \begin{array}{r}
 \!\! \chi_L(r)\in\ZZ, \\
 \!\! 1\in\ZZ_2, \end{array} \right.
 \ r\in \hat G_L, \!\!\!\!\!\!\!\!\!
\end{equation}
by the linear extension of the latter formula. One easily checks that
\begin{align}
\notag
\hat\varepsilon_{g,L}(\hat V_{g,L})
&=\ttilde\varepsilon_{g,L}(\ttilde V_{g,L})
-\ttilde\varepsilon_{h,L}(\ttilde V_{h,L})
\quad \mbox{where}\quad h:=\bar\aa^L\bar\bb^ng, \ n\mbox{ odd},\\
\label{eq:varepsilon:mod2}
 \!\! \hhat\varepsilon_{g,L}(\hhat V_{g,L})
 &=\begin{cases}
 \ttilde\varepsilon_g(\ttilde V_g), & g\in\{(\bar\aa^L\bar\bb)^k\mid k\in\ZZ\},\\
\ttilde\varepsilon_g(\ttilde V_g)\mod 2,
  & g\in\{\bar\bb^{2kn}\bar\aa^m\mid k,m\in\ZZ,\ m\ne0\}, \\
\ttilde\varepsilon_g(\ttilde V_g)+\ttilde\varepsilon_{ij_L\cdot g}(\ttilde 
V_{ij_L\cdot g}),
 & \mbox{otherwise}
 .  \end{cases}
\end{align}
Observe that, if an element $g\in\pi$ is defective, then all elements
$h\in\ttilde\L_{g,L}$ (as well as $h\in\hat\L_{g,L}$ if $n$ is odd) are also
defective (since
$\chi_L(r)=\chi_L(srs^{-1})$ for any $r,s\in\ttilde G_L$), furthermore
$\ttilde\varepsilon_{h,L}=\ttilde\varepsilon_{g,L}$ is the usual augmentation
on $\ZZ[\ttilde\L_{g,L}]$ reduced modulo~2.
For non-defective $g\in\pi$, the $\chi_L$--twisted $\ttilde G_L$--augmentation
$\ttilde\varepsilon_{g,L}$ is well-defined, since the equality
$r_1\cdot g=r_2\cdot g$ implies $r_1^{-1}r_2\in\Stab_{\ttilde G_L}(g)$, hence
$\chi_L(r_1^{-1}r_2)=1$ and $\chi_L(r_1)=\chi_L(r_2)$. We have
$\ttilde\varepsilon_{r\cdot g,L}=\chi_L(r)\ttilde\varepsilon_{g,L}$,
for any $r\in\ttilde G_L$, and for any non-defective $g\in\pi$.

\begin{Lem} \label{lem:defec}
An element $g\in\pi=\pi_-$ is defective if and only if
$g=\bar\bb^{nk}\bar\aa^m$ for some $k,m\in\ZZ$.
In particular, all $\hhat G$--singular elements are defective.
\end{Lem}

\begin{proof}
The assertion easily follows from the following formulae for the stabilizer
$\Stab_{\ttilde G_L}(g)$ of an element $g\in\pi=\pi_-$.
Suppose that $g$ is not of the form $\bar\bb^{nk}\bar\aa^m$, $k,m\in\ZZ$.
If $g=(\bar\aa^L\bar\bb)^k$, $k\in\ZZ$, then $\Stab_{\ttilde G_L}(g)$ is the
cyclic subgroup of $\ttilde G_L$ generated by $ij_L$;
otherwise $\Stab_{\ttilde G_L}(g)=\{1\}$. Therefore
$\chi_L(\Stab_{\ttilde G_L}(g))=\{1\}$, hence $g$ is non-defective.
Suppose that $g=(\bar\aa^L\bar\bb)^{nk}\bar\aa^m$, $k,m\in\ZZ$.
If $m\ne0$ then $\Stab_{\ttilde G_L}(g)$ is the cyclic subgroup of $\ttilde G_L$
generated by $t^ki$ (if $nk$ is odd) or by $t^kj_L$
(if $nk$ is even); otherwise $\Stab_{\ttilde G_L}(g)$ is generated by two
elements $ij_L$, $t^ki$. Therefore
$\chi_L(\Stab_{\ttilde G_L}(g))=\{1,-1\}$, hence $g$ is defective.
\end{proof}

Denote $\Z:=\ZZ$, $V':=V\in\ZZ[\pi]$ for the equation~\eqref{eqn2_2bar}, and
$\Z:=\ZZ_2$, $V':=V\mod 2\in\ZZ_2[\pi]$ for the equation~$(\bar 4_2^\f)$,
where $\pi=\pi_-$.
For any $g\in\pi=\pi_-$, denote by $\hhat V_g'$, $\ttilde V_{g,L}'$,
$\hat V_{g,L}'$, and $V_{g,L}'$ the projections of the element $V'\in\Z[\pi]$
to $\Z[\hhat\L_g]$, $\Z[\ttilde\L_{g,L}]$, $\Z[\hat\L_{g,L}]$, and
$\Z[\L_{g,L}]$, respectively.

\begin{Thm} \label {thm:partit24f}
Suppose $n\in\ZZ\setminus\{0\}$ and $V\in\ZZ[\pi]$ are defined by an element
$v\in F_2$, $\bar v=\bar\bb^{2n}\ne 1$,
as in~$\eqref{eq:**}$. For every $L\in\ZZ$, consider the
left actions~$\eqref{eq:G'}$, $\eqref{eq:hatG'}$, $\eqref{eq:ttildeG}$ of
the groups
$\hhat G\subset\ttilde G_L\subset\hat G_L\subset G_L$ on $\pi$, and the
corresponding (twisted) augmentations $\hhat\varepsilon_g$,
$\ttilde\varepsilon_{g,L}$, $\hat\varepsilon_{g,L}$,
see~$\eqref{eq:varepsilong}$, $\eqref{eq:hhatvarepsilong}$,
$\eqref{eq:defec}$, $\eqref{eq:defec:hat}$.
Each of the equations~\eqref{eqn2_2bar} and~\eqref{eqn4_2fbar} has the following properties:

{\rm (A)}\qua
For every fixed $L,\ell\in\ZZ$ as in~$\eqref{ell}$, the corresponding
equation with the unknown $Z'\in\Z[\pi]$ splits into the system of independent
equations in the subspaces $(\Z[\L_{g,L}\setminus\{1\}])/\sim$ with the
unknowns $Z'_g\in\Z[\L_{g,L}]$, where $g\in\pi$.

{\rm (B)}\qua The following conditions {\rm(i)}, {\rm(ii)} and {\rm(iii)}
are pairwise equivalent:
\begin{enumerate}
\item[\rm (i)] the equation admits a solution;
\item[\rm (ii)] the equation admits a solution with $\ell=\ell_{\max}$;
\item[\rm (iii)] the following conditions {\rm(iii$_1$)} and {\rm(iii$_2$)} hold for
$\ell:=\ell_{\max}$ (compare Lemmas~\ref{lem:reg} and~\ref{lem:defec}):

{\rm (iii$_1$)}\qua If $n$ is even then, for every pair of
elements $g,h\in\pi\setminus\{\bar\bb^{kn}\mid k\in\ZZ\}$ (thus both
$g,h$ are $\hhat G$--regular) with $h=\bar\bb^{2\ell r}g$, $r\in\ZZ$,
one has
$$\hhat\varepsilon_g(\hhat V'_g)=\hhat\varepsilon_h(\hhat V'_h) \in \Z;$$
{\rm (iii$_2$)}\qua There exists $L\in\ZZ$ satisfying the
following conditions. For every pair of elements $g,h\in\pi$ with
$g\not\in\{\bar\bb^{2\ell k}\bar\aa^m\mid k,m\in\ZZ\}$, $w_-(g)=1$,
and $h=\bar\aa^L\bar\bb^\ell g$ (thus both $g,h$ are non-defective),
one has
$$\ttilde\varepsilon_{g,L}(\ttilde V'_{g,L})
  =\ttilde\varepsilon_{h,L}(\ttilde V'_{h,L})\in\Z.$$
Moreover, if $n$ is odd then, for every $m\in\NN$, the pair of
elements $g=\bar\aa^m$, $h=\bar\aa^L\bar\bb^n g$ (thus both $g,h$
are defective) satisfies the following equality in $\ZZ_2$:
$$\ttilde\varepsilon_{g,L}(\ttilde V'_{g,L})
+\ttilde\varepsilon_{h,L}(\ttilde V'_{h,L})
=\begin{cases} 1,& 0<m\le P(L),\\ 0,& m>P(L) \end{cases}\quad
P(L):=\begin{cases} L-1,& L\ge1,\\ -L,& L\le0. \end{cases}
 $$
\end{enumerate}
\end{Thm}

\begin{Rems} \label{rem:indep}
(A)\qua Condition~(iii) is equivalent to the following condition:

\begin{enumerate}
\item[\rm (iv)] the following conditions~{\rm(iv$_{\mbox{\footnotesize
e}}$)} and~{\rm(iv$_{\mbox{\footnotesize o}}$)} hold
(compare Lemmas~\ref{lem:reg} and~\ref{lem:defec}):

{\rm (iv$_{\mbox{\footnotesize e}}$)}\qua Suppose that $n$ is
even, and put $\ell:=\ell_{\max}$. Then, for every pair of elements
$g,h\in\pi\setminus\{\bar\bb^{kn}\mid k\in\ZZ\}$ (thus both $g,h$
are $\hhat G$--regular) with $h=\bar\bb^{2\ell r}g$, $r\in\ZZ$, one
has
 $$
\hhat\varepsilon_g(\hhat V'_g)=\hhat\varepsilon_h(\hhat V'_h) \in \Z.
 $$
Moreover, there exists $L\in\ZZ$ such that, for every pair of
elements $g,h\in\pi$ with $g\not\in\{\bar\bb^{2\ell k}\bar\aa^m\mid
k,m\in\ZZ\}$, $w_-(g)=1$, and $h=\bar\aa^L\bar\bb^\ell g$ (thus both
$g,h$ are non-defective), one has
 $$
\ttilde\varepsilon_{g,L}(\ttilde V'_{g,L})=\ttilde\varepsilon_{h,L}(\ttilde V'_{h,L})
\in\Z.
 $$
{\rm (iv$_{\mbox{\footnotesize o}}$)}\qua Suppose that $n$ is
odd. Then there exists $L\in\ZZ$ satisfying the following
conditions. For every element
$g\in\pi\setminus\{\bar\bb^{2nk}\bar\aa^m\mid k,m\in\ZZ\}$ with
$w_-(g)=1$ (thus both $g,\bar\aa^L\bar\bb^n g$ are non-defective),
one has
 $$
\hat\varepsilon_{g,L}(\hat V'_{g,L})=0 \in \Z.
 $$
Moreover, if $g=\bar\aa^m$ with $m\in\NN$ (thus both $g,\bar\aa^L\bar\bb^ng$
are defective), then
 $$
\hat\varepsilon_{g,L}(\hat V'_{g,L}) = \left\{ \begin{array}{ll}
1,& 0<m\le P(L),\\ 0,& m>P(L) \end{array} \right.
\mbox{ in }\ZZ_2.
 $$
\end{enumerate}
(B)\qua Condition~(iii$_1$) (respectively, the first part
of~(iv$_{\mbox{\footnotesize e}}$)) is equivalent to the similar condition
where $g,h$ run through the sets
$g\in\{\bar\bb^{2k}\bar\aa^m\mid -\ell<2k<\ell,\ m>0\}\cup
     \{\bar\bb^{2k} \mid 0<2k<\ell \}\cup
     \{\bar\bb^{2k+1}\bar\aa^m\mid 0<2k+1\le\ell,\ m\in\ZZ\}$
(thus $g$ is automatically $\hhat G$--regular), and
$h=\bar\bb^{2\ell r}g$ is $\hhat G$--regular with $1\le r<|n|/\ell$.

(C)\qua The first part of the condition~(iii$_2$) (respectively, the
second part of~(iv$_{\mbox{\footnotesize e}}$) or the first part
of~(iv$_{\mbox{\footnotesize o}}$)) is equivalent to the similar
condition where $g$ runs through the set
$\{\bar\bb^{2k}\bar\aa^m\mid 0<2k<\ell_{\max},\ m\ge0\}$.
\end{Rems}

\begin{proof}
(A)\qua Similar to the proof of \fullref{thm:partit34nf}(A).

(B)\qua (i)$\implies$(ii)\qua Suppose that $(L,\ell,Z')$ is a solution.
Then the left-hand side equals
 $$
\frac{1-\bar\bb^{2n}}{1-\bar\bb^\ell} \cdot \varphi^L(Z')
=\frac {1-\bar\bb^{2n}}{1-\bar\bb^{\ell_{\max}}}\cdot
\frac {1-\bar\bb^{\ell_{\max}}}{1-\bar\bb^\ell} \cdot \varphi^L(Z').
 $$
Since
the right-hand sides of~\eqref{eqn2_2bar} and~\eqref{eqn4_2fbar} do not depend on $\ell$,
the triple
$$\Bigl(L,\ell_{\max},\varphi^{-L}\Bigl(\tfrac{1-\bar\bb^{\ell_{\max}}}
  {1-\bar\bb^\ell}\Bigr) \cdot Z'\Bigr)$$
is a solution.

(ii)$\implies$(iii)\qua Consider the case of the equation~\eqref{eqn2_2bar}.
Suppose $(L,\ell_{\max},Z)$ is a solution, and denote $\ell:=\ell_{\max}$.
Observe that, under the assumption $\ell=\ell_{\max}$, the
equation~\eqref{eqn2_2bar} is equivalent to
the following congruence in $\ZZ[\pi_-]$ modulo $K$:
\begin{equation}\label{eq:blue*}
V \equiv \frac{1-\bar\bb^{2n}}{1-\bar\aa^L \bar\bb^{\ell} }\cdot (Z+C_1)
   + \left\{ \begin{array}{cl} 0, & n\mbox{ even}, \\
 - \frac{1-\bar\aa^L}{1-\bar\aa}, & n\mbox{ odd}, \end{array} \right.
\end{equation}
where $C_1:=0$ if $n$ is even,
$C_1:=                      \frac{1-\bar\aa^L}{1-\bar\aa}$ if $n$ is odd and $>0$,
$C_1:=-\bar\aa^L\bar\bb^{-n}\frac{1-\bar\aa^L}{1-\bar\aa}$ if $n$ is odd and $<0$.
The first summand of the right-hand side of this congruence is a
linear combination of the elements
$\frac{1-\bar\bb^{2n}}{1-\bar\aa^L \bar\bb^{\ell} } f\in\ZZ[\pi]$, $f\in\pi$,
with integer coefficients, and thus a linear combination of the elements
$$U=(1+\bar\bb^{2\ell}+\bar\bb^{4\ell}+\ldots+\bar\bb^{2|n|-2\ell})
   (1+\bar\aa^{L}\bar\bb^{\ell})f,\quad f\in\pi.$$
In particular, it is a linear combination of the elements
$\frac{1-\bar\bb^{2n}}{1-\bar\bb^{2\ell}} f$, $f\in\pi$, and thus
a linear combination of the elements
$$W=(1+\bar\bb^{2\ell}+\bar\bb^{4\ell}+\ldots+\bar\bb^{2|n|-2\ell})
f,\quad f\in\pi.$$
In order to prove (iii$_1$), consider the polynomial $W\in\ZZ[\pi]$
from above and observe that, for any $\hhat G$--regular element
$h_r:=\bar\bb^{2\ell r}f$, $r\in\ZZ$, the $\chi$--twisted $\hhat
G$--augmentation of $\hhat W_{h_r}\in\ZZ[\hhat \L_{h_r}]$ (based at
$h_r$) equals
\begin{equation}\label{eq:iii1}
  \hhat\varepsilon_{h_r}(\hhat W_{h_r}) =
 -\hhat\varepsilon_{h_r^{-1}}(\hhat W_{h_r^{-1}}) = \left\{ \begin{array}{ll}
                1,&q\mbox{ even and $p\ne0$, or }\ell \nmid q,\\
                0,&q\mbox{ odd or $p=0$, and }\ell\mid q, \end{array} \right.
\end{equation}
where $f=\bar\bb^q\bar\aa^p$, the ``canonical'' form of $f\in\pi$,
similar to~\eqref{eq:canon}. Obviously, for any element
$h\in\pi\setminus\{h_r,h_r^{-1}\mid r\in\ZZ\}$, the $\chi$--twisted
$\hhat G$--augmentation of $\hhat W_h$ (based at $h$) vanishes.
Observe also that the right-hand side of~\eqref{eq:iii1} does not
depend on $r$. This shows that the element $W\in\ZZ[\pi]$ satisfies
the condition~(iii$_1$). For $n$ even, this implies that $V$ also
satisfies~(iii$_1$), since $V$ is a linear combination of such
elements $W$, together with the elements $f+f^{-1}$,
$f\in\pi\setminus\{1\}$, and $1\in\pi$.

In order to prove (iii$_2$), let us consider the integer $L$ and the polynomial
$U\in\ZZ[\pi]$ from above. Recall that $U$ has the form
$U=(1+\bar\bb^{2\ell}+\bar\bb^{4\ell}+\ldots+\bar\bb^{2|n|-2\ell})
   (1+\bar\aa^L \bar\bb^{\ell})\cdot f$, for some $f\in\pi$,
$f=\bar\bb^q\bar\aa^p$, $p,q\in\ZZ$. Take any
$g\in\pi\setminus\{\bar\bb^{2\ell r}\bar\aa^m\mid r,m\in\ZZ\}$ with
$w_-(g)=1$; thus the elements $g$ and $h:=\hat t_L\cdot g=\bar\aa^L
\bar\bb^\ell g$ are automatically non-defective, see \fullref{lem:defec}. If $g$ belongs to the set
$S_{f,L}:=\{(\bar\aa^L\bar\bb^\ell)^rf\mid r\in\ZZ\}$ then
$\ell\nmid q$, thus the $\chi_L$--twisted $\ttilde
G_L$--augmentation of $\ttilde U_{g,L}\in\ZZ[\tilde\L_{g,L}]$ (based
at $g$) equals
\begin{multline*}
\ttilde\varepsilon_{g,L}(\ttilde U_{g,L})
= \ttilde\varepsilon_{j_L\cdot g^{-1},L}(\ttilde U_{j_L\cdot g^{-1},L}) \\
=-\ttilde\varepsilon_{g^{-1},L}(\ttilde U_{g^{-1},L})
=-\ttilde\varepsilon_{j_L\cdot g,L}(\ttilde U_{j_L\cdot g,L})
= 1 \in\ZZ.
\end{multline*}
If $g,g^{-1},j_L\cdot g,j_L\cdot g^{-1}\not\in S_{f,L}$
then the $\chi_L$--twisted $\ttilde G_L$--augmentation of $\ttilde U_{g,L}$
(based at $g$) vanishes. Observe that $g\in S_{f,L}$ if and only if
$h=\hat t_L\cdot g\in S_{f,L}$,
for any $g\in\pi$ (without assumption $w_-(g)=1$). Hence
$g^{-1}\in S_{f,L}$ if and only if $j_L\cdot h\in S_{f,L}$;
$j_L\cdot g\in S_{f,L}$ if and only if $h^{-1}\in S_{f,L}$;
$j_L\cdot g^{-1}\in S_{f,L}$ if and only if $j_L\cdot h^{-1}\in S_{f,L}$.
Together with the above properties of the $\chi_L$--twisted
$\ttilde G_L$--augmentation, this proves the desired equality
$\ttilde\varepsilon_{g,L}(\ttilde U_{g,L})
=\ttilde\varepsilon_{h,L}(\ttilde U_{h,L})\in\ZZ$, thereby proving the first
part of~(iii$_2$) for the element $U\in\ZZ[\pi]$.
Therefore $V$ also satisfies the first part of~(iii$_2$), since $V$ is a linear
combination of such elements $U$, together with the elements $f+f^{-1}$,
$\bar\aa^r$, and $1\in\pi$, where $f\in\pi\setminus\{1\}$, $r\in\ZZ$.

Suppose that $n$ is odd, and take any element $g=\bar\aa^m$ with $m\in\NN$.
Denote, similarly to above, $h:=\bar\aa^L\bar\bb^n g$ (thus both $g,h$ are
defective). It is obvious that
$\ttilde\varepsilon_{g,L}(\ttilde U_{g,L})=0\in\ZZ_2$
if and only if $g,g^{-1}\not\in S_{f,L}$, moreover
$\ttilde\varepsilon_{h,L}(\ttilde U_{g,L})=0\in\ZZ_2$
if and only if $h,j_L\cdot h\not\in S_{f,L}$.
Since $g,g^{-1}\not\in S_{f,L}$ is equivalent to $h,j_L\cdot h\not\in S_{f,L}$,
we obtain
$\ttilde\varepsilon_{g,L}(\ttilde U_{g,L})
+\ttilde\varepsilon_{h,L}(\ttilde U_{h,L})=0$.
Therefore
$\ttilde\varepsilon_{g,L}(\ttilde V_{g,L})
+\ttilde\varepsilon_{h,L}(\ttilde V_{h,L})
=\ttilde\varepsilon_{g,L}(\ttilde D_{g,L})
+\ttilde\varepsilon_{h,L}(\ttilde D_{h,L})$ where
$D:=-\frac{1-\bar\aa^L}{1-\bar\aa}\in\ZZ[\pi]$, since $V-D$ is a linear
combination of such elements $U$, together with the elements $f+f^{-1}$,
$f\in\pi\setminus\{1\}$, and $1\in\pi$. One easily computes
$$
\ttilde\varepsilon_{h,L}(\ttilde D_{h,L})=0, \quad
\ttilde\varepsilon_{g,L}(\ttilde D_{g,L})
=\left\{ \begin{array}{ll} 1,& 0<m\le P(L),\\ 0,& m>P(L) \end{array} \right.
\qquad \mbox{in }\ZZ_2.
$$
This completes the proof of~(iii$_2$).

Consider the case of the equation~\eqref{eqn4_2fbar}.
Suppose $(L,\ell_{\max},Z')$ is a solution, and denote $\ell:=\ell_{\max}$.
It follows from~\eqref{eqn4_2fbar} that the
congruence~\eqref{eq:blue*} in $\ZZ_2[\pi_-]$ holds modulo $K'$,
where the coefficients are reduced modulo 2. It follows from the case
of~\eqref{eqn2_2bar} that $V'$ satisfies the $\mod 2$ analogue of the condition~(iii).

(iii)$\implies$(i)\qua Let us consider the case of the
equation~\eqref{eqn2_2bar}.
Suppose $n$ is odd, put $\ell:=\ell_{\max}=|n|$.

\textbf{Step 1}\qua For every $L\in\ZZ$ and for every polynomial $V\in\ZZ[\pi]$, there
exists a (unique) presentation satisfying the following congruence modulo $K$:
\begin{equation}\label{eq:Vodd}
\V := V+\frac{1-\bar\aa^L}{1-\bar\aa} \equiv U + W_1 + W_2 + W_3 + R,
\end{equation}
where $U$ is a linear combination of $(1-\bar\bb^{2n})h$, $h\in\pi$, while
$W_1,W_2,W_3,R\in\ZZ[\pi]$ have the form
\begin{align*}
W_1 &=\sum_{m>0,\ 0<2k<\ell}
  (a^+_{g_{m,k}}g_{m,k} + a^-_{g_{m,k}}ij_L\cdot g_{m,k}
  +b^+_{g_{m,k}}h_{m,k} + b^-_{g_{m,k}}ij_L\cdot h_{m,k}), \\
W_2 &= \sum_{0<2k<\ell}
(a_{g_{0,k}}g_{0,k} + b_{g_{0,k}}h_{0,k}), \\
W_3 &= \sum_{m>0}
(a_{g_{m,0}}g_{m,0} + b^+_{g_{m,0}}h_{m,0} + b^-_{g_{m,0}}ij_L\cdot h_{m,0}), \\
R &= b_1 h_{0,0},
\end{align*}
where
$g_{m,k}:=\bar\bb^{2k}\bar\aa^m$, $h_{m,k}:=\bar\aa^L\bar\bb^{\ell}g_{m,k}$,
and $a_g^\pm,b_g^\pm,a_g,b_g\in\ZZ$ with the additional condition that
$b_1,b^+_{\bar\aa^m}\in\{0,1\}$,
$b^-_{\bar\aa^m}\in\{ a_{\bar\aa^m}-b^+_{\bar\aa^m},\
                      a_{\bar\aa^m}-b^+_{\bar\aa^m}+1 \}$, $m>0$
(these coefficients correspond to $\hhat G$--singular elements $h_{m,0}$).
Here uniqueness follows from the equalities
\begin{align*}
a^+_g&=\hhat\varepsilon_g(\hhat \V_g), &
a^-_g&=\hhat\varepsilon_{ij_L\cdot g}(\hhat \V_{ij_L\cdot g}), \\
b^+_g&=\hhat\varepsilon_h(\hhat \V_h), &
b^-_g&=\hhat\varepsilon_{ij_L\cdot h}(\hhat \V_{ij_L\cdot h}), &
h&:=\bar\aa^L\bar\bb^{\ell}g,
\end{align*}
while
$a_g=\hhat\varepsilon_g(\hhat \V_g)$ for $W_2, W_3$;
$b_g=\hhat\varepsilon_h(\hhat \V_h)$ for $W_2$;
furthermore
$b_g^+\mmod2=\hhat\varepsilon_h(\hhat\V_h)$ and
$b_g^-\mmod2=\hhat\varepsilon_{ij_L\cdot h}(\hhat\V_{ij_L\cdot h})\in\ZZ_2$
for $\hhat G$--singular $h=h_{m,0}$ in $W_3,R$, where $\hhat\varepsilon_h$ is
defined similarly to the case of $\hhat G$--regular $h$, by reducing $\mmod 2$.

\textbf{Step 2}\qua Observe that every summand of the sums $W_1,W_2,W_3$ has the form
\begin{align}
\notag a^+g &+ a^-ij_L\cdot g + b^+h + b^-ij_L\cdot h \\
\notag &\equiv a^+g - a^-(j_L\cdot g+h+h^{-1}) + b^+h - b^-(j_L\cdot h +g+g^{-1}) \\
\label{eq:zvezda}
&= (a^+-b^-) g + (b^+-a^-) h - a^-(h^{-1}+j_L\cdot g) - b^-(g^{-1}+j_L\cdot h),
\end{align}
where $w_-(g)=1$ and $h:=\bar\aa^L\bar\bb^{\ell}g$. Here
$a^\pm:=a^\pm_g$, $b^\pm:=b^\pm_g$ for $W_1$;
$a^+:=a_g$, $a^-:=0$, $b^+:=b_g$, $b^-:=0$ for $W_2$;
$a^+:=a_g$, $a^-:=0$, $b^\pm:=b_g^\pm$ for $W_3$.

Let us show that $a^+-b^- = b^+-a^-$, provided that $L\in\ZZ$ is taken as in
the condition~(iii). Indeed, from the above formulae for
$a^{\pm}_g$, $b^{\pm}_g$, we have that, for $W_1$ and $W_2$,
\begin{align*}
a^+-b^- - (b^+-a^-) &= a^+ + a^- - (b^+ + b^-) \\
&= a^+_g+a^-_g-(b^+_g+b^-_g) \\
&= \hhat\varepsilon_g(\hhat\V_g) + \hhat\varepsilon_{ij_L\cdot g}(\hhat\V_{ij_L\cdot g})
- (\hhat\varepsilon_h(\hhat\V_h) +
\hhat\varepsilon_{ij_L\cdot h}(\hhat\V_{ij_L\cdot h})) \\
&= \ttilde\varepsilon_{g,L}(\ttilde\V_{g,L}) -
\ttilde\varepsilon_{h,L}(\ttilde\V_{h,L}) \\
&= \hat\varepsilon_{g,L}(\hat\V_{g,L}) ,
\end{align*}
see~\eqref{eq:varepsilon:mod2}.
Now, if $L\in\ZZ$ is taken as in the condition~(iii), then the latter
expression vanishes, due to~(iv$_{\mbox{\footnotesize o}}$) or~(iii$_2$).
Similarly, for $g=\bar\aa^m$, $m>0$, as in $W_3$, we obtain
$(b^+_g+b^-_g-a_g)\mod2=\hat\varepsilon_{g,L}(\hat\V_{g,L})=0\in\ZZ_2$,
due to~\eqref{eq:varepsilon:mod2} and the second part of~(iii$_2$).
Since $b^+_g+b^-_g-a_g\in\{0,1\}$, see above, we have $a_g=b^+_g+b^-_g$.

Since $a^+-b^- = b^+-a^-$, the expression~\eqref{eq:zvezda} equals
$$(a^+{-}b^-) (1{+}\bar\aa^L\bar\bb^{\ell})g
{-} a^-(1{+}\bar\aa^L\bar\bb^{\ell})h^{-1}
{-} b^-(1{+}\bar\aa^L\bar\bb^{\ell})j_L\cdot h
= (1{+}\bar\aa^L\bar\bb^{\ell}) Z_g,$$
where $Z_g:=(a^+-b^-)g - a^-h^{-1} - b^-j_L\cdot h$.

\textbf{Step 3}\qua For the remainder term $R$, observe that
$h_{0,0}=\bar\aa^L\bar\bb^{\ell}\equiv 1+\bar\aa^L\bar\bb^{\ell}$.
This shows that every summand in the right-hand side of~\eqref{eq:Vodd}
is divisible (modulo $K$) by $1+\bar\aa^L\bar\bb^{\ell}$. Hence
it is also divisible by
$\smash{\frac{1-\bar\bb^{2n}}{1-\bar\aa^L\bar\bb^\ell}}$,
since $\ell=|n|$. This means that $V$ has the form~\eqref{eq:blue*}
and therefore~\eqref{eqn2_2bar} admits a solution.

Suppose that $n$ is even and that $V\in\ZZ[\pi]$ satisfies~(iii$_1$), or the
first part of~(iv$_{\mbox{\footnotesize e}}$). Put $\ell:=\ell_{\max}$.

\textbf{Step 1}\qua For every $L\in\ZZ$ and for every polynomial $V\in\ZZ[\pi]$
satisfying~(iii$_1$), there exists a (unique) presentation satisfying the
following congruence modulo $K$:
\begin{equation}\label{eq:Veven}
V \equiv U + W_1 + W_2 + W_3 + R,
\end{equation}
where $U$ is a linear combination of $(1-\bar\bb^{2n})h$, $h\in\pi$, while
$W_1,W_2,W_3,R\in\ZZ[\pi]$ have the form
\begin{align*}
W_1 &= \hspace{-3mm}\sum_{\begin{array}{c}\scriptstyle m>0,\\[-1.5ex]
  \scriptstyle0<2k<\ell\end{array}}\hspace{-2mm}
\frac {1-\bar\bb^{2n}}{1-\bar\bb^{2\ell}}
(a^+_{g_{m,k}}g_{m,k} + a^-_{g_{m,k}}ij_L\cdot g_{m,k}
+b^+_{g_{m,k}}h_{m,k} + b^-_{g_{m,k}}ij_L\cdot h_{m,k}), \\
W_2 &= \sum_{0<2k<\ell}
\frac {1-\bar\bb^{2n}}{1-\bar\bb^{2\ell}}
(a_{g_{0,k}}g_{0,k} + b_{g_{0,k}}h_{0,k}), \\\
W_3 &= \sum_{m>0}
\frac {1-\bar\bb^{2n}}{1-\bar\bb^{2\ell}} a_{g_{m,0}}g_{m,0}, \\
R &= a_{g_{0,n/2}}g_{0,n/2} = a_{\bar\bb^n} \bar\bb^n,
\end{align*}
where
$g_{m,k}:=\bar\bb^{2k}\bar\aa^m$,
$h_{m,k}:=\bar\aa^L\bar\bb^{\ell}g_{m,k}$,
$a_g^\pm,b_g^\pm,a_g,b_g\in\ZZ$ with the additional condition that
$a_{\bar\bb^n}\in\{0,1\}$ (this coefficient corresponds to the
$\hhat G$--singular elements $\bar\bb^{(2q+1)n}$, $q\in\ZZ$). Here uniqueness
follows from the equalities
\begin{align*}
a^+_g&=\hhat\varepsilon_g(\hhat V_g), &
a^-_g&=\hhat\varepsilon_{ij_L\cdot g}(\hhat V_{ij_L\cdot g}), \\
b^+_g&=\hhat\varepsilon_h(\hhat V_h), &
b^-_g&=\hhat\varepsilon_{ij_L\cdot h}(\hhat V_{ij_L\cdot h}), &
h&:=\bar\aa^L\bar\bb^{\ell}g,
\end{align*}
moreover
$a_g=\hhat\varepsilon_g(\hhat V_g)$ (as an equality modulo 2 if $g=\bar\bb^n$),
$b_g=\hhat\varepsilon_h(\hhat V_h)$
(observe that, for any $g,h$ as in $W_1,W_2,W_3$, and for any $q\in\ZZ$,
the elements $\bar\bb^{2\ell q}g,\bar\bb^{2\ell q}h$ are $\hhat G$--regular).
Here $g_{0,n/2}=\bar\bb^n$ appears in the remainder term $R$
(corresponding to the case $k=m=0$),
since the condition~(iii$_1$), or the first part
of~(iv$_{\mbox{\footnotesize e}}$), poses no restriction to
the coefficients of $V$ at $\bar\bb^{(2q+1)n}$, $q\in\ZZ$.
Actually, $V$ admits similar presentations, with the additional terms
\begin{align*}
\frac {1-\bar\bb^{2n}}{1-\bar\bb^{2\ell}}
(b^+_{g_{m,0}}h_{m,0}+b^-_{g_{m,0}}ij_L\cdot h_{m,0}) &\equiv 0 &  &
\text{in } W_3,\ m>0,\\
\text{and}\qquad
\frac {1-\bar\bb^{2n}}{1-\bar\bb^{2\ell}}(a_1+b_1 h_{0,0}) &\equiv
  a_1\bar\bb^n & &\text{in } R,
\end{align*}
where the coefficients $b^\pm_{g_{m,0}}$, $a_1$, $b_1$ are arbitrary
integers. In the presentation~\eqref{eq:Veven}, these terms are omitted, in
order to have uniqueness.

\textbf{Step 2}\qua Fix arbitrary $g\in\pi$, $q\in\ZZ$, denote
$\tilde g:=\bar\bb^{2\ell q}g$, $\tilde{\tilde g} :=\bar\bb^{-2\ell q}g$,
$h:=\bar\aa^L\bar \bb^\ell g$.
One easily observes
\begin{align*}
\widetilde{ij_L\cdot g}&{=}ij_L\cdot \tilde g, &
\tilde h&{=}\widetilde{\bar\aa^L\bar\bb^{\ell}g}{=}\bar\aa^L\bar\bb^{\ell}\tilde g,
&
\widetilde{ij_L\cdot h}&{=}ij_L\cdot \tilde h{=}ij_L\cdot \bar\aa^L\bar\bb^{\ell}\tilde g, \\
\widetilde{g^{-1}}&{=}(\tilde{\tilde g})^{-1}, &
\widetilde{h^{-1}}&{=}(\tilde{\tilde h})^{-1}, &
\widetilde{j_L\cdot g}&{=}j_L\cdot \tilde{\tilde g}, &\hspace{-6mm}
\widetilde{j_L\cdot h}&{=}j_L\cdot \tilde{\tilde h}.
\end{align*}
Hence, by applying to each summand of the sums $W_1,W_2$ the
arguments of Step~2 of the case of $n$ odd, it follows that the
condition~(iii$_2$), or the second part of~(iv$_{\mbox{\footnotesize e}}$),
implies
 $$
\frac {1-\bar\bb^{2n}}{1-\bar\bb^{2\ell}}
(a^+_gg + a^-_gij_L\cdot g + b^+_gh + b^-_gij_L\cdot h)
\equiv
\frac {1-\bar\bb^{2n}}{1-\bar\bb^{2\ell}} (1+\bar\aa^L\bar\bb^\ell) E_g
=
\frac {1-\bar\bb^{2n}}{1-\bar\aa^L\bar\bb^\ell} E_g,
 $$
for some $E_g\in\ZZ[\pi]$, where $g\in\pi$ as in $W_1,W_2$. Therefore
$W_1+W_2=\frac {1-\bar\bb^{2n}}{1-\bar\aa^L\bar\bb^\ell}E$, for $E:=\sum_g E_g$.

\textbf{Step 3}\qua For the term $W_3$, we observe that
$$\frac {1-\bar\bb^{2n}}{1-\bar\bb^{2\ell}}\bar\aa^m
\equiv \frac {1-\bar\bb^{2n}}{1-\bar\aa^L\bar\bb^\ell}F,$$
for some $F\in\ZZ[\pi]$, due to \fullref{cor:ab}(c). For the remainder
term $R$, we observe that
$$\bar\bb^n \equiv \frac
{1-\bar\bb^{2n}}{1-\bar\aa^L\bar\bb^\ell}\bar\aa^L\bar\bb^{\ell-n},$$
due to \fullref{cor:ab}(a).
This shows that every summand of the right-hand side of~\eqref{eq:Veven}
is divisible (modulo $K$) by
$\smash{\frac{1-\bar\bb^{2n}}{1-\bar\aa^L\bar\bb^\ell}}$.
This means that $V$ has the form~\eqref{eq:blue*},
therefore~\eqref{eqn2_2bar}
admits a solution.

For the equation~\eqref{eqn4_2fbar}, the implication (iii)$\implies$(i)
immediately follows from the case of the equation~\eqref{eqn2_2bar}.
\end{proof}

Define the notions of a defective $\tilde G$--orbit and a defective
$\hat G_L$--orbit, similarly to the definition of a defective
$\ttilde G_L$--orbit, see above~\eqref{eq:defec}. Below we consider
a set $S$ as a subset of the abelian group $\ZZ[S]$.

\begin{Pro} \label {pro:orbits24f}
Suppose that $\bar v=\bar u^{2\mu}=\bar\bb^{2n}$ in the group
$$\pi=\pi_-=\langle\aa,\bb\mid\aa\bb\aa\bb^{-1}\rangle,$$
where
$n\in\ZZ\setminus\{0\}$, $\bar u=\bar\bb^{\ell}$,
$\mu=\smash{\frac{n}{|n|}}2^s$, $s\ge0$, $\ell=\ell_{\max}>0$ odd,
$n=\mu\ell$. Consider the corresponding actions~$\eqref{eq:hatG'}$,
$\eqref{eq:ttildeG}$ of the groups $\hhat G\subset\ttilde
G_L\subset\hat G_L$ on $\pi$. Then the orbits $\hhat \L_h$,
$\ttilde\L_{h,L}$, $\hat \L_{h,L}$, $h\in\pi$, under these actions
have the following form:

{\rm (A)}\qua For $h=\bar\bb^{2q}\bar\aa^p$, $p,q\in\ZZ$, one has
\begin{align*}
\hhat \L_h
&= \{\bar v^k h^{\pm1}\mid k\in \ZZ\}
= \{\bar\bb^{2q+2kn}\bar\aa^p\mid k\in \ZZ\} \cup
  \{\bar\bb^{-2q+2kn}\bar\aa^{\;-p}\mid k\in \ZZ\} , \\
 \ttilde\L_{h,L}
&= \{\bar v^k h^{\pm1}\mid k\in \ZZ\} \cup
\{\bar v^k\bar\aa^L\bar\bb^\ell h^{\pm1} \bar\aa^L\bar\bb^{-\ell} \mid
k\in\ZZ\} \\
&= \{\bar\bb^{2q+2kn}\bar\aa^p\mid k\in \ZZ\} \cup
  \{\bar\bb^{-2q+2kn}\bar\aa^{\;-p}\mid k\in \ZZ\} \\
&\qquad \cup
  \{\bar\bb^{2q+2kn}\bar\aa^{\;-p}\mid k\in \ZZ\} \cup
  \{\bar\bb^{-2q+2kn}\bar\aa^p\mid k\in \ZZ\} , \\
 \hat \L_{h,L}
&= \{\bar u^{2k} h^{\pm1}\mid k\in \ZZ\} \sqcup
  \{\bar u^{2k} (\bar\aa^L\bar\bb^\ell h)^{\pm1}\mid k\in \ZZ\} \\
&= (\{\bar\bb^{2q+2k\ell}\bar\aa^p\mid k\in \ZZ\}
 \cup \{\bar\bb^{-2q+2k\ell}\bar\aa^{\;-p}\mid k\in \ZZ\}) \\
&\qquad\sqcup
  (\{\bar\bb^{-2q-\ell+2k\ell}\bar\aa^{p-L}\mid k\in \ZZ\} \cup
  \{\bar\bb^{2q+\ell+2k\ell}\bar\aa^{p-L}\mid k\in \ZZ\}) .
\end{align*}
{\rm (B)}\qua For $h=\bar\bb^{2q+\ell}\bar\aa^p$, $p,q\in\ZZ$, one has
\begin{align*}
\hhat \L_h
&= \{\bar v^k h^{\pm1}\mid k\in \ZZ\}
= \{\bar\bb^{2q+\ell+2kn}\bar\aa^p\mid k\in \ZZ\} \cup
  \{\bar\bb^{-2q-\ell+2kn}\bar\aa^p\mid k\in \ZZ\}, \\
  \ttilde\L_{h,L}
&= \{\bar v^k h^{\pm1}\mid k\in \ZZ\} \cup \{\bar
v^k\bar\aa^L\bar\bb^\ell h^{\pm1} \bar\aa^L\bar\bb^{-\ell} \mid
k\in\ZZ\} \\
&= \{\bar\bb^{2q+\ell+2kn}\bar\aa^p\mid k\in \ZZ\}
  \cup \{\bar\bb^{-2q-\ell+2kn}\bar\aa^p\mid k\in \ZZ\} \\
&\qquad \cup
\{\bar\bb^{2q+\ell+2kn}\bar\aa^{\;-p-2L}\mid k\in \ZZ\} \cup
      \{\bar\bb^{-2q-\ell+2kn}\bar\aa^{\;-p-2L}\mid k\in \ZZ\} , \\
 \hat \L_{h,L}
&= \{\bar u^{2k} h^{\pm1}\mid k\in \ZZ\} \sqcup
  \{\bar u^{2k} (\bar\aa^L\bar\bb^{-\ell} h)^{\pm1}\mid k\in \ZZ\} \\
&= (\{\bar\bb^{2q+\ell+2k\ell}\bar\aa^p\mid k\in \ZZ\}
  \cup \{\bar\bb^{-2q-\ell+2k\ell}\bar\aa^p\mid k\in \ZZ\}) \\
&\qquad \sqcup
  (\{\bar\bb^{2q+2k\ell}\bar\aa^{p+L}\mid k\in \ZZ\} \cup
  \{\bar\bb^{-2q+2k\ell}\bar\aa^{\;-p-L}\mid k\in \ZZ\}) .
\end{align*}
{\rm(C)}\qua For any $h\in F_2$, let us enumerate consecutively the subsets
appearing in the above decompositions of the orbits $\hhat \L_h$,
$\ttilde\L_{h,L}$ and $\hat \L_{h,L}$, thus the decompositions have the
forms
\begin{align*}
\hhat \L_h&=S_{h,1}\cup S_{h,2}, \\
\ttilde\L_{h,L}&=S_{h,3}\cup S_{h,4}\cup S_{h,5}\cup S_{h,6}, \\
\hat \L_{h,L}&=(S_{h,7}\cup S_{h,8})\sqcup (S_{h,9}\cup S_{h,10}).
\end{align*}
Then, for any non-defective orbit, the restriction to this orbit of
the corresponding twisted augmentation (based at $h$) sends
$S_{h,2k+1}\to 1$, $S_{h,2k}\to -1$, see~\eqref{eq:hhatvarepsilong},
\eqref{eq:defec} and~\eqref{eq:defec:hat}.
\end{Pro}

\begin{proof}
The above presentations of the orbits $\hhat \L_h$ follow from
\fullref{pro:orbits34nf} with $\varepsilon=-1$, $\bar u=\bar\bb^{2n}$.
In other cases, the proof is similar to the proof of
\fullref{pro:orbits34nf}, using~\eqref{eq:G'}--\eqref{eq:actions}.
\end{proof}

\begin{Ex} \label{ex:Wicks}
Let us investigate existence of non-faithful solutions of the equation~\eqref{eqn2'}
with $v=B_\aa B_{\aa^{-1}}$, $B=\aa\bb\aa\bb^{-1}$, $\vartheta=-1$,
see \fullref{subsec:appl}.
Observe that this is a ``mixed'' case (see \fullref{tbl2}) with $\bar v=1\in\pi$,
$V=\aa+\aa^{-1}\in\ZZ[\pi]$, thus $p_Q(V)=0\in Q$. This means that the first
and the second derived equations~\eqref{eqn2_1} and~\eqref{eqn2_2} admit solutions,
see \fullref{subsec:2der} and \fullref{lem:sol21}. However, we will
use the method of Wicks~\cite{W} to show that
the equation~\eqref{eqn2'} does not admit non-faithful solutions.
In more detail, consider the cyclically reduced word $W$ obtained from the
right-hand side $vB^{-1}v^{-1}B=[v,B^{-1}]$ of the equation~\eqref{eqn2'}. For each
word $W_i$ obtained from $W$ by cyclic permutation, see below, we will find all
presentations of this
word in the form $abca^{-1}b^{-1}c^{-1}$ due to Wicks, see~\cite{W}.
We will observe that the corresponding ``canonical'' solutions
$(x,y)=(ab,cb)$ are faithful. This allows one to conclude that the equation
\eqref{eqn2'} with $v=B_\aa B_{\aa^{-1}}$, $B=\aa\bb\aa\bb^{-1}$, $\vartheta=-1$
has only faithful solutions.

Here $W_i=V_iU_i$ where $U_i,V_i$ are the subwords of $W$ which are defined by
the properties $W=U_iV_i$, $|U_i|=i$, and $|W|=|U_i|+|V_i|$, where $|\cdot|$
means the length of the word.
The cyclically reduced word $W$ has the form
 $$
W=xxyxy^{-1}x^{-1}yxy^{-1}xyx^{-1}y^{-1}x^{-1}x^{-1}yx^{-1}y^{-1}
xyx^{-1}y^{-1}x^{-1}yxy^{-1},
 $$
and it has length~26.
By a straightforward calculation, the words $W_i,W_{i+13}$ with
$i=0,1,6,7,8,9,10$, and only such, have the Wicks form
$abca^{-1}b^{-1}c^{-1}$, or $ded^{-1}e^{-1}$, with non-empty subwords
$a,b,c,d,e$:
\begin{align*}
W=&W_0 && \text{with } (a,b,c)=
  (\aa,\ \aa\bb\aa\bb^{-1}\aa^{-1}\bb\aa\bb^{-1}\aa,\  \bb\aa^{-1}\bb^{-1});\\
&W_1 && \text{with } (a,b,c)=
  (\aa\bb\aa\bb^{-1}\aa^{-1}\bb\aa\bb^{-1}\aa,\ \bb\aa^{-1}\bb^{-1},\ \aa^{-1});\\
&W_6 && \text{with } (d,e)=
  (\bb\aa\bb^{-1},\
  \aa\bb\aa^{-1}\bb^{-1}\aa^{-1}\aa^{-1}\bb\aa^{-1}\bb^{-1}\aa);\\
&W_7 && \text{with } (d,e)=
(\aa,\ \bb^{-1}\aa\bb\aa^{-1}\bb^{-1}\aa^{-1}
  \aa^{-1}\bb\aa^{-1}\bb^{-1}\aa\bb);\\
&W_8 && \text{with } (d,e)=
  (\bb^{-1}\aa\bb\aa^{-1}\bb^{-1}\aa^{-1}\aa^{-1}\bb\aa^{-1}\bb^{-1}\aa\bb,\ 
  \aa^{-1});\\
&W_9 && \text{with } (d,e)=
  (\aa\bb\aa^{-1}\bb^{-1}\aa^{-1}\aa^{-1}\bb\aa^{-1}\bb^{-1}\aa,\
  \bb\aa^{-1}\bb^{-1});\\
&W_{10} && \text{with } (a,b,c)=
  (\bb\aa^{-1}\bb^{-1},\ \aa^{-1},\
  \aa^{-1}\bb\aa^{-1}\bb^{-1}\aa\bb\aa^{-1}\bb^{-1}\aa^{-1}).
\end{align*}
Here the Wicks form for the word $W_{i+13}$ is obtained from the Wicks form of
$W_i$ in the obvious way.
\end{Ex}

\begin{Ex} \label{ex:Wicks2}
Similarly to \fullref{ex:Wicks}, one can investigate existence of
solutions of the equations~\eqref{eqn3'} and~\eqref{eqn4'} with
$v=B_\aa B_{\aa^{-1}}$, $\vartheta=-1$, where $B=\aa\bb\aa^{-1}\bb^{-1}$
for~\eqref{eqn3'}, while $B=\aa\bb\aa\bb^{-1}$ for~\eqref{eqn4'}, see
\fullref{subsec:appl}.
In this case, one considers the ``non-orientable'' forms
$abcbac^{-1}$ and $a^2bc^2b^{-1}$, due to Wicks~\cite{W2}.
\end{Ex}

\section{Tables for the ``mixed'' cases of Tables~\ref{tbl1} and~\ref{tbl2}}
\label{sec:Quadtab}

In this section, we summarize the main results of Sections~\ref{sec:Quad3}
and~\ref{sec:quad4} in two tables below. \fullref{tbl3} deals with
faithful solutions in
the so called ``mixed'' case~(4c) of \fullref{tbl1}, while \fullref{tbl4} deals with
non-faithful solutions in the ``mixed'' cases~(2d),~(3c),~(4e) of
\fullref{tbl2},
see \fullref{rem:tables}. Observe that $\vartheta=-1$,
$\sgn_\varepsilon(v)=1$ in all ``mixed'' cases.

Specifically, we denote
$B=\aa\bb\aa^{-\varepsilon}\bb^{-1}\in F_2=\langle \aa,\bb\mid\rangle$ and
study the equation
$$\xalpha\ybeta\xalpha^{-\delta}\ybeta^{-1}=vB^{-1}v^{-1}B$$
in the group $N=\llangle B\rrangle $ with two unknowns $\xalpha\in N$, $\ybeta\in F_2$.
A solution of this equation is called {\it faithful} if
$w_\varepsilon(\ybeta)=\delta$. The parameters $\varepsilon,\delta\in\{1,-1\}$
and the conjugation parameter $v\in F_2$ of the equation are not arbitrary, but
run through the following families, corresponding to the ``mixed'' cases,
see \fullref{rem:tables} and \fullref{def:mixed}:

``Mixed'' case for faithful solutions (case~(4c) of \fullref{tbl1}):
\begin{enumerate}
\item[(4)] $\delta=\varepsilon=-1$, $\bar v=\bar \bb^{2n}$, $n\in\ZZ$.
\end{enumerate}

``Mixed'' cases for non-faithful solutions (cases~(2d), (3c), (4e) of
\fullref{tbl2}):
\begin{enumerate}
\item[(2)] $\delta=1$, $\varepsilon=-1$, $\bar v=\bar \bb^{2n}$, $n\in\ZZ$;
\item[(3)] $\delta=-1$, $\varepsilon=1$, $\bar v=\bar \aa^{2m}\bar
\bb^{2n}$, $m,n\in\ZZ$;
\item[(4)] $\delta=\varepsilon=-1$, $\bar v=\bar \aa^{2m}\bar \bb^{4n}$,
$m,n\in\ZZ$.
\end{enumerate}

As above, $\bar v\in\pi$ denotes the class of $v\in F_2$ in $\pi=F_2/N$.
In each of these four ``mixed'' cases, let us write the element $v$ in the
following canonical form:
\begin{align*}
v&=\bb^{2n}\prod B_{v_i}^{n_i}; \\
v&=\bb^{2n}\prod B_{v_i}^{n_i}; \\
v&=c^{2d}\prod B_{v_i}^{n_i}, & c&=\aa^{m/d}\bb^{n/d} \text{ if } |m|+|n|>0, &
  v&=\prod B_{v_i}^{n_i} \text{ if } m=n=0; \\
v&=c^{2d}\prod B_{v_i}^{n_i}, & c&=\aa^{m/d}\bb^{2n/d} \text{ if } |m|+|n|>0, &
  v&=\prod B_{v_i}^{n_i} \text{ if } m=n=0;
\end{align*}
respectively, where $\prod B_{v_i}^{n_i}=\prod_{i=1}^r B_{v_i}^{n_i}$,
$v_i\in F_2$, $n_i\in\ZZ$, $B_{v_i}=v_i B v_i^{-1}$, $1\le i\le r$,
$d=\gcd(m,n)$ if $|m|+|n|>0$.

%%%tables here %%%%
\begin{table}[ht!]
\begin{minipage}{\linewidth}
\renewcommand{\thempfootnote}{\rm(\roman{mpfootnote})}
\begin{footnotesize}
\begin{center}
\begin{tabular}{||r|c|c|c|c|c|c||}
\hhline{|t:=======:t|}
Case& $\delta$ & $\varepsilon$ & \multicolumn{3}{c|}{conditions on $v$} &
faithful solution $(\xalpha,\ybeta)$ \\
\hhline{|:=======:|}
(4) a &$-$&$-$& \multicolumn{2}{c|}{$v=\bb^{2n}\prod B^{n_i}_{v_i}$}
       & $n=0$, $p_{Q'}(V')\ne0$ & $\emptyset$\hyperlink{mpfn23}{$^{\rm(iii)}$}
\\
\hhline{||~|~|~|~|~|-|-||}
b &   &   & \multicolumn{2}{c|}{}
 & $n\in\ZZ\setminus\{0\}$, $V'$ does not  & $\emptyset$\hyperlink{mpfn23}{$^{\rm(iii)}$}
\\
  &   &   & \multicolumn{2}{c|}{}
  & \!satisfy~(iv) of~\fullref{rem:indep}(A)\!   &
\\
\hhline{||~|~|~|-|-|-|-||}
c &   &   & \multicolumn{2}{c|}{$v=u^2$} & $w_-(u)=-1$ &
         $([u^2 B^{-1}, u^{-1}],u^{-1})$,
\\
  &   &   & \multicolumn{2}{c|}{}        &             &
         $([u,B^{-1}], B^{-1}uB)$\hyperlink{mpfn21}{$^{\rm(i)}$}
\\
\hhline{||~|~|~|-|-|-|-||}
d &   &   & \multicolumn{2}{c|}{$v=(\aa\bb)^{2n}$} & $n\in\ZZ$
                                         & $([(\aa\bb)^{2n},\bb],\bb)$\hyperlink{mpfn21}{$^{\rm(i)}$}
\\
\hhline{||~|~|~|-|-|-|-||}
e &   &   & \multicolumn{2}{c|}{$v=(\aa\bb\aa\bb^{-1})^m$} & $w_-(u)=-1$ &
$(1,\ u)$\hyperlink{mpfn21}{$^{\rm(i)}$}
\\
\hhline{|b:=======:b|}
\end{tabular}
\end{center}
\end{footnotesize}
\smallskip
\caption{Mixed cases for faithful solutions of
$\xalpha\ybeta\xalpha\ybeta^{-1}=vB^{-1} v^{-1}B$ where
$B=\aa\bb\aa\bb^{-1}$, $p_K(v)=\bar \bb^{2n}$}
\label{tbl3}
\end{minipage}
\end{table}

\begin{table}[ht!]
\begin{minipage}{\linewidth}
\renewcommand{\thempfootnote}{\rm(\roman{mpfootnote})}
\begin{footnotesize}
\begin{tabular}{||r|c|c|c|c|c|c||}
\hhline{|t:=======:t|}
Case & $\delta$ & $\varepsilon$ & \multicolumn{3}{c|}{conditions on $v$}
& non-faithful solution $(\xalpha,\ybeta)$
\\
\hhline{|:=======:|}
(2) a & +    &  $-$   & \multicolumn{2}{c|}{$v=\bb^{2n}\prod B^{n_i}_{v_i}$}
                                 & $n=0$, $p_Q(V)\ne0$ & $\emptyset$\hyperlink{mpfn23}{$^{\rm(iii)}$}
\\
\hhline{||~|~|~|~|~|-|-||}
b &   &   & \multicolumn{2}{c|}{}
 & $n\in\ZZ{\setminus}\{0\}$, $V$ does not  & $\emptyset$\hyperlink{mpfn23}{$^{\rm(iii)}$}
\\
                &   &   & \multicolumn{2}{c|}{}
  & satisfy~(iv) of~\ref{rem:indep}(A)   &
\\
\hhline{||~|~|~|-|-|-|-||}
c &   &   & \multicolumn{2}{c|}{$v=u^{2k}$} & $w_-(u)=-1$,
                                                            $k\in\ZZ$ &
                   $(u^{2k}(uB)^{-2k}, B^{-1}u^{-1})$\hyperlink{mpfn21}{$^{\rm(i)}$}
\\
\hhline{||~|~|~|-|-|-|-||}
d &     &     & \multicolumn{2}{c|}{$v=B\bb^{2n}$} & $n\in\ZZ$ &
                $((\aa\bb\aa)^{2n}\bb^{-2n},\bb^{2n}(\aa\bb\aa)^{1-2n})$\hyperlink{mpfn21}{$^{\rm(i)}$}
\\
\hhline{||~|~|~|-|-|-|-||}
e &     &     & \multicolumn{2}{c|}{$v=\bb^2B_\aa$} &   &
   $(B_{\bb^2\aa} B^{-1}_{\bb^2}
B^{-1}_{\bb^2\aa}B^{-1}_{\bb^2\aa^2\bb^{-1}},\quad$ \\
     &     &     & \multicolumn{2}{c|}{}           &   &
    $\quad B^{-2}B^{-1}_\aa
\aa^2\bb^{-1})$\hyperlink{mpfn21}{$^{\rm(i)}$}\hyperlink{mpfn22}{$^{\rm(ii)}$}
\\
\hhline{||~|~|~|-|-|-|-||}
f &     &     &  \multicolumn{2}{c|}{$v=\bb^2B_{\aa^k}$} & $k\in\ZZ$,
$k\ne 0,1$  & $\emptyset$\hyperlink{mpfn22}{$^{\rm(ii)}$}\hyperlink{mpfn23}{$^{\rm(iii)}$}
\\
\hhline{||~|~|~|-|-|-|-||}
g &     &     & \multicolumn{2}{c|}{$v=B_{\aa}B_{\aa^{-1}}$} & & $\emptyset$\hyperlink{mpfn22}{$^{\rm(ii)}$}
\\
\hhline{|:=======:|}
(3) a & $-$  &  +  & \multicolumn{2}{c|}{$v=\prod B^{n_i}_{v_i}$}
                       & $p_{Q'}(V')\ne 0$ & $\emptyset$\hyperlink{mpfn23}{$^{\rm(iii)}$}
\\
\hhline{||~|~|~|-|-|-|-||}
   b &     &     & \multicolumn{2}{c|}{$v=c^{2d}\prod B^{n_i}_{v_i}$}
      & $\exists g\in\pi{\setminus}\hat\L_1$, $\hat\varepsilon_g(\hat V'_g)\ne 0$
      & $\emptyset$\hyperlink{mpfn23}{$^{\rm(iii)}$}
\\
\hhline{||~|~|~|-|-|-|-||}
   c   &      &     & \multicolumn{2}{c|}{$v=u^2$} & &
                        $([u^2B^{-1}, u^{-1}], u^{-1})$,
\\
        &      &     & \multicolumn{2}{c|}{} & &
                        $([u,B^{-1}], B^{-1}uB)$\hyperlink{mpfn21}{$^{\rm(i)}$}
\\
\hhline{|:=======:|}
(4) a & $-$    &  $-$   & \multicolumn{2}{c|}{$v=\prod B^{n_i}_{v_i}$}
                         & $p_{Q'}(V')\ne0$ & $\emptyset$\hyperlink{mpfn23}{$^{\rm(iii)}$}
\\
\hhline{||~|~|~|-|-|-|-||}
   b &     &     & \multicolumn{2}{c|}{$v=c^{2d}\prod B^{n_i}_{v_i}$}
      & $\exists g\in\pi{\setminus}\hat\L_1$, $\hat\varepsilon_g(\hat V'_g)\ne 0$
      & $\emptyset$\hyperlink{mpfn23}{$^{\rm(iii)}$}
\\
\hhline{||~|~|~|-|-|-|-||}
   c &     &     & \multicolumn{2}{c|}{$v=B^m$} & $m\in\ZZ$, $w_-(u)=1$ &
                           $(1, u)$\hyperlink{mpfn21}{$^{\rm(i)}$}
\\
\hhline{||~|~|~|-|-|-|-||}
d &        &        & \multicolumn{2}{c|}{$v=u^2$} &
                                           $w_-(u)=1$ &
                $([u^2B^{-1}, u^{-1}],u^{-1})$,
\\
       &        &        & \multicolumn{2}{c|}{} &   &
                $([u,B^{-1}], B^{-1}uB)$\hyperlink{mpfn21}{$^{\rm(i)}$}
\\
\hhline{|b:=======:b|}
\end{tabular}
\end{footnotesize}
\smallskip

\caption{Mixed  cases for non-faithful solutions of
$\xalpha\ybeta\xalpha^{-\delta}\ybeta^{-1}=vB^{-1} v^{-1}B$
where $B=\aa\bb\aa^{-\varepsilon}\bb^{-1}$ and 
(due to \fullref{tbl2}) $p_K(v)=\bar\bb^{2n}$ in Case (2),
 $p_T(v)=\bar\aa^{2m}\bar\bb^{2n}$ in Case (3),
 $p_K(v)=\bar\aa^{2m}\bar\bb^{4n}$ in Case (4);
 if $|m|+|n|>0$ in Case (3) or (4), one denotes $d:=\gcd(m,n)$ and $c:=\aa^{m/d}\bb^{n/d}$ or $c:=\aa^{m/d}\bb^{2n/d}$ (respectively).}
\label{tbl4}

\footnotetext[1]{\hypertarget{mpfn21}{Direct} calculation.}
\footnotetext[2]{\hypertarget{mpfn22}{Using} the Wicks forms (see Wicks~\cite{W},
Vdovina~\cite{V1,V3}, Culler~\cite{Cu} and~\fullref{ex:Wicks}).}
\footnotetext[3]{\hypertarget{mpfn23}{There} is no solution of the second
derived equation, see Theorems~\ref{thm:second}, \ref{thm:partit34nf},
\ref{thm:partit24f}, and \fullref{rem:indep}.}
\end{minipage}
\end{table}

Denote $\Z=\ZZ$ if $\delta=1$, $\Z=\ZZ_2$ if $\delta=-1$. Denote by
$\bar u\in\pi=F_2/N$ the image of $u\in F_2$ under the projection
$F_2\to\pi$, by $V\in \ZZ[\pi]$ the polynomial $V=\sum n_i\bar v_i
\in \ZZ[\pi]$, and by $V'\in\Z[\pi]$ either $V':=V$ if $\delta=1$ or
$V':=V\mod 2$ if $\delta=-1$. Consider the actions on $\pi$ of the
groups $\hhat G\subset\ttilde G_L\subset\hat G_L$, $L\in\ZZ$, in the
first two ``mixed'' cases, see~\eqref{eq:hatG'} and \eqref{eq:ttildeG}, and the action of the group $\hat G$ in the remaining
two ``mixed'' cases, see~\eqref{eq:hatG}. Consider the
corresponding orbits $\hhat\L_g$, $\ttilde\L_{g,L}$, $\hat\L_{g,L}$,
and $\hat\L_g$, $g\in\pi$, see Propositions~\ref{pro:orbits34nf}
and~\ref{pro:orbits24f}. Consider the corresponding augmentations (or
the twisted augmentations in the case of the equation~\eqref{eqn2}
\begin{align*}
\hhat\varepsilon_g\co  \Z[\hhat\L_g] &\to \Z,  &
  \ttilde\varepsilon_{g,L}\co  \Z[\ttilde\L_{g,L}] &\to\Z, \\
\hat\varepsilon_{g,L} \co  \Z[\hat\L_{g,L}] &\to \Z \mbox{ or }\ZZ_2, &
  \hat\varepsilon_g \co  \ZZ_2[\hat\L_g] &\to \ZZ_2,
\end{align*}
see~\eqref{eq:varepsilong}, \eqref{eq:hhatvarepsilong}, \eqref{eq:defec},
\eqref{eq:defec:hat} and~\eqref{eq:hatvarepsilong}. Here $g\in\pi$ as
in~(iv$_{\mbox{\footnotesize e}}$), (iv$_{\mbox{\footnotesize o}}$) of
\fullref{rem:indep}(A) in the first two ``mixed'' cases, while
$g\in\pi\setminus\hat\L_1$ in the other two ``mixed'' cases.
Consider the quotients $Q$ and $Q'$ as in~\eqref{eq:Q} and \eqref{eq:Q'},
and the projections $p_Q\co \ZZ[\pi]\to Q$ and $p_{Q'}\co \ZZ_2[\pi]\to Q'$.

Many of the non-existence results in Tables~\ref{tbl3} and~\ref{tbl4} follow from the non-existence
of a solution of the corresponding second derived equation, see
Theorems~\ref{thm:second}, \ref{thm:partit34nf} and~\ref{thm:partit24f},
and \fullref{rem:indep}.

\bibliographystyle{gtart}
\bibliography{link}

\begin{thebibliography}{}
\providecommand\bibmarginpar{\leavevmode\marginpar}
\def\urlstyle#1{{\tt #1}}

\bibitem{BGKZ}
\textbf{S Bogatyi}, \textbf{D\,L Gon\c{c}alves}, \textbf{E Kudryavtseva},
  \textbf{H Zieschang}, \emph{On the {W}ecken property for the root problem of
  mappings between surfaces}, Mosc. Math. J. 3 (2003) 1223--1245
  \xox{MR}{2058797}

\bibitem{BGZ0}
\textbf{S Bogatyi}, \textbf{D\,L Gon\c{c}alves}, \textbf{H Zieschang},
  \emph{The minimal number of roots of surface mappings and quadratic equations
  in free groups}, Math. Z. 236 (2001) 419--452 \xox{MR}{1821299}

\bibitem{BGZ}
\textbf{S\,A Bogaty\u{\i}}, \textbf{D\,L Gon\c{c}alves}, \textbf{H Zieschang},
  \emph{Coincidence theory: the minimization problem}, Tr. Mat. Inst. Steklova
  225 (1999) 52--86 \xox{MR}{1725933}

\bibitem{B}
\textbf{R Brooks}, \href{http://dx.doi.org/10.2307/2373695} {\emph{Certain
  subgroups of the fundamental group and the number of roots of $f(x)=a$}},
  Amer. J. Math. 95 (1973) 720--728 \xox{MR}{0346777}

\bibitem{Br}
\textbf{K\,S Brown}, \emph{Cohomology of groups}, Graduate Texts in Mathematics
  87, Springer, New York (1994) \xox{MR}{1324339}\ Corrected reprint of the
  1982 original

\bibitem{BSc}
\textbf{R\,F Brown}, \textbf{H Schirmer}, \emph{Nielsen root theory and {H}opf
  degree theory}, Pacific J. Math. 198 (2001) 49--80 \xox{MR}{1831972}

\bibitem{CL}
\textbf{D\,E Cohen}, \textbf{R\,C Lyndon},
  \href{http://dx.doi.org/10.2307/1993597} {\emph{Free bases for normal
  subgroups of free groups}}, Trans. Amer. Math. Soc. 108 (1963) 526--537
  \xox{MR}{0170930}

\bibitem{Cu}
\textbf{M Culler}, \href{http://dx.doi.org/10.1016/0040-9383(81)90033-1}
  {\emph{Using surfaces to solve equations in free groups}}, Topology 20 (1981)
  133--145 \xox{MR}{605653}

\bibitem{Ep}
\textbf{D\,B\,A Epstein}, \emph{The degree of a map}, Proc. London Math. Soc.
  $(3)$ 16 (1966) 369--383 \xox{MR}{0192475}

\bibitem{G2}
\textbf{D\,L Gon\c{c}alves}, \emph{Coincidence of maps between surfaces}, J.
  Korean Math. Soc. 36 (1999) 243--256 \xox{MR}{1688777}

\bibitem{GKZ2}
\textbf{D\,L Gon\c{c}alves}, \textbf{E Kudryavtseva}, \textbf{H Zieschang},
  \emph{Intersection index of curves on surfaces and applications to quadratic
  equations in free groups}, Atti Sem. Mat. Fis. Univ. Modena 49 (2001)
  339--400 \xox{MR}{1881102}

\bibitem{GKZ1}
\textbf{D\,L Gon\c{c}alves}, \textbf{E Kudryavtseva}, \textbf{H Zieschang},
  \href{http://dx.doi.org/10.1007/s002290100238} {\emph{Roots of mappings on
  nonorientable surfaces and equations in free groups}}, Manuscripta Math. 107
  (2002) 311--341 \xox{MR}{1906200}

\bibitem{GZ}
\textbf{D\,L Gon\c{c}alves}, \textbf{H Zieschang},
  \href{http://dx.doi.org/10.1007/PL00004856} {\emph{Equations in free groups
  and coincidence of mappings on surfaces}}, Math. Z. 237 (2001) 1--29
  \xox{MR}{1836771}

\bibitem{GK}
\textbf{R\,I Grigorchuk}, \textbf{P\,F Kurchanov}, \emph{On quadratic equations
  in free groups}, from: ``Proceedings of the International Conference on
  Algebra, Part 1 (Novosibirsk, 1989)'', Contemp. Math. 131, Amer. Math. Soc.
  (1992)  159--171 \xox{MR}{1175769}

\bibitem{GriKurZie}
\textbf{R\,I Grigorchuk}, \textbf{P\,F Kurchanov}, \textbf{H Zieschang},
  \emph{Equivalence of homomorphisms of surface groups to free groups and some
  properties of 3--dimensional handlebodies}, from: ``Proceedings of the
  International Conference on Algebra, Part 1 (Novosibirsk, 1989)'', Contemp.
  Math. 131, Amer. Math. Soc. (1992)  521--530 \xox{MR}{1175803}

\bibitem{Hil}
\textbf{P Hilton}, \emph{Nilpotent actions on nilpotent groups}, from:
  ``Algebraic and logic (Fourteenth Summer Res. Inst., Austral. Math. Soc.,
  Monash Univ., Clayton, 1974)'', Lecture Notes in Mathematics 450, Springer,
  Berlin (1975)  174--196 \xox{MR}{0382447}

\bibitem{HMR}
\textbf{P Hilton}, \textbf{G Mislin}, \textbf{J Roitberg}, \emph{Localization
  of nilpotent groups and spaces}, North-Holland Mathematics Studies 15,
  North-Holland Publishing Co., Amsterdam (1975) \xox{MR}{0478146}

\bibitem{Hm1a}
\textbf{J\,I Hmelevski{\u\i}}, \emph{Systems of equations in a free group I},
  Izv. Akad. Nauk SSSR Ser. Mat. 35 (1971) 1237--1268 \xox{MR}{0313395}

\bibitem{Hm1b}
\textbf{J\,I Hmelevski{\u\i}}, \emph{Systems of equations in a free group II},
  Izv. Akad. Nauk SSSR Ser. Mat. 36 (1972) 110--179 \xox{MR}{0313395}

\bibitem{Hm2}
\textbf{J\,I Hmelevski{\u\i}}, \emph{Equations in free semigroups}, Amer. Math.
  Soc. (1976) \xox{MR}{0393284}\ Translated by G\,A Kandall from the Russian
  original: Trudy Mat. Inst. Steklov. 107 (1971)

\bibitem{HS}
\textbf{G Hochschild}, \textbf{J-P Serre},
  \href{http://dx.doi.org/10.2307/1990851} {\emph{Cohomology of group
  extensions}}, Trans. Amer. Math. Soc. 74 (1953) 110--134 \xox{MR}{0052438}

\bibitem{H}
\textbf{H Hopf}, \href{http://dx.doi.org/10.1007/BF01782365} {\emph{Zur
  {T}opologie der {A}bbildungen von {M}annigfaltigkeiten}}, Math. Ann. 102
  (1930) 562--623 \xox{MR}{1512596} \xox{JFM}{55.0965.02}

\bibitem{K1928}
\textbf{H Kneser}, \href{http://dx.doi.org/10.1007/BF01448865}
  {\emph{Gl\"attung von {F}l\"achenabbildungen}}, Math. Ann. 100 (1928)
  609--617 \xox{MR}{1512504}

\bibitem{K}
\textbf{H Kneser}, \href{http://dx.doi.org/10.1007/BF01455699} {\emph{Die
  kleinste {B}edeckungszahl innerhalb einer {K}lasse von
  {F}l\"achenabbildungen}}, Math. Ann. 103 (1930) 347--358 \xox{MR}{1512626}
  \xox{JFM}{56.1130.02}

\bibitem{KLM}
\textbf{P\,H Kropholler}, \textbf{P\,A Linnell}, \textbf{J\,A Moody},
  \href{http://dx.doi.org/10.2307/2046771} {\emph{Applications of a new
  {$K$}-theoretic theorem to soluble group rings}}, Proc. Amer. Math. Soc. 104
  (1988) 675--684 \xox{MR}{964842}

\bibitem{KWZ}
\textbf{E Kudryavtseva}, \textbf{R Weidmann}, \textbf{H Zieschang},
  \emph{Quadratic equations in free groups and topological applications}, from:
  ``Recent advances in group theory and low-dimensional topology (Pusan,
  2000)'', Res. Exp. Math. 27, Heldermann, Lemgo (2003)  83--122
  \xox{MR}{2004634}

\bibitem{L1}
\textbf{R\,C Lyndon}, \href{http://dx.doi.org/10.2307/1969440}
  {\emph{Cohomology theory of groups with a single defining relation}}, Ann. of
  Math. $(2)$ 52 (1950) 650--665 \xox{MR}{0047046}

\bibitem{L2}
\textbf{R\,C Lyndon},
  \href{http://projecteuclid.org/getRecord?id=euclid.mmj/1028998143} {\emph{The
  equation {$a\sp{2}b\sp{2}=c\sp{2}$} in free groups}}, Michigan Math. J 6
  (1959) 89--95 \xox{MR}{0103218}

\bibitem{L3}
\textbf{R\,C Lyndon}, \href{http://dx.doi.org/10.2307/1993533} {\emph{Equations
  in free groups}}, Trans. Amer. Math. Soc. 96 (1960) 445--457
  \xox{MR}{0151503}

\bibitem{LS}
\textbf{R\,C Lyndon}, \textbf{P\,E Schupp}, \emph{Combinatorial group theory},
  Ergebnisse der Mathematik und ihrer Grenzgebiete 89, Springer, Berlin (1977)
  \xox{MR}{0577064}

\bibitem{MKS}
\textbf{W Magnus}, \textbf{A Karrass}, \textbf{D Solitar}, \emph{Combinatorial
  group theory}, Dover Publications Inc., Mineola, NY (2004) \xox{MR}{2109550}\
  \ Reprint of the 1976 second edition

\bibitem{M}
\textbf{G\,S Makanin}, \emph{Equations in a free group}, Izv. Akad. Nauk SSSR
  Ser. Mat. 46 (1982) 1199--1273, 1344 \xox{MR}{682490}

\bibitem{Ni}
\textbf{J Nielsen}, \href{http://dx.doi.org/10.1007/BF02421324}
  {\emph{Untersuchungen zur {T}opologie der geschlossenen zweiseitigen
  {F}l\"achen}}, Acta Math. 50 (1927) 189--358 \xox{MR}{1555256}
  \xox{JFM}{53.0545.12}

\bibitem{Ol'shanski}
\textbf{A\,Y Ol'shanski\u{\i}}, \href{http://dx.doi.org/10.1007/BF00970919}
  {\emph{Diagrams of homomorphisms of surface groups}}, Sibirsk. Mat. Zh. 30
  (1989) 150--171 \xox{MR}{1043443}

\bibitem{O}
\textbf{P Olum}, \href{http://dx.doi.org/10.2307/1969748} {\emph{Mappings of
  manifolds and the notion of degree}}, Ann. of Math. $(2)$ 58 (1953) 458--480
  \xox{MR}{0058212}

\bibitem{OZ}
\textbf{R\,P Osborne}, \textbf{H Zieschang},
  \href{http://dx.doi.org/10.1007/BF01389191} {\emph{Primitives in the free
  group on two generators}}, Invent. Math. 63 (1981) 17--24 \xox{MR}{608526}

\bibitem{R}
\textbf{A\,A Razborov}, \emph{Systems of equations in a free group}, Izv. Akad.
  Nauk SSSR Ser. Mat. 48 (1984) 779--832 \xox{MR}{755958}

\bibitem{Schreier}
\textbf{O Schreier}, \emph{Die Untergruppen der freien Gruppen}, Abh. Math.
  Sem. Univ. Hamburg 5 (1927) 161--183

\bibitem{Sk}
\textbf{R Skora}, \href{http://dx.doi.org/10.1007/BF01450838} {\emph{The degree
  of a map between surfaces}}, Math. Ann. 276 (1987) 415--423 \xox{MR}{875337}

\bibitem{S}
\textbf{J Stallings}, \href{http://dx.doi.org/10.1016/0021-8693(65)90017-7}
  {\emph{Homology and central series of groups}}, J. Algebra 2 (1965) 170--181
  \xox{MR}{0175956}

\bibitem{St}
\textbf{A Steinberg},
  \href{http://projecteuclid.org/getRecord?id=euclid.mmj/1029000593} {\emph{On
  equations in free groups}}, Michigan Math. J. 18 (1971) 87--95
  \xox{MR}{0289614}

\bibitem{V1}
\textbf{A\,A Vdovina}, \href{http://dx.doi.org/10.1080/00927879508825398}
  {\emph{Constructing of orientable {W}icks forms and estimation of their
  number}}, Comm. Algebra 23 (1995) 3205--3222 \xox{MR}{1335298}

\bibitem{V2}
\textbf{A Vdovina}, \emph{On the number of nonorientable {W}icks forms in a
  free group}, Proc. Roy. Soc. Edinburgh Sect. A 126 (1996) 113--116
  \xox{MR}{1378835}

\bibitem{V3}
\textbf{A Vdovina}, \href{http://dx.doi.org/10.1142/S0218196797000216}
  {\emph{Products of commutators in free products}}, Internat. J. Algebra
  Comput. 7 (1997) 471--485 \xox{MR}{1459623}

\bibitem{W2}
\textbf{M\,J Wicks}, \emph{The equation $X^2Y^2=g$ over free products}, from:
  ``Proc. 2nd Congress Singapore Nat. Acad. Sci. 1971, Sci. Urban Environment
  Tropics'' (1973)  238--248

\bibitem{W}
\textbf{N\,J Wicks}, \emph{Commutators in free products}, J. London Math. Soc.
  37 (1962) 433--444 \xox{MR}{0142610}

\bibitem{Z1964}
\textbf{H Zieschang}, \emph{Alternierende {P}rodukte in freien {G}ruppen}, Abh.
  Math. Sem. Univ. Hamburg 27 (1964) 13--31 \xox{MR}{0161901}

\bibitem{Z1965}
\textbf{H Zieschang}, \emph{Alternierende {P}rodukte in freien {G}ruppen II},
  Abh. Math. Sem. Univ. Hamburg 28 (1965) 219--233

\bibitem{Z1966}
\textbf{H Zieschang}, \emph{Discrete groups of plane motions and plane group
  images}, Uspehi Mat. Nauk 21 (1966) 195--212 \xox{MR}{0195954}

\bibitem{ZVC}
\textbf{H Zieschang}, \textbf{E Vogt}, \textbf{H-D Coldewey}, \emph{Surfaces
  and planar discontinuous groups}, Lecture Notes in Mathematics 835, Springer,
  Berlin (1980) \xox{MR}{606743}\ Translated from the German by John Stillwell

\end{thebibliography}

\end{document}